\documentclass[12pt,reqno,twoside,a4paper]{amsart}

\makeatletter
\@namedef{subjclassname@2020}{\textup{2020} Mathematics Subject Classification}
\makeatother

\usepackage{amsmath, amsthm, amssymb, amscd, amsfonts, tikz, float, enumitem, graphicx, mathrsfs, wasysym}
\usepackage[toc]{appendix}

\usepackage{color}
\usepackage[all]{xy}             
 \usepackage{hyperref}          

\font \smallrm=cmr10 at 10truept
\font \smallsl=cmsl10 at 10truept
\font \ssmallrm=cmr10 at 9truept
\font \smallbf=cmbx10 at 10pt

\numberwithin{equation}{section}

\newcommand{\subu}[2]{{#1}_{\raise-2pt\hbox{$ \scriptstyle #2 $}}}
\newcommand{\subd}[3]{{#1}_{\raise-2pt\hbox{$ \scriptstyle #2 #3 $}}}

\newtheorem{lema}{Lemma}[subsection]
\newtheorem{theorem}[lema]{Theorem}

\newtheorem{prop}[lema]{Proposition}

\theoremstyle{definition}
\newtheorem{definition}[lema]{Definition}

\newtheorem{rmk}[lema]{Remark}
\newtheorem{rmks}[lema]{Remarks}
\newtheorem{free text}[lema]{}

\newtheorem{obs's}[lema]{Observations}

\theoremstyle{remark}

\newcommand \End{\operatorname{End}}

\newcommand \Ker{\operatorname{Ker}}

\newcommand \smallast {{\operatorname{\raise1,5pt\hbox{$ \scriptscriptstyle \ast $}}}}
\newcommand \smallp {{\operatorname{\raise1pt\hbox{$ \scriptscriptstyle + $}}}}
\newcommand \smallm {{\operatorname{\raise1pt\hbox{$ \scriptscriptstyle - $}}}}
\newcommand \smallpm {{\operatorname{\raise1pt\hbox{$ \scriptscriptstyle \pm $}}}}
\newcommand \smallmp {{\operatorname{\raise1pt\hbox{$ \scriptscriptstyle \mp $}}}}

\newcommand{\zero}{{\bar{0}}}
\newcommand{\one}{{\bar{1}}}

\newcommand{\F}{{\mathcal F}}

\newcommand{\D}{{\mathcal D}}

\def \bq {\mathbf{q}}

\def \NN{\mathbb{N}}

\def \ZZ {\mathbb{Z}}
\def \k {\Bbbk}

\def \kh{{\Bbbk[[\hbar]]}}

\def \kqqm{{\Bbbk\big[\,q\,,q^{-1}\big]}}
\def \kq{{\Bbbk(q)}}
\def \kbqqm{{\Bbbk\big[q^{\pm 1},\bq^{\pm 1}\big]}}
\def \kbq{{\Bbbk(q\,,\bq\,)}}

\def \D {\mathcal{D}}

\def \J {\mathcal{J}}

\def \Dpicc {{\scriptscriptstyle D}}
\def \Ppicc {{\scriptscriptstyle P}}

\def \lieg{\mathfrak{g}}

\def \liegd{\mathfrak{g}^{\raise-1pt\hbox{$ \Dpicc $}}}

\def \liegdP{\mathfrak{g}^{\raise-1pt\hbox{$ \Dpicc $}}_\Ppicc}
\def \liegRpP{\mathfrak{g}^{\raise1pt\hbox{$ \scriptscriptstyle \mathcal{R},p $}}_\Ppicc}
\def \liegRpPa{\mathfrak{g}^{\raise1pt\hbox{$ \scriptscriptstyle \mathcal{R},p $}}_{\Ppicc_{\raise-3pt\hbox{$ \scriptscriptstyle \!\! A $}}}}
\def \liegPsi{\mathfrak{g}^{\raise-1pt\hbox{$ \scriptscriptstyle \Psi $}}}
\def \liegdPsi{\mathfrak{g}^{\raise-1pt\hbox{$ {\scriptscriptstyle D,\Psi} $}}}

\def \liegdotdq{\dot{\lieg}_{\raise-2pt\hbox{$ \Dpicc , \hskip-0,9pt{\scriptstyle \bq} $}}}
\def \liegtildq{\tilde{\lieg}_{\raise-2pt\hbox{$ \Dpicc , \hskip-0,9pt{\scriptstyle \bq} $}}}
\def \liegdq{\lieg_{\raise-2pt\hbox{$ \Dpicc , \hskip-0,9pt{\scriptstyle \bq} $}}}
\def \liegdqcheck{\lieg_{\raise-2pt\hbox{$ \Dpicc , \hskip-0,9pt{\scriptstyle \check{\bq}} $}}}
\def \Gtildeqstar {\widetilde{G}^{\,\raise2pt\hbox{$ \scriptstyle * $}}_{\!\Dpicc , \hskip+0,9pt{\scriptstyle \bq}}}
\def \liegl{\mathfrak{gl}}

\def \liebd{\mathfrak{b}^{\raise-1pt\hbox{$ \Dpicc $}}}

\def \liehd{\mathfrak{h}^{\raise-1pt\hbox{$ \Dpicc $}}}

\def \lieso{\mathfrak{so}}

\def \Phipicc {{\scriptscriptstyle \Phi}}
\def \Mpicc {{\scriptscriptstyle M}}
\def \Uhglnp {U_\hbar\big(\liegl_{\,n}^{\,p}\big)}
\def \Fhglnp {F_\hbar\big[\big[GL_{\,n}^{\,p}\big]\big]}
\def \Uqglnp {U_q\big(\liegl_{\,n}^{\,p}\big)}
\def \Uqintglnp {U_q^{\textsl{int}\!}\big(\liegl_{\,n}^{\,p}\big)}
\def \Fqmlnp {F_q\big[M\!L_{\,n}^{\,p}\big]}
\def \Fqintmlnp {F_q^{\textsl{int}}\big[M\!L_{\,n}^{\,p}\big]}
\def \Fqbullintmlnp {F_{q,\,\bullet}^{\textsl{int}}\big[M\!L_{\,n}^{\,p}\big]}
\def \Fbqmlnp {F_\bq\big[M\!L_{\,n}^{\,p}\big]}
\def \Fbqintmlnp {F_\bq^{\,\textsl{int}}\big[M\!L_{\,n}^{\,p}\big]}
\def \MLnp {\textit{ML}_{\,n}^{\,p}}
\def \UhPhiglnp {U_\hbar^\Phipicc\!\big(\liegl_{\,n}^{\,p}\big)}
\def \FhPhiglnp {F_\hbar^\Phipicc\big[\big[GL_{\,n}^{\,p}\big]\big]}
\def \fhg{{F_\hbar[[G\hskip1,5pt]]}}

\def \uhg{U_\hbar(\hskip0,5pt\lieg)}

\newcommand{\sigmachi}{{\raise-1pt\hbox{$ \, \scriptstyle \dot\sigma_\chi $}}}

\def \pf{\begin{proof}}
\def \epf{\end{proof}}

\theoremstyle{plain}

\begin{document}

\title[MULTIPARAMETER QUANTUM GENERAL LINEAR SUPERGROUP]
 {Multiparameter Quantum  \\
  General Linear Supergroup}

\author[Fabio GAVARINI \ , \ \  Margherita PAOLINI]
{Fabio Gavarini$\,{}^\sharp\,$,  \ \ Margherita PAOLINI$\,{}^\flat\,$}

\address{
   \newline
 ${}^\sharp$  Dipartimento di Matematica,
   \newline
 Universit\`a degli Studi di Roma ``Tor Vergata''   ---   INdAM\,/\,GNSAGA
   \newline
 Via della ricerca scientifica 1,  \; I\,-00133 Roma, ITALY   ---   {\tt gavarini@mat.uniroma2.it}%
 \vspace*{0.5cm}
   \newline
  ${}^\flat$  Independent Researcher   ---   {\tt margherita.paolini.mp@gmail.com} }

\thanks{\noindent 2020 \emph{MSC:}\,  20G42, 16T05, 81R50   ---
\emph{Keywords:} Quantum Groups, Quantum Function Algebras.}


\begin{abstract}
 \medskip
   We introduce uniparametric and multiparametric quantisations of the general linear supergroup, in the form of ``quantised function algebras'',
 both in a  \textsl{formal\/}  setting   --- yielding ``quantum formal series Hopf superalgebras'', \`a la Drinfeld ---   and in a  \textsl{polynomial\/}  one   --- closer to Manin's point of view.
                                                                       \par
   In the uniparametric setting, we start from quantised universal enveloping superalgebras over  $ \liegl_n $  (endowed with a super-structure), as in  \cite{Ya}  and  \cite{Zh}:  through a direct approach, we construct their linear dual, thus finding the quantum formal series Hopf superalgebras mentioned above, which are described in detail via an explicit presentation.  Starting from the latter, then, we perform a deformation by a well-chosen 2--cocycle, thus getting a multiparametric quantisation, described again by an explicit presentation: this is, in turn, the dual to the multiparametric quantised universal enveloping algebra over  $ \liegl_n $  from  \cite{GGP}.
                                                                       \par
   We also provide some ``polynomial versions'' of these quantisations, both for the uniparametric and the multiparametric case.  In particular, we compare the latter to Manin's quantum function algebras from  \cite{Ma}.
                                                                       \par
   Finally, both for the uniparametric and the multiparametric setting, we provide suitable PBW-like theorems, in ``formal'' and in ``polynomial'' versions alike.
\end{abstract}

{\ } \vskip-51pt

   \centerline{ \smallrm  {\smallsl Journal of Pure and Applied Algebra\/}
{\smallbf 230}  (2026), no.\ 10, 108344, 35 pages }
   \centerline{\smallrm {\smallbf DOI:}  10.1016/j.jpaa.2026.108344}
 {\ }
\vskip17pt

\maketitle

\tableofcontents



\section{Introduction}  \label{sec: intro}

\vskip7pt

   The so-called ``quantum groups'' appear in literature in two forms: either quantised universal enveloping algebras (in short, QUEA's) over some Lie algebra  $ \lieg \, $,  or quantised function algebras over some Lie or algebraic group  $ G $.  In both cases, ``quantised'' is meant in two possible ways, namely:
                                                                        \par
   --- \  a  \textit{formal\/}  one, where we consider  \textsl{topological\/}  Hopf algebras, say  $ U_\hbar(\lieg) $  or  $ F_\hbar[[G]] \, $,  over  the ring  $ \kh $  of formal power series in the deformation parameter  $ \hbar \, $,
                                                                        \par
   --- \  a  \textit{polynomial\/}  one, where we consider  \textsl{standard\/}  Hopf algebras,  $ U_q(\lieg) $  or  $ F_q[G] \, $,  over a base ring where an element  $ q $  takes the role of ``deformation parameter'', e.g.\  $ \kq $  or  $ \kqqm $  (the ring of rational functions or Laurent polynomials in  $ q \, $).
 \vskip3pt
   In either case, the quantisation canonically defines, as ``semiclassical limit'', an additional Poisson structure on the underlying geometrical object, namely a Lie cobracket on  $ \lieg $   --- turning the latter into a Lie bialgebra ---   and a Poisson bracket on  $ G $   --- making it into a Poisson (Lie or algebraic) group; see  \cite{Dr}  or  \cite{CP}  for details.
                                                                        \par
   One can also consider  \textit{multiparametric\/}  quantisations, involving several parameters: nevertheless, only one of them ``rules'' the quantisation, whereas the others are responsible for the induced Poisson structure at the semiclassical limit.  A typical procedure to produce such quantisations involve Hopf-theoretical deformations, either via twists or via 2--cocycles: one starts with a uniparametric quantisation, and then applying a deformation (by twist or by 2--cocycle) ends up with a multiparametr quantisation, whose ``extra parameters'' come from the twist or the 2--cocycle involved in the process (cf.\ e.g.,  \cite{GG1},  \cite{GG2}  and references therein for more details).
                                                                        \par
   All the above applies as well, up to technicalities, to the context of ``quantum supergroups'', i.e.\ quantisations of Lie superalgebras and supergroups.
 \vskip5pt
   In the present paper, we deal with the general linear Lie superalgebra and supergroup, that is  $ \liegl_n $  and  $ \textit{GL}_n $  endowed with some ``parity''.  In this setup,  \textsl{uniparametric\/}  quantised universal enveloping superalgebras (in short, QUESA's) have been introduced in  \cite{Ya}  (in great generality) in a  \textit{formal\/}  version, and then taken up again in  \cite{Zh}  in  \textit{polynomial\/}  form.  Starting from that,  \textsl{multiparametric\/}  QUESA's have been constructed in  \cite{GGP}  via the process of deformation by twist explained above.
                                                                        \par
   Our goal is to work out the dual side, i.e.\ to introduce suitable dual objects to the QUESA's for  $ \liegl_{\,n}^{\,p} $  mentioned above, in all their variants   --- uniparametric or multiparametric, formal or polynomial.
                                                                        \par
   In the uniparametric setting, we start from Yamane's QUESA  $ \Uhglnp \, $,  where the superscript  ``$ \, p \, $''  accounts for the underlying parity.  Through a direct approach, we construct its full linear dual, which is concretely realised as a topological Hopf superalgebra  $ \Fhglnp $  with a non-degenerate Hopf pairing with  $ \Uhglnp \, $.  In particular, we find that this  $ \Fhglnp $  is indeed a quantum formal series Hopf superalgebra (in short, QFSHA) as we were looking for: we provide for it an explicit presentation by generators and relations and a suitable PBW-like theorem.
                                                                        \par
   For the multiparametric side of the story, we rely on the uniparametric one and resort to a deformation procedure.  Indeed, as every multiparametric QUESA  $ \UhPhiglnp $  from  \cite{GGP}  is obtained from Yamane's QUESA  $ \Uhglnp $  via deformation by some twist  $ \F_\Phi \, $,  one can get the dual  $ \, {\big( \UhPhiglnp \big)}^* \, $
 as deformation of  $ {\big( \Uhglnp \big)}^* $  by the 2--cocycle  $ \sigma_\Phipicc $  corresponding to  $ \F_\Phi \, $.  But    $ \, {\big( \UhPhiglnp \big)}^* = \Fhglnp \, $,  the uniparametric QFSHA that we just constructed; so we only have to compute the deformation  $ \, {\big( \Fhglnp \big)}_{\sigma_\Phipicc} \, $.  The final outcome is a multiparametric QFSHA  $ \FhPhiglnp \, $  for which we find a presentation by generators and relations and a PBW-like theorem.
 \vskip3pt
   After achieving our goal for  \textit{formal\/}  quantisations, we obtain the parallel result for  \textsl{polynomial\/}  ones in a very simple way   --- roughly, selecting suitable subalgebras inside  $ \Fhglnp $  (for the uniparametric case) and  $ \FhPhiglnp $  (for the multiparametric one).  In particular, for the latter we discuss a bit its direct comparison with Manin's multiparametric QFSA's introduced in  \cite{Ma}.

\vskip11pt

   \centerline{\ssmallrm ACKNOWLEDGEMENTS}
 \vskip3pt
   {\smallrm The authors thankfully acknowledge support by:
   \;\textit{(1)}\,  the MIUR  {\smallsl Excellence Department Project
   MatMod@TOV (CUP E83C23000330006)\/}  awarded to the Department of Mathematics,
   University of Rome ``Tor Vergata'' (Italy),
   \;\textit{(2)}\,  the research branch  {\smallsl GNSAGA\/}  of INdAM (Italy),
   \;\textit{(3)}\,  the European research network  {\smallsl COST Action CaLISTA CA21109}.}

\bigskip
 \vskip13pt

\section{Multiparametric QUESA for the general linear supergroup}  \label{sec: mp-QUESA}
 \vskip7pt
   In this section, we first recall Yamane's QUESA's, bounding ourselves to those of type  $ A $  (in general, they are defined for each possible type of simple Lie superalgebra of basic type), as introduced in  \cite{Ya}.  Then we recall also their ``multiparametric version'', introduced in  \cite{GGP}.

\vskip11pt

\subsection{Yamane's QUESA's for the general linear supergroup}  \label{subsec: Yamane's QUESA}
 In his ground-breaking paper  \cite{Ya},  Yamane introduced explicit quantisations of (the universal enveloping superalgebra of) any simple Lie superalgebra of basic type, say  $ \lieg \, $: as such, they are quantised universal enveloping superalgebras (=QUESA's) in the sense of Drinfeld.  Their definition is given via a presentation by generators and relations: in turn, the latter depends on the choice of a specific Dynkin diagram (or, equivalently, of a specific Cartan matrix) associated with  $ \lieg \, $;  in particular, different diagrams (or matrices) give rise to different quantisations, and it is not clear whether they are isomorphic (as topological Hopf superalgebras) or not, though their semiclassical limit always is.  In this paper, we consider such QUESA's only for type  $ A \, $:  we will call them ``Yamane's QUESA's''.

\medskip

\begin{definition}  \label{def: Yamane-Uhg}
 Fix a field  $ \k $  of characteristic zero.  Given  $ \, n \in \NN \, $,  $ \, n \geq 2 \, $,  let  $ \, I_n := \{1,\dots,n\} \, $  and  $ \, I_{n-1} := I \setminus \{n\} \, $.  Let  $ \mathbf{V}_n $  be a  $ \k $--vector  space of dimension  $ n \, $,  and  $ \, {\{\varepsilon_i\}}_{i \in I_n} \, $  be a  $ \k $--basis  of  $ \mathbf{V}_{\!n} \, $.  We fix a  \textit{parity function\/}  $ \, p \,:\, I_n \relbar\joinrel\longrightarrow \ZZ_2 \, $,  \,and the (unique)  $ \k $--bilinear  form  $ \, (\cdot\,, \cdot) \, $  on  $ \mathbf{V}_{\!n} $  such that  $ \, (\varepsilon_i \, , \varepsilon_j \big) = {(-1)}^{p(i)} \delta_{i,j} \, $.  We define the  \textit{root system}  $ \, \Phi = \Phi_+ \cup \Phi_- \subseteq \mathbf{V}_{\!n} \, $  by setting
 $ \, \Phi_+ := \big\{\, \alpha_{i,\,j} := \varepsilon_i - \varepsilon_j \,\big|\, i\,, j \in I_n \, , \, i < j \,\big\} \, $
 and
 $ \, \Phi_- := -\Phi_+ = \big\{ -\!\alpha_{i,\,j} = \varepsilon_j - \varepsilon_i \,\big|\, i\,, j \in I_n \, , \, i < j \,\big\} \, $,
 \,and we let the set of (positive)  \textit{simple roots\/}  to be
 $ \, \Pi := \big\{\, \alpha_i := \alpha_{i,\,i+1} \,\big|\, i \in I_{n-1} \big\} \, $.
 \vskip5pt
   For later use, we fix the important notation  $ \; p_{i,j} := p(i)+p(j) \; $  for all  $ \, i \, , j \in I_n \, $.
 \vskip5pt
   We define the  \textit{Dynkin diagram\/}  (associated with the previous data) as follows.  First, we consider a standard Dynkin diagram of type $ A_{n-1} \, $,  namely a (non-oriented) graph with  $ (n-1) $  vertices, labelled by  $ 1 $,  $ \dots $,  $ n-1 $,  with just one single edge connecting the vertex  $ \, i \, $  with the vertex  $ \, i+1 \, $  (and viceversa),  for all  $ \, 1 \leq i \leq n-1 \, $.  Second, we represent the  $ i $--th  vertex by drawing either a ``white'' node  $ \fullmoon $  or a ``grey'' node  $ \otimes \, $,  according to whether  $ \, (\alpha_i\,,\alpha_i) \not= 0 \, $  or  $ \, (\alpha_i\,,\alpha_i) = 0 \, $,  \,respectively (or, equivalently, whether  $ \, p(i) \not= p(i+1) \, $  or  $ \, p(i) = p(i+1) \, $,  \,or still whether  $ \, p_{i,\,i+1} = \one \, $  or  $ \, p_{i,\,i+1} = \zero \, $).
                                                                                \par
   Overall, we will refer to such a bunch  $ \D $  of data as to a  \textit{Dynkin datum\/}:  note that it is uniquely determined by the pair  $ \, \big(\textit{rank\/}(\D)\,,\textit{parity\/}(\D)\big) := (n\,,p) \, $.   \hfill  $ \diamond $
\end{definition}

\medskip

\begin{rmk}  \label{rmk: Dynk-dat===>pres-gl(n)}
 Any Dynkin diagram as in  Definition \ref{def: Yamane-Uhg}  is nothing but a ``Dynkin diagram'' in the (standard) sense of Kac   --- see  \cite{Ka}  ---   for the Lie superalgebra structure over  $ \liegl(n) $  corresponding to the parity  $ p $   --- hereafter denoted by  $ \, \liegl_{\,n}^{\,p} \, $.  Among all these Dynkin diagrams, those with just one single grey node (every other one being white), called   \textsl{distinguished},  classify all isomorphism classes of simple Lie superalgebras onto  $ \liegl(n) $  that are not just totally even (i.e., they are not Lie algebra structures).  In particular, the same Lie superalgebra structure can be associated to different Dynkin diagrams: in this case, every such diagram gives rise to a different presentation of the Lie algebra, though the latter stands the same.  Correspondingly, one also has different presentations of the universal enveloping superalgebra  $ U\big(\liegl_{\,n}^{\,p}\big) \, $.
\end{rmk}

\medskip

   Following Yamane (see  \cite{Ya})   --- but with a choice of notation much closer to Zhang's  (cf.\ \cite{Zh})  ---   we introduce our quantisation of  $ U\big(\liegl_{\,n}^{\,p}\big) $  as follows:

\medskip

\begin{definition}  \label{def: form-Uhg}
 Fix a Dynkin datum  $ \D $  as above, with rank  $ n $  and parity  $ p \, $.  Fix a field  $ \k $  of characteristic zero.  In the ring of formal power series  $ \kh \, $,  set
  $$  q \, := \, \exp(\hbar) \; ,  \quad \qquad  q_i \, := \, q^{{(-1)}^{p(i)}} = \, q^{(\varepsilon_i,\,\varepsilon_i)}  \qquad  \forall \;\; i \in I_n  $$
   \indent   We define  $ \Uhglnp $  as follows: it is the unital, associative,  $ \hbar $--adically  complete  $ \k[[\hbar]] $--superalgebra  with generators
  $$  E_1 \, , \, E_2 \, , \, \dots \, , \, E_{n-1} \, , \varGamma_1 \, , \, \varGamma_2 \, , \, \dots \, , \, \varGamma_{n-1} \, , \, \varGamma_n \, ,  \, F_1 \, , \, F_2 \, , \, \dots \, , \, F_{n-1}  $$
 having parity  $ \; |E_r| := p_{r,r+1} \; $,  $ \; |\varGamma_s| := 0 \; $,  $ \; |F_r| := p_{r+1,r} \; $   --- for all  $ \, r \in I_{n-1} \, $  and  $ \, s \in I_n \, $  ---   and relations (for all  $ \, i \, , j \in I_{n-1} := \{1, \dots , n-1\} \, $,  $ \, k \, , \ell \in I_n := \{1, \dots , n\} \, $)
\begin{equation}  \label{eq: UhgR1}
  \begin{gathered}
    [\varGamma_k \, , \varGamma_\ell\,] \, = \, 0
  \end{gathered}
\end{equation}
\begin{equation}  \label{eq:  UhgR2}
  \begin{gathered}
    [\varGamma_k \,, F_j] \, - \, (\delta_{k,j+1} - \delta_{k,j}) \, F_j \, = \, 0  \quad ,   \qquad  [\varGamma_k \,, E_j] \, - \, (\delta_{k,j} - \delta_{k,j+1}) \, E_j \, = \, 0  \\
  \end{gathered}
\end{equation}
\begin{equation}  \label{eq:UqgR3}
  \begin{gathered}
    [E_i \, , F_j] \, - \, \delta_{i,j} \, \frac{\; e^{+\hbar H_i} - e^{-\hbar H_i} \;}{\; q_i^{+1} - q_i^{-1} \;} \, = \, 0
  \end{gathered}
\end{equation}
\begin{equation}  \label{eq:UqgR4}
  \begin{gathered}
    \quad   E_i^2 \, = \, 0 \; ,  \quad  F_i^2 \, = \, 0   \quad \qquad  \text{if \ }  p_{i,\,i+1} = \one = p_{i+1,\,i}
  \end{gathered}
\end{equation}
\begin{equation}  \label{eq:UqgR5}
  \begin{gathered}
    [E_i \, , E_j] \, = \, 0  \quad ,   \qquad   [F_i \, , F_j] \, = \, 0  \qquad  \text{if \ }  |i-j| > 1
  \end{gathered}
\end{equation}
\begin{equation}  \label{eq:UqgR6}
  \begin{gathered}
    \hskip-5pt  E_i^2 \, E_j \, - \, \big( q + q^{-1} \big) \, E_i \, E_j \, E_i \, + \, E_j \, E_i^2 \, = \, 0   \qquad  \text{if \ }  p_{i,\,i+1} = \zero  \text{\;\ \ and \;\ }  |i-j| = 1  \\
    \hskip-0pt  F_i^2 \, F_j - \big( q + q^{-1} \big) \, F_i \, F_j \, F_i \, + \, F_j \, F_i^2 \, = \, 0   \;\qquad  \text{if \ }  p_{i+1,\,i} = \zero  \text{\;\ \ and \;\ }  |i-j| = 1
  \end{gathered}
\end{equation}
\begin{equation}  \label{eq:UqgR7}
  \begin{gathered}
    \big[ \big[ {[E_i,E_j]}_{q_j} , E_k \big]_{q_{j+1}} , E_j \big] \, = \, 0   \qquad  \text{\; if \ }  p_{j,\,j+1} = \one  \text{\;\ \ and \;\ }  k = j+1 = i+2  \\
    \big[ \big[ {[F_i,F_j]}_{q_j} , F_k \big]_{q_{j+1}} , E_j \big] \, = \, 0   \qquad  \text{\; if \ }  p_{j+1,\,j} = \one  \text{\;\ \ and \;\ }  k = j+1 = i+2
  \end{gathered}
\end{equation}
where hereafter we use such notation as
  $$  \displaylines{
   H_i  \, := \,  {(-1)}^{p(i)} \varGamma_i - {(-1)}^{p(i+1)} \varGamma_{i+1}  \cr
   \hfill   {[A\,,B]}_c  \, := \,  A B - c \, {(-1)}^{|A|\,|B|} B A  \quad ,   \qquad   [A\,,B]  \, := \,  {[A\,,B]}_1   \hfill   \diamond  }  $$
\end{definition}

\vskip9pt

   The most relevant fact about the algebra  $ \Uhglnp $  defined above is that it admits a Hopf structure, as specified in the following result:

\vskip9pt

\begin{theorem}  \label{thm: Hopf-struct x Uhglnp}
 \textsl{(cf.\  \cite{Ya})}  There exists a structure of topological Hopf superalgebra on  $ \Uhglnp $   --- with respect to the  $ \hbar $--adic  topology ---   which is uniquely described by the following formulas on generators (for all  $ \, i \in I_{n-1} \, , \; k \in I_n \, $):
  $$  \displaylines{
   \Delta(E_i)  \; = \;  E_i \otimes 1 + e^{+\hbar H_i} \otimes E_i \; ,  \qquad  S(E_i)  \; = \;  -e^{-\hbar H_i} E_i \; ,  \qquad  \epsilon(E_i) \; = \; 0  \cr
   \quad  \Delta(\varGamma_k)  \; = \;  \varGamma_k \otimes 1 + 1 \otimes \varGamma_k \; ,  \hskip-5pt \qquad \qquad S(\varGamma_k) \; = \; -\varGamma_k \; ,  \hskip5pt \quad \qquad  \epsilon(\varGamma_k)  \; = \;  0  \cr
   \Delta(F_i)  \; = \;  F_i\otimes e^{-\hbar H_i} + 1 \otimes F_i \; ,  \qquad  S(F_i) \; = \; - F_i \, e^{+\hbar H_i} \; ,  \qquad  \epsilon(F_i) \; = \; 0  }  $$
\end{theorem}

\vskip7pt

\begin{rmk}  \label{rmk: Uhglnp-is-QUESA}
 Because of its properties, including what is detailed in  Theorem \ref{thm: Hopf-struct x Uhglnp},  the Hopf superalgebra  $ \Uhglnp $  introduced in  Definition \ref{def: form-Uhg}  is a ``quantized universal enveloping superalgebra'' in the sense of Drinfeld   --- cf.\  \cite{Dr}, \S 7,  and  \cite{Ga}, Definition 1.2, suitably adapted to the present setup of quantum  \textsl{super\/}groups).  This motivates the following definition.
\end{rmk}

\vskip7pt

\begin{definition}  \label{def: QUESA Uhglnp}
 We call the Hopf superalgebra  $ \Uhglnp $  introduced in  Definition \ref{def: form-Uhg}  above the  \textit{quantized universal enveloping superalgebra (or just ``QUESA'', in short) associated with the Lie superbialgebra  $ \liegl_{\,n}^{\,p} \, $}.   \hfill  $ \diamond $
\end{definition}

\vskip11pt

\subsection{From Yamane's QUESA's to multiparametric QUESA's}  \label{subsec: Yam-Mp_QUESA's}
 In  \cite{GGP},  a multiparametric counterpart of Yamane's QUESA's was introduced.  In short, it was found through a process of  \textsl{deformation by twist},  after carefully choosing the twist in a broad family, called of  \textit{``toral'' twists\/}  in that, roughly, they belong to the ``toral part'' of Yamane's QUESA that we start from.  Hereafter we shortly summarise that construction, when is is applied to Yamane's QUESA of type  $ A \, $,  namely the  $ \Uhglnp $  given in  Definition \ref{def: form-Uhg}  above, with its Hopf structure from  Theorem \ref{thm: Hopf-struct x Uhglnp}; for further details, we refer to  \cite{GGP}.
                                                                              \par
   To begin with, we choose any antisymmetric matrix of size  $ (n+1) $  with entries in  $ \, \kh \, $,  say  $ \, \Phi = {\big( \phi_{t,\ell} \big)}_{t=1,\dots,n;}^{\ell=1,\dots,n;} \in \mathfrak{so}_n\big(\kh\big) \, $.  Out of that, we define the element
\begin{equation}   \label{eq: Resh-twist_F_Phi - 1}
  \F_\Phi  \,\; := \;\,  \exp\Big(\hskip1pt \hbar \; 2^{-1} {\textstyle \sum_{t,\ell=1}^n} \, \phi_{t,\ell} \, \varGamma_t \otimes \varGamma_\ell \Big)  \;\; \in \;\;  \Uhglnp \,\widehat{\otimes}\, \Uhglnp
\end{equation}
 which is actually a  {\sl twist\/}  element for  $ \Uhglnp \, $,  in the standard sense of Hopf algebra theory (cf.\  \cite{GGP}, \S 4.1.5).  Using it, we construct a new (topological) Hopf algebra  $ \, {\Uhglnp}^{\F_\Phi} $,  \,known as the  \textsl{deformation of  $ \, \Uhglnp $  by the twist $ \F_\Phi \, $}:  \,this is isomorphic to  $ \Uhglnp $  as a superalgebra but with a new, twisted (super)coalgebra structure, whose coproduct, counit and antipode are described in the next statement:

\vskip9pt

\begin{theorem}  \label{thm: Hopf-struct x Mp-UhPhiglnp}
 \textsl{(cf.\  \cite{GGP},  \S 4.2.7)}
   Fix any antisymmetric matrix of size  $ n+1 $  with entries in  $ \, \kh \, $,  say  $ \, \Phi = {\big( \phi_{t,\ell} \big)}_{t=1,\dots,n;}^{\ell=1,\dots,n;} \in \mathfrak{so}_n\big(\kh\big) \, $,  \,and define in  $ \Uhglnp $
   the elements
%
 (for all  $ \, i \in I_{n-1} \, $)
  $$  T_{i,+}^{\,\Phipicc} \, := \,
 2^{-1} \, {\textstyle \sum_{\ell=1}^n} \big( \phi_{i,\ell} \, \varGamma_\ell \, - \, \phi_{i+1,\ell} \, \varGamma_\ell \big) \;\; ,
   \qquad  T_{i,-}^{\,\Phipicc} \, := \,
 2^{-1} \, {\textstyle \sum_{\ell=1}^n} \big( \phi_{\ell,i} \, \varGamma_\ell \, - \, \phi_{\ell,i+1} \, \varGamma_\ell \big)  $$
%
%
 Then there exists a structure of topological Hopf superalgebra on  $ \Uhglnp $   --- with respect to the  $ \hbar $--adic  topology ---   which is uniquely described by the following formulas on generators (for all  $ \, i \in I_{n-1} \, , \; k \in I_n \, $):
  $$  \displaylines{
   \Delta\big(E_i\big)  \; = \;
 E_i \otimes e^{+\hbar \, T_{i,+}^{\,\Phipicc}} \, + \, e^{+\hbar \, (H_i + T_{i,-}^{\,\Phipicc})} \otimes
E_i  \quad ,
   \qquad  \epsilon\big(E_i\big) \, = \, 0  \cr
   \Delta\big(\varGamma_k\big)  \; = \;  \varGamma_k \otimes 1 \, + \, 1 \otimes \varGamma_k  \quad ,   \quad \qquad  \epsilon\big(\varGamma_k\big) \, = \, 0   \quad  \cr
   \Delta\big(F_i\big)  \; = \;
   F_i \otimes e^{-\hbar \, (H_i + T_{i,+}^{\,\Phipicc})} + \, e^{-\hbar \, T_{i,-}^{\,\Phipicc}} \otimes F_i  \quad ,
   \qquad  \epsilon\big(F_i\big) \, = \, 0  \cr
   S\big(E_i\big)  =  - e^{-\hbar \, (H_i +  T_{i,-}^{\,\Phipicc})} E_i \, e^{-\hbar \, T_{i,+}^{\,\Phipicc}}  \; ,
   \;\;  S\big(\varGamma_k\big)  =  -\varGamma_k  \; ,
   \;\;  S\big(F_i\big)  =  - e^{+\hbar \, T_{i,-}^{\,\Phipicc}} \, F_\ell \, e^{+\hbar \, (H_i + T_{i,+}^{\,\Phipicc})}  }  $$
   \indent   In particular, such a Hopf algebra is isomorphic with  $ \, {\Uhglnp}^{\F_\Phi} $,  \,the deformation of  $ \, \Uhglnp $  by the twist  $ \F_\Phi $  in  \eqref{eq: Resh-twist_F_Phi - 1}.     \qed
\end{theorem}

\vskip9pt

\begin{definition}   \label{def: Multipar UhPhiglnp}
 We call  $ \Uhglnp $  with the Hopf structure described in  Theorem \ref{thm: Hopf-struct x Mp-UhPhiglnp}  above  \textit{the multiparametric QUESA}   --- over $ \liegl_{\,n}^{\,p} $  ---   with multiparameter  $ \Phi \, $,  \,and we denote it by  $ \UhPhiglnp \, $.   \hfill  $ \diamond $
\end{definition}

\vskip9pt

\begin{rmk}
 By its very construction, the multiparametric QUESA  $ \UhPhiglnp $  always admits a presentation (as a topological  $ \kh $--superalgebra)  as in  Definition \ref{def: form-Uhg}   --- which is independent of  $ \Phi $  --- while its Hopf structure is given as in  Theorem \ref{thm: Hopf-struct x Mp-UhPhiglnp}   --- which explicitly depends on our matrix of ``parameters''  $ \Phi \, $.  Conversely, it is proved in  \cite{GGP}  (cf.\ in particular  Proposition 4.2.8 and Theorem 4.2.9 therein) that one can also describe the same object with a different presentation.  Namely, in the latter presentation the parameters show up in the relations among generators   --- which look like those in  Definition \ref{def: form-Uhg},  but explicitly involve some extra parameters, that in turn depend on  $ \Phi $  ---   while the formulas for the coproduct, counit and antipode are independent of  $ \Phi $  (and of any other ``parameters'' whatsoever), but constantly look like those in  Theorem \ref{thm: Hopf-struct x Uhglnp}  for Yamane's initial QUESA.  Nevertheless, we are now focusing on the first type of presentation, thus choosing the point of view where the parameters govern the co-algebraic structure.
\end{rmk}

\bigskip
 \vskip13pt

\section{Formal quantum general linear supergroups}  \label{sec: mp-QFSHSA's}
 \vskip7pt
   In this section we introduce the new contribution of the present paper, namely a ``quantum formal supergroup'' of type  $ A $,  both in uniparametric and in multiparametric version, realised as dual to Yamane's QUESA  $ \Uhglnp $   --- from  \S \ref{subsec: Yamane's QUESA}  --- and to its multiparametric counterpart  $ \UhPhiglnp $   --- from  \S \ref{subsec: Yam-Mp_QUESA's}.

\vskip11pt

\subsection{QFSHSA's for Yamane's quantum general linear supergroup}  \label{subsec: QFSHSA's x Yamane}
 Let us consider Yamane's QUESA  $ \Uhglnp $  as in  \S \ref{subsec: Yamane's QUESA},  and its linear dual, denoted by  $ \Fhglnp \, $.  In Drinfeld's language   --- as in  \cite{Dr}, \S 7,  only slightly adapted to fit the (quantum) supergroup framework ---   the latter is a ``quantum formal series Hopf superalgebra.  We now undertake a detailed description of this object, eventually yielding an explicit presentation by generators and relations.

\medskip

\begin{free text}
 \textbf{The QFSHSA over  $ \textit{GL}_n^{\,p} \, $:  first construction.}
 For our QUESA  $ \Uhglnp $  there exists a well-known ``standard'' (or ``natural'', or ``vector'') representation  $ V_n $  which is straightforward quantisation of the standard representation of the classical superalgebra  $ U\big(\liegl_{\,n}^{\,p}\big) \, $.  Explicitly,  $ V_n $  is the free  $ \kh $--module  of rank  $ n $  with  $ \kh $--basis  $ \, \{ v_1 \, , \dots , v_n \} \, $  with  $ \Uhglnp $--action  defined on generators by
  $$  E_i\,.\,v_\ell \, := \, \delta_{i+1,\ell} \, v_{\ell-1} \; ,  \!\!\quad  \varGamma_k\,.\,v_\ell \, := \, \delta_{i,\ell} \, v_\ell \; ,  \!\!\quad  F_i\,.\,v_\ell \, := \, \delta_{i,\ell} \, v_{\ell+1}  \;\quad  \forall \;\; i \in I_{n-1} \, , \; k, \ell \in I_n  $$
 Indeed, one checks by straightforward computations that this yields a correctly defined representation of  $ \Uhglnp $  onto  $ V_n \, $,  which then is a  $ \Uhglnp $--module.
                                                                                       \par
   For each pair  $ \, (i\,,j\,) \in I_n \, $,  the corresponding  \textit{matrix coefficient}  $ \, x_{i{}j} \, $  is defined: it is the  $ \kh $--linear  function  $ \, x_{i{}j} : \Uhglnp \relbar\joinrel\relbar\joinrel\relbar\joinrel\longrightarrow \kh \, $  defined by  $ \, x_{i{}j}(u) := v_i^*(u\,.\,v_j) \, , \; \forall \, u \in \Uhglnp \, $,  \,where  $ \, \{ v^*_1 \, , \dots , v^*_n \} \, $  is the dual  $ \kh $--basis  of  $ V_n $  to the built-in  $ \kh $--basis  $ \, \{ v_1 \, , \dots , v_n \} \, $.
 In other words,  $ x_{i,k} $  is the composition of the representation morphism  $ \, \rho\raise-3pt\hbox{$ {}_{V_n} $} : \Uhglnp \relbar\joinrel\relbar\joinrel\longrightarrow \End_\kh\!\big( V_n \big) \cong \operatorname{Mat}_{\,n \times n}\!\big(\kh\big) \, $  with the function  $ \, \operatorname{Mat}_{\,n \times n}\!\big(\kh\big) \relbar\joinrel\relbar\joinrel\relbar\joinrel\longrightarrow \kh \, , \; {\big( m_{\,r,s} \big)}_{r \in I_n}^{s \in I_n} \mapsto m_{\,i,j} \, $  --- i.e., evaluation with respect to the  $ (i\,,j\,) $--th  entry.
                                                                                       \par
   Let  $ \, {\Uhglnp}^* \, $  be the linear dual to the Hopf superalgebra  $ \Uhglnp \, $.  This dual is in turn a topological Hopf superalgebra, the topology involved being now the $ * $--weak  one: in particular, its coproduct is dual to the product in  $ \Uhglnp \, $,  and it takes values into a suitable completion of the algebraic tensor product  $ \, {\Uhglnp}^* \otimes {\Uhglnp}^* \, $.  However, for all the  $ x_{i{}j} $'s   --- that do belong to  $ {\Uhglnp}^* $  indeed ---   it follows by construction that their coproduct obeys the remarkable formula
\begin{equation}  \label{eq: Delta(x_r,s)}
  \qquad   \Delta(x_{r{}s})  \,\; = \;\,  {\textstyle \sum_{k=1}^n}\, {(-1)}^{p_{i{}k}\,p_{k{}j}} x_{r{}k} \otimes x_{k{}s}   \qquad \qquad  \forall \;\; r, s \in I_n
\end{equation}
 --- where  $ \, p_{r{}s} := p(r) + p(s) \, $  ---   in short, because of the simple reason that the product in  $ \operatorname{Mat}_{n+1}\!\big(\kh\big) $  is nothing but row-by-column multiplication; in particular, each  $ \, \Delta(x_{i{}j}) \, $  belongs to the standard, algebraic tensor product  $ \, {\Uhglnp}^* \otimes {\Uhglnp}^* \, $.  Let us then consider the unital  $ \kh $--subalgebra  $ \, \dot{F}_\hbar = \dot{F}_\hbar\big[\textit{GL}_{\,n}^{\,p}\big] \, $  of  $ {\Uhglnp}^* $  generated by  $ \, \big\{ x_{r{}s} \,\big|\, r, s \in I_n \big\} \, $.  Thanks to  \eqref{eq: Delta(x_r,s)},  $ \dot{F}_\hbar\big[\textit{GL}_n^p\big] $  is in fact a sub-super\textsl{bi\/}\-algebra of  $ {\Uhglnp}^* $,  and even in ``standard'' (i.e., non-topological) sense; in particular, its counit map  $ \, \epsilon := \epsilon_{{U_\hbar(\liegl_{\,n}^{\,p})}^*}{\Big|}_{\dot{F}_\hbar} \, $  is described by  $ \, \epsilon(x_{i{}j}) = \delta_{i{}j} \, $  for all  $ \, i \, , j \in I_n \, $.
                                                                                       \par
   Let now  $ \, I_{\dot{F}_\hbar} := \Ker(\epsilon) + \hbar\,\dot{F}_\hbar \, $.  This is a two-sided ideal of  $ \dot{F}_\hbar \, $,  hence we can consider the associated  $ I_{\dot{F}_\hbar} $--adic  topology in  $ \dot{F}_\hbar \, $.  Indeed, such a topology coincides with the one induced by the  $ * $--weak  topology in  $ {\Uhglnp}^* \, $:  the the  $ I_{\dot{F}_\hbar} $--adic  completion of  $ \dot{F}_\hbar $  coincides with the closure of  $ \dot{F}_\hbar $  itself in  $ {\Uhglnp}^* $  with respect to the  $ * $--weak  topology: we denote this closure by  $ \, F_\hbar = \Fhglnp \, $.  Note that  $ F_\hbar $  is clearly a sub-super\textsl{bi\/}algebra of  $ {U\big(\liegl_{\,n}^{\,p}\big)}^* \, $;  \,moreover, by construction the semiclassical limit of $ F_h $  is the  $ \dot{I} $--adic  completion  $ F $  of the subalgebra  $ \dot{F} $  of  $ {\Uhglnp}^* $  generated by the ``classical'' matrix coefficient functions  $ x_{i,j} $  (for all  $ \, i, j \in I_n \, $),  where  $ \, \dot{I} := \Ker\big( \epsilon_{\!{}_F} \big) \, $.  But classically we have  $ \, F = {U\big(\liegl_{\,n}^{\,p}\big)}^* \, $,  whence we conclude that  $ \, F_\hbar := \Fhglnp = {\Uhglnp}^* \, $  as well.  Thus, in the end, the (topological) Hopf superalgebra  $ {\Uhglnp}^* $  can be realised as topologically generated by the  $ x_{i,j} $'s;  we will presently improve this remark by providing a concrete, explicit presentation.  We begin by fixing the terminology:
\end{free text}

\vskip7pt

\begin{definition}  \label{def: QFSHSA-h}
 We call  \textit{quantum formal series Hopf superalgebra\/}  (=QFSHSA in short) over the supergroup  $ \textit{GL}_{\,n}^{\,p} $  the $ I_{\dot{F}_\hbar} $--adic  completion of  $ \dot{F}_\hbar $   --- or, equivalently (see above), the closure of  $ \dot{F}_\hbar $  in  $ {\Uhglnp}^* $  with respect to the  $ * $--weak  topology.
                                                                            \par
   We denote this QFSHSA by  $ \, \Fhglnp \, $.   \hfill  $ \diamond $
\end{definition}

\vskip7pt

   The previous construction of  $ \, \Fhglnp \, $  is certainly a bit hasty, with several details to be filled in in order to clarify all steps.  However, we will now present instead an alternative construction, which leads to a concrete realisation of  $ F_\hbar $  that is in fact totally independent of the above constructions.

\medskip

\begin{free text}  \label{free: constr.-FhG}
 \textbf{New construction of  $ \, \Fhglnp \, $.}
 We proceed in several steps, simultaneously building both the topological superbialgebra  $ F_\hbar $  and its pairing with  $ \Uhglnp \, $.
 \vskip7pt
   \textbf{\textit{(1)}}\;  We begin with the free unital  $ \kh $--superalgebra  $ \mathbb{F}_\hbar $  with generators  $ \, x_{i{}j} \, $   --- for all  $ \, i \, , j \in I_n \, $  ---   having parity  $ \, |x_{i{}j}| = p_{i{}j} := p(i) + p(j) \, $   --- notation as in  Definition \ref{def: Yamane-Uhg}.  This  $ \mathbb{F}_\hbar $  admits a unique superbialgebra structure whose coproduct and counit are given on generators by the following formulas (the proof is trivial):
\begin{equation}  \label{eq: coprod+counit for x(i,j)}
  \Delta(x_{ij}) = {\textstyle \sum\limits_{a=1}^n} \, {(-1)}^{p_{i{}a} p_{a{}j}} x_{i{}a} \otimes x_{a{}j}   \quad ,  \qquad
      \epsilon(x_{ij}) = \delta_{ij}
   \qquad \quad  \forall \;\; i \, , j \in I_n   \quad
\end{equation}
   \indent   On the other hand, let  $ \mathbb{U}_\hbar $  be the free unital  $ \kh $--superalgebra  generated by  $ \, E_i \, , \varGamma_t \, , F_i \, $   --- for all  $ \, i \in I_{n-1} \, $  and  $ \, t \in I_n \, $  ---   with parity  $ \, \big|E_i\big| = p_{i,\,i+1} = p_{i+1,\,i} = \big|F_i\big| \, $  and  $ \, \big|\varGamma_t\big| = \zero \, $.  This  $ \mathbb{U}_\hbar $  is free as a  $ \kh $--module,  with basis given by all ordered monomials in the generators.  Let now  $ \, \hat{\mathbb{U}}_\hbar \, $  be the  $ \hbar $--adic  completion of  $ \mathbb{U}_\hbar \, $,  with its natural structure of Complete topological) unital  $ \kh $--superalgebra:  as such, it is freely generated by  $ \, E_i \, , \varGamma_t \, , F_i \, $  (for  $ \, i \in I_{n-1} \, $  and  $ \, t \in I_n \, $).  Then the formulas in  Theorem \ref{thm: Hopf-struct x Uhglnp}  yield well-defined coproduct and counit  $ \hat{\mathbb{U}}_\hbar $  that make the latter into a topological  super\textit{bi\/}algebra  (again, the proof is trivial).
 \vskip7pt
   \textbf{\textit{(2)}}\;  There exists a unique  $ \kh $--bilinear  pairing  $ \; \langle\,\ ,\ \rangle : \mathbb{F}_h \times \hat{\mathbb{U}}_\hbar \relbar\joinrel\longrightarrow \kh \; $  which is defined as follows.  For each generator  $ x_{i{}j} $  of  $ \mathbb{F}_\hbar \, $,  we define the values of  $ \, \big\langle x_{i{}j} \, , \mathcal{M} \big\rangle \, $  for every monomial  $ \mathcal{M} $  in the generators of  $ \hat{\mathbb{U}}_\hbar \, $:  as these monomials form a  $ \kh $--basis  of  $ \hat{\mathbb{U}}_\hbar \, $,  this is enough to uniquely define, by  $ \kh $--linear  extension, all values of  $ \, \big\langle x_{i{}j} \, , \,-\, \big\rangle \, $  on all of  $ \hat{\mathbb{U}}_\hbar \, $.  Moreover, as the  $ x_{i{}j} $'s  are  \textsl{free\/}  generators of  $ \mathbb{F}_h $  as a $ \kh $--algebra,  we can choose those values  $ \, \big\langle x_{i,j} \, , \mathcal{M} \big\rangle \, $  as freely as possible.  We set
\begin{equation}  \label{eq: values for pairing degree(M)=0}
   \big\langle\, x_{i{}j} \, , 1 \big\rangle  \; := \;  \epsilon\big(x_{i{}j}\big)  \; := \;  \delta_{i{}j}  \\
\end{equation}
 which accounts for the unique ``monomial''  $ \, \mathcal{M} = 1 \, $  having degree 0, and then
\begin{equation}  \label{eq: values for pairing degree(M)=1}
   \big\langle x_{i{}j} \, , E_t \big\rangle  \; = \;  \delta_{i,\,t} \, \delta_{j,\,t+1} \;\; ,  \quad
\big\langle x_{i{}j} \, , F_t \big\rangle  \; = \,  \delta_{i,\,t+1} \, \delta_{j,\,t} \;\; ,  \quad  \big\langle x_{i{}j} \, , \varGamma_k \big\rangle  \; = \;  \delta_{i,\,k} \, \delta_{j,\,k}
\end{equation}
 which accounts for all possible monomials  $ \mathcal{M} $  of degree 1.  As to monomials  $ \mathcal{M} $  of degree greater than 1, we define  $ \, \big\langle x_{i{}j} \, , \mathcal{M} \big\rangle \, $  by induction on such a degree via
\begin{equation}  \label{eq: values for pairing degree(M)>1}
   \big\langle x_{i{}j} \, , \mathcal{M} \,\big\rangle  \; = \;  \big\langle x_{i{}j} \, , \mathcal{Y} \, \mathcal{Z} \,\big\rangle  \; := \;  \big\langle \Delta\big(x_{i{}j}\big) \, , \mathcal{Y} \otimes \mathcal{Z} \,\big\rangle  \\
\end{equation}
 for any possible factorisation of  $ \mathcal{M} $  as a product  $ \, \mathcal{M} = \mathcal{Y} \, \mathcal{Z} \, $  of two monomials $ \mathcal{Y} $  and  $ \mathcal{Z} $  of strictly smaller degree, where in the rightmost term we are using the formula
\begin{equation}  \label{eq: pairing with coproduct}
  \big\langle \Delta(X) \, , \mathcal{Y} \otimes \mathcal{Z} \,\big\rangle  \; := \;  {(-1)}^{|X_{(2)}| \, |\mathcal{Y}|} \, \big\langle X_{(1)} \, , \mathcal{Y} \,\big\rangle \, \big\langle X_{(2)} \, , \mathcal{Z} \,\big\rangle
\end{equation}
 where  $ \; \Delta(X) = X_{(1)} \otimes X_{(2)} \; $  is the standard Sweedler's notation for the coproduct of  $ X \, $   --- now applied to  $ \, X := x_{i,j} \, $.  Due to the coassociativity of the coproduct  $ \Delta $  of  $ \mathbb{F}_\hbar \, $,  formula  \eqref{eq: values for pairing degree(M)>1}  is independent of any particular factorisation  $ \, \mathcal{M} = \mathcal{Y} \, \mathcal{Z} \, $,  \,hence in particular it does make sense to properly define  $ \, \big\langle x_{i{}j} \, , \mathcal{M} \,\big\rangle \, $.
                                                                                                   \par
   Now, with  \eqref{eq: values for pairing degree(M)=0},  \eqref{eq: values for pairing degree(M)=1}  and  \eqref{eq: values for pairing degree(M)>1}  we have in fact defined a linear operator  $ \, \rho\big(x_{i{}j}\big) \in {\hat{\mathbb{U}}}_\hbar^{\,*} \, $  for each generator  $ x_{i{}j} $  of  $ \mathbb{F}_h \, $:  as these are  \textsl{free\/}  generators, the above also extends to yield a well-defined morphism  $ \; \rho : \mathbb{F}_h \!\relbar\joinrel\longrightarrow {\hat{\mathbb{U}}}_\hbar^{\,*} \; $  of unital  $ \kh $--algebras.  In detail, this means that one has (for all  $ \, f', f'' \in \mathbb{F}_h \, $,  $ \, u \in \hat{\mathbb{U}}_\hbar \, $)
\begin{equation}  \label{eq: rho-product}
\begin{aligned}
  \big(\rho\big(\, f' f'' \,\big)\big)(u)  \;  &  = \;  \big(\, \rho\big(\,f'\,\big) \star \rho\big(\,f''\,\big) \big)(u)  \; = \;  \big(\, \rho\big(\,f'\,\big) \otimes \rho\big(\,f''\,\big) \big) \big(\Delta(u)\big)  \; =  \\
    &  = \;  {(-1)}^{|f''| \, |u_{(1)}|} \, \big(\, \rho\big(\,f'\,\big) \big)\!\big( u_{(1)} \big) \cdot \big(\, \rho\big(\,f''\,\big) \big)\!\big( u_{(2)} \big)
\end{aligned}
\end{equation}
 just because of the very definition of the product  $ \, \star \, $  (``convolution'') in  $ {\hat{\mathbb{U}}}_\hbar^{\,*} \, $.
                                                                                                   \par
   In turn, the map  $ \; \rho : \mathbb{F}_h \!\relbar\joinrel\longrightarrow {\hat{\mathbb{U}}}_\hbar^{\,*} \; $  uniquely corresponds to a  $ \kh $--bilinear  pairing
\begin{equation}  \label{eq: pairing BbbF-hatBbbU}
   \langle\,\ ,\ \rangle : \mathbb{F}_h \times \hat{\mathbb{U}}_\hbar \relbar\joinrel\relbar\joinrel\relbar\joinrel\longrightarrow \kh  \quad ,  \qquad  \big\langle\, f \, , u \,\big\rangle \, := \, \big(\,\rho(f)\big)(u)
\end{equation}
 which is uniquely characterised (and described) by formulas  \eqref{eq: values for pairing degree(M)=0},  \eqref{eq: values for pairing degree(M)=1}  and  \eqref{eq: values for pairing degree(M)>1}.  Furthermore, those same formulas guarantee that this pairing is also a  \textsl{(super) coalgebra--algebra pairing},  in that the following holds (for all  $ \, f \in \mathbb{F}_h \, $,  $ \, u' , u'' \in \hat{\mathbb{U}}_\hbar \, $):
\begin{equation}  \label{eq: pairing coproduct-product}
  \big\langle\, f \, , \, u' \, u'' \,\big\rangle  \; = \;  \big\langle \Delta(f) \, , u' \otimes u'' \,\big\rangle  \; := \;  {(-1)}^{|f_{(2)}| \, |u'|} \, \big\langle\, f_{(1)} \, , u' \,\big\rangle \, \big\langle\, f_{(2)} \, , u'' \,\big\rangle
\end{equation}
   \indent   On the other hand, symmetrically, formula  \eqref{eq: rho-product}  can be read also
 as
\begin{equation}  \label{eq: pairing product-coproduct}
  \big\langle\, f' f'' , \, u \,\big\rangle  \; = \;  \big\langle\, f' \otimes f'' , \Delta(u) \,\big\rangle  \; := \;  {(-1)}^{|f''| \, |u_{(1)}|} \, \big\langle\, f' , u_{(1)} \,\big\rangle \, \big\langle\, f'' , u_{(2)} \,\big\rangle
\end{equation}
 which means that this pairing is also a  \textsl{(super) algebra--coalgebra pairing}.  To sum up,  \eqref{eq: pairing coproduct-product}  and  \eqref{eq: pairing product-coproduct},  along with the unital/counital properties
\begin{equation}  \label{eq: pairing unit/counit-counit/unit}
  \big\langle\, f \, , \, 1 \,\big\rangle  \; = \;  \epsilon(f)  \quad ,  \qquad
  \big\langle\, 1 \, , \, u \,\big\rangle  \; = \;  \epsilon(u)   \qquad \qquad  \forall \;\; f \in \mathbb{F}_h \, , \; u \in \hat{\mathbb{U}}_\hbar
\end{equation}
 tell us that the pairing in  \eqref{eq: pairing BbbF-hatBbbU}  is actually a pairing  \textsl{of (topological)  $ \kh $--super\textit{bi}al\-gebras},  or  \textsl{a (topological)  $ \kh $--super\textit{bi}algebra  pairing}.
                                                                                                   \par
   Finally, via a canonical recipe, the pairing  \eqref{eq: pairing BbbF-hatBbbU}  also induces a similar pairing
\begin{equation}  \label{eq: pairing tensor-product}
  \langle\,\ ,\ \rangle : \big(\, \mathbb{F}_h \otimes \mathbb{F}_h \big) \times \big(\, \hat{\mathbb{U}}_\hbar \otimes \hat{\mathbb{U}}_\hbar \big) \relbar\joinrel\relbar\joinrel\relbar\joinrel\relbar\joinrel\longrightarrow \kh
\end{equation}
 given by
  $ \; \big\langle\, f' \otimes f'' , u' \otimes u'' \,\big\rangle \, := \, {(-1)}^{|f''| \, |u'|} \big\langle\, f' , u' \,\big\rangle \, \big\langle\, f' , u' \,\big\rangle \; $.
%
%
 \vskip7pt
   \textbf{\textit{(3)}}\;  We now consider the (super)bi-ideal  $ \, \mathbb{I}_\hbar := \Ker\big(\epsilon_{\!{}_{\mathbb{F}_\hbar}}\big) + \hbar\,\mathbb{F}_\hbar \, $,  and the  $ \mathbb{I}_\hbar $--adic  completion of  $ \mathbb{F}_\hbar \, $,  denoted by  $ \, \tilde{\mathbb{F}}_\hbar \, $:  \,by construction, this is a topological  $ \kh $--super\textit{bi\/}algebra,  whose coproduct takes values in  $ \, \tilde{\mathbb{F}}_\hbar \,\widetilde{\otimes}\, \tilde{\mathbb{F}}_\hbar \, $,  the  $ \mathbb{I}_\hbar^\otimes $--adic  completion of the algebraic tensor product  $ \, \tilde{\mathbb{F}}_\hbar \otimes \tilde{\mathbb{F}}_\hbar \, $  with $ \, \mathbb{I}_\hbar^\otimes := \mathbb{F}_\hbar \otimes \mathbb{I}_\hbar + \mathbb{I}_\hbar \otimes \mathbb{F}_\hbar \, $.  One then easily sees that the pairing  \eqref{eq: pairing BbbF-hatBbbU}  extends   --- by continuity  ---   to a similar pairing
\begin{equation}  \label{eq: pairing hatBbbF-hatBbbU}
   \langle\,\ ,\ \rangle : \tilde{\mathbb{F}}_h \times \hat{\mathbb{U}}_\hbar \relbar\joinrel\relbar\joinrel\relbar\joinrel\longrightarrow \kh
\end{equation}
 which is in turn a  $ \kh $--linear  (topological)  \textsl{superbialgebra pairing\/}  on its own.
 \vskip7pt
   \textbf{\textit{(4)}}\;  Let  $ \, \mathcal{R} $  be the  $ \hbar $--adically  closed, two-sided ideal in  $ \hat{\mathbb{U}}_\hbar $  generated by all the elements that appear as left-hand side terms in the identities  \eqref{eq: UhgR1}--\eqref{eq:UqgR7}.  Then the quotient topological  $ \kh $--superalgebra  $ \, \hat{\mathbb{U}}_\hbar \Big/ \mathcal{R} \, $  is isomorphic to  $ \Uhglnp \, $,  just by the very definition of the latter.  We can then identify  $ \; \hat{\mathbb{U}}_\hbar \Big/ \mathcal{R} \, = \, \Uhglnp \, $,  \;and through this we get the following claim:
 \vskip3pt
   \textsl{$ \underline{\text{Claim}} $:}  \textit{The pairing  \eqref{eq: pairing hatBbbF-hatBbbU}  factors (on the right) through  $ \mathcal{R} \, $,  thus inducing a similar pairing of (topological) superbialgebras}
\begin{equation}  \label{eq: pairing tildeBbbF-Uhglnp}
   \langle\,\ ,\ \rangle : \tilde{\mathbb{F}}_h \times \Uhglnp \relbar\joinrel\relbar\joinrel\relbar\joinrel\longrightarrow \kh
\end{equation}
 \vskip3pt
   To see that, since  $ \tilde{\mathbb{F}}_h $  is generated (as a unital, topological algebra) by the  $ x_{i,\,j} $'s,  it is enough to show that:
 \vskip3pt
   \textbf{\textit{(4--a)}}\;  $ \, \big\langle\, 1 \, , \rho \,\big\rangle = 0 \, $  and  $ \, \big\langle\, x_{i,j} \, , \rho \,\big\rangle = 0 \, $  for every  $ \rho $  in the selected set of generators of the ideal  $ \, \mathcal{R} $  and for all  $ \, i \, , j \in I_n \, $;
 \vskip3pt
   \textbf{\textit{(4--b)}}\;  the two-sided ideal  $ \, \mathcal{R} $  is in fact a coideal as well, which boils down to having  $ \; \Delta(\rho\hskip0,5pt) \in \, \mathcal{R} \otimes \hat{\mathbb{U}}_\hbar + \hat{\mathbb{U}}_\hbar \otimes \mathcal{R} \; $  for every  $ \rho $  in our fixed set of generators of  $ \, \mathcal{R} \, $.
 \vskip7pt
   Proving these two points is a matter of sheer computations, though non-trivial ones.  The interested reader find the details in  Appendix \ref{appendix}.
 \vskip9pt
   \textbf{\textit{(5)}}\;  Let  $ \, \J $  be the closed, two-sided ideal in  $ \tilde{\mathbb{F}}_\hbar $  generated (in topological sense) by the following elements
\begin{equation}  \label{eq: gen.'s-ideal-J-in-Fh}
  \begin{gathered}
   \hskip69pt   x_{i{}j}^{\,2}   \hskip169pt \qquad  \text{\ for \ \ }  p_{i{}j} = \one  \cr
   x_{i{}j} \, x_{i{}k} \, - \, {(-1)}^{p_{i{}j} p_{i{}k}} \, e^{+\hbar \, {(-1)}^{p(i)}} \, x_{i{}k} \, x_{i{}j}  \hskip75pt \qquad  \text{for \ \ } j < k
  \cr
   x_{i{}j} \, x_{h{}j} \, - \, {(-1)}^{p_{i{}j} p_{h{}j}} \, e^{+\hbar \, {(-1)}^{p(j)}} \, x_{h{}j} \, x_{i{}j}  \hskip71pt \qquad  \text{for \ \ } i<h
  \cr
   \qquad \quad   x_{i{}j} \, x_{h{}k} \, - \, {(-1)}^{p_{i{}j} p_{h{}k}} x_{h{}k} \, x_{i{}j}  \qquad \hskip51pt \qquad \  \text{for \ \ }  i<h \, , \; j>k   \hskip15pt  \cr
   x_{i{}j} \, x_{h{}k} \, - \, {(-1)}^{p_{i{}j} p_{h{}k}} x_{h{}k} \, x_{i{}j} \, - \, {(-1)}^{p_{i{}j} p_{i{}k}} \Big( e^{+\hbar \, {(-1)}^{p(i)}} \! - e^{-\hbar \, {(-1)}^{p(i)}} \Big) \, x_{i{}k} \, x_{h{}j}   \hskip5pt
  \\
   \hfill   \text{for \ \ }  i<h \, , \; j<k   \hskip15pt
  \end{gathered}
\end{equation}
 and set  $ \, \Fhglnp := \tilde{\mathbb{F}}_\hbar \Big/ \J \, $  for the quotient topological  $ \kh $--superalgebra.
 \vskip3pt
   \textsl{$ \underline{\text{Note}} $:}\,  the last-line elements can be written in a slightly different form, namely
  $$  x_{i{}j} \, x_{h{}k} \, - \, {(-1)}^{p_{i{}j} p_{h{}k}} x_{h{}k} \, x_{i{}j} \, - \,
 {(-1)}^{p_{i{}j} p_{i{}k} + p(i)} \big(\, e^{+\hbar} - e^{-\hbar} \,\big) \, x_{i{}k} \, x_{h{}j}  $$
 (for  $ \, i<h \, $  and  $ \, j<k \, $):  we shall indeed use such an alternative form in the sequel.
 \vskip3pt
   We aim now to prove the following
 \vskip5pt
   \textsl{$ \underline{\text{Claim}} $:}  \textit{The pairing  \eqref{eq: pairing tildeBbbF-Uhglnp}  factors (on the left) through  $ \J \, $,  thus inducing a similar pairing of (topological) superbialgebras}
\begin{equation}  \label{eq: pairing Fhglnp-Uhglnp}
   \langle\,\ ,\ \rangle : \Fhglnp \times \Uhglnp \relbar\joinrel\relbar\joinrel\relbar\joinrel\relbar\joinrel\relbar\joinrel\longrightarrow \kh
\end{equation}
 \vskip3pt
   To see that, it is enough to show that:
 \vskip3pt
   \textbf{\textit{(5--a)}}\;  $ \, \big\langle\, \eta \, , 1 \,\big\rangle = 0 \, $  and  $ \, \big\langle\, \eta \, , \gamma \,\big\rangle = 0 \, $  for every  $ \eta $  in the selected set  \eqref{eq: gen.'s-ideal-J-in-Fh}  of generators of the ideal  $ \, \J $  and for all  $ \, \gamma \in \big\{ E_r \, , \varGamma_s \, , F_r \;\big|\; r \in I_{n-1} \, , \, s \in I_n \big\} \, $;
 \vskip3pt
   \textbf{\textit{(5--b)}}\;  the two-sided ideal  $ \, \J $  is in fact a coideal as well, which boils down to having  $ \; \Delta(\eta\hskip0,5pt) \in \, \J \otimes \tilde{\mathbb{F}}_\hbar + \tilde{\mathbb{F}}_\hbar \otimes \J \; $  for every  $ \eta $  in our given set of generators of  $ \, \J \, $.
 \vskip7pt
   Once again, the proof of the above two points is just a matter of computations, although non-trivial indeed; the interested reader will find them in  Appendix \ref{appendix}.
\end{free text}

\vskip11pt

   We are now ready to recap the payout of the previous construction.

\vskip17pt

\begin{theorem}  \label{thm: pres-Fhg}
 There exists a unique topological  $ \kh $--superbialgebra  $ \Fhglnp $  enjoying the following properties:
 \vskip5pt
   (a)\;  it is topologically generated by  $ \, \big\{\, x_{i{}j} \,\big|\, i \, , j \in\! I_n \big\} \, $,
   with parity
 $ \, |x_{i{}j}| = p_{i{}j} := p(i)+p(j) \, $,  \,subject to the following relations, where  $ \, q := e^\hbar \, $,  $ \, q_s := q^{{(-1)}^{p(s)}} \, $:
%
%
  $$  \displaylines{
   \hskip35pt   x_{i{}j}^2  \; = \;  0   \hskip185pt \qquad  \text{\ for \ \ }  p_{i{}j} = \one  \cr
   x_{i{}j} \, x_{i{}k}  \; = \;  {(-1)}^{p_{i{}j} p_{i{}k}} q_i \, x_{i{}k} \, x_{i{}j}  \hskip125pt \qquad  \text{for \ \ } j < k  \cr
   x_{i{}j} \, x_{h{}j}  \; = \;  {(-1)}^{p_{i{}j} p_{h{}j}} q_j \, x_{h{}j} \, x_{i{}j}  \hskip125pt \qquad  \text{for \ \ } i<h  \cr
   \qquad \hskip19pt   x_{i{}j} \, x_{h{}k}  \; = \;  {(-1)}^{p_{i{}j} p_{h{}k}} x_{k{}h} \, x_{i{}j}  \hskip119pt \qquad  \text{for \ \ }  i < h \, , \, j > k   \hskip61pt  \cr
   \hskip3pt   x_{i{}j} \, x_{h{}k}  \; = \;  {(-1)}^{p_{i{}j} p_{h{}k}} x_{h{}k} \, x_{i{}j} \, + {(-1)}^{{p_{i{}j} p_{i{}k}}} \big( q_i^{+1} - q_i^{-1} \big) \, x_{i{}k} \, x_{h{}j}   \hskip11pt  \text{\;\;\ \ for \ \ } i < h \, , \, j < k  }  $$
 \vskip5pt
   (b)\;  it is complete with respect to the  $ I_\hbar $--adic  relation,  where  $ I_\hbar $  is the two-sided ideal of  $ \Fhglnp $  generated by the set  $ \, \big\{ x_{i{}j} - \delta_{i{}j} \,\big|\, i , j \in I_n \big\} \cup \big\{\, \hbar \, 1_{F_\hbar[[\textit{GL}_n^p]]} \,\big\} \, $;
 \vskip5pt
   (c)\;  its coproduct and counit are defined, in terms of the above presentation, by
\begin{equation}  \label{eq: coprod+counit for x(i,j) - BIS}
  \Delta(x_{ij}) = {\textstyle \sum\limits_{a=1}^n} \, {(-1)}^{p_{i{}a} p_{a{}j}} x_{i{}a} \otimes x_{a{}j}   \quad ,  \qquad
      \epsilon(x_{i{}j}) = \delta_{i{}j}
   \qquad \quad  \forall \;\; i \, , j \in I_n   \quad
\end{equation}
%
%
 \vskip5pt
   (d)\;  it admits a non-degenerate pairing of (topological)\/  $ \kh $--superbialgebras
  $$  \langle\,\ ,\ \rangle : \Fhglnp \times \Uhglnp \relbar\joinrel\relbar\joinrel\relbar\joinrel\longrightarrow \kh  $$
 given on generators   --- $ \, \forall \, i \, , j \, , k \in I_n \, $,  $ \, t \in I_{n-1} \, $  ---   by
\begin{equation}  \label{eq: pairing-formulas}
  \big\langle x_{i{}j} \, , E_t \big\rangle  \; = \;  \delta_{i,t} \, \delta_{j,t+1} \;\; ,  \quad
\big\langle x_{i{}j} \, , F_t \big\rangle  \; = \,  \delta_{i,t+1} \, \delta_{j,t} \;\; ,  \quad  \big\langle x_{i{}j} \, , \varGamma_k \big\rangle  \; = \;  \delta_{i,k} \, \delta_{j,k}
\end{equation}
 \vskip5pt
   (e)\;  through the pairing in (d) above, it is isomorphic (as a topological\/  $ \kh $--superbialgebra)  to  $ \, {\Uhglnp}^* $,  \,the linear dual to  $ \Uhglnp \, $.  In particular,  $ \Fhglnp $  itself is a  \textsl{(topological) Hopf superalgebra}  over\/  $ \kh \, $.
\end{theorem}

\begin{proof}
 The construction in  \S \ref{free: constr.-FhG}  provides us with a topological  $ \kh $--superbialgebra  $ \Fhglnp $  which, indeed, do fulfill all requirements of claims  \textit{(a)\/}  through  \textit{(d)\/} above,  \textsl{but\/}  for the non-degeneracy of the pairing in  \textit{(d)}.  As to this last property, as well as to claim  \textit{(e)},  we still need some extra work.
                                                              \par
   To begin with, when passing to the ``semiclassical limit'', we have that
  $$  F_0\big[\big[\textit{GL}_{\,n}^{\,p}\big]\big]  \; := \;  \Fhglnp \Big/ \hbar \, \Fhglnp  $$
 is a  \textsl{commutative\/}  topological  $ \k $--superbialgebra  which is generated by the cosets of the  $ x_{i{}j} $'s   --- that we still denote by the same symbols, slightly abusing notation ---   having parity  $ \, |x_{i{}j}| = p_{i{}j} \, $.  Moreover,  $ F_0\big[\big[\textit{GL}_{\,n}^{\,p}\big]\big] $  is complete with respect to the ideal generated by all the  $ (x_{i{}j} - \delta_{i{}j}) $'s,  and it has (super)coalgebra structure described as in  \eqref{eq: coprod+counit for x(i,j) - BIS}  again.  Finally, the  $ \kh $--bilinear  pairing
%
%
 $ \, \langle\,\ ,\ \rangle \, $
 in  \textit{(d)\/}  between  $ \Fhglnp $  and  $ \Uhglnp $  induces   --- just by modding out  $ \hbar $  ---   a  $ \k $--bilinear  pairing
  $ \; {\langle\,\ ,\ \rangle}_0 : F_0\big[\big[\textit{GL}_{\,n}^{\,p}\big]\big] \times U_0\big(\liegl_{\,n}^{\,p}\big) \relbar\joinrel\relbar\joinrel\relbar\joinrel\longrightarrow \kh \; $
 of (topological)  $ \k $--superbialgebras.  Now recall that  $ \; U_0\big(\liegl_{\,n}^{\,p}\big) := \Uhglnp \Big/ \hbar \, \Uhglnp \; $  naturally identifies with  $ U\big(\liegl_{\,n}^{\,p}\big) \, $.  Then, using the above presentation of  $ \, F_0\big[\big[\textit{GL}_{\,n}^{\,p}\big]\big] \, $   --- with specific generators whose coproduct and counit are given by the specific formulas  \eqref{eq: coprod+counit for x(i,j) - BIS}  ---   and its superbialgebra pairing  $ \, {\langle\,\ ,\ \rangle}_0 \, $ with  $ \, U_0\big(\liegl_{\,n}^{\,p}\big) = U\big(\liegl_{\,n}^{\,p}\big) \, $   --- also explicitly described by  \eqref{eq: pairing Fhglnp-Uhglnp}  ---   one quickly realises that  $ \, F_0\big[\big[\textit{GL}_{\,n}^{\,p}\big]\big] \, $  in turn identifies with  $ \, F\big[\big[\textit{GL}_{\,n}^{\,p}\big]\big] = {U\big(\liegl_{\,n}^{\,p}\big)}^* \, $,  the  \textsl{``algebra of functions'' on the formal general linear supergroup}  $ \textit{GL}_{\,n}^{\,p} \, $.  In fact, the identification is the unique one through which  $ \, \big\{\, x_{i{}j} \,\big|\, i \, , j \in I_n \,\big\} \, $  identifies with the dual basis to the basis  $ \, \big\{\, \textrm{e}_{i{}j} \,\big|\, i \, , j \in I_n \,\big\} \, $  of elementary matrices of  $ \liegl_{\,n}^{\,p} \, $.  In particular, the ``semi\-clas\-sical'' pairing  $ \, {\langle\,\ ,\ \rangle}_0 \, $  is non-degenerate.
                                                              \par
   The previous analysis proves that  $ \; F_0\big[\big[\textit{GL}_{\,n}^{\,p}\big]\big] \, := \, \Fhglnp \Big/ \hbar \, \Fhglnp \; $  is isomorphic to  $ \, {U\big(\liegl_{\,n}^{\,p}\big)}^* = {\Big( \Uhglnp \Big/ \hbar \, \Uhglnp \!\Big)}^{\!*} \, $,  with the pairing  $ \, {\langle\,\ ,\ \rangle}_0 \, $  identifying with the standard ``evaluation pairing'', which is non-degenerate.  In addition, we note that, since  $ I_\hbar $  contains  $ \, \hbar \, \Fhglnp \, $,  then  $ \Fhglnp $  is also  $ \hbar $--adically  complete.  Therefore, by a standard argument (namely an ``\,approximation modulo  $ \hbar^{\,n} \, $''  process, followed by taking the limit for  $ \, n \rightarrow +\infty \, $)  we can conclude that the pairing
 $ \; \langle\,\ ,\ \rangle : \Fhglnp \times \Uhglnp \relbar\joinrel\relbar\joinrel\longrightarrow \kh \; $
 itself is non-degenerate and, even more, through it  $ \Fhglnp $  identifies with  $ {\big( \Uhglnp \big)}^* $,  \,hence  \textit{(e)\/}  is proved too.
\end{proof}

\vskip7pt

\begin{rmk}  \label{rmk: Fhglnp-is-QFSHSA}
 Because of its properties, detailed in  Theorem \ref{thm: pres-Fhg}  above   --- in particular, claim  \textit{(e)}  ---   the Hopf superalgebra  $ \Fhglnp $  introduced there is a ``quantized formal series Hopf superalgebra'' in the sense of Drinfeld   --- cf.\  \cite{Dr}, \S 7,  and  \cite{Ga}, Definition 1.2, suitably adapted to the present setup of quantum  \textsl{super\/}groups).  This leads us to the following definition.
\end{rmk}

\vskip7pt

\begin{definition}  \label{def: QFSHSA Fhglnp}
 We call the (topological) Hopf superalgebra  $ \Fhglnp $  introduced in  Theorem \ref{thm: pres-Fhg}  above the  \textit{quantized formal series Hopf superalgebra (or just ``QFSHSA'', in short) associated with the Poisson supergroup  $ \textit{GL}_{\,n}^{\,p} \, $}.   \hfill  $ \diamond $
\end{definition}

\vskip9pt

   We conclude this part with a PBW-like theorem, providing for  $ \Fhglnp $  a (topological) basis of ordered monomials in the generators:

\vskip11pt

\begin{theorem}  \label{thm: PBW x Fhglnp}
 \textsl{(PBW Theorem for  $ \Fhglnp \, $)}
 Let us fix any total order in the set  $ \, \big\{\, x_{i{}j} \,\big|\, i , j \in I_n \,\big\} \, $  of generators of  $ \Fhglnp \, $.  Then the set of all truncated ordered monomials in the  $ x_{i{}j} $'s,  namely
  $$  \bigg\{\, {\textstyle \mathop{\overrightarrow{\prod}}\limits_{i, j \in I_n}} \hskip-3pt x_{i{}j}^{\,e_{i{}j}} \;\bigg|\; e_{i{}j} \in \NN \, , \, \forall \; i \, , j \in \! I_n \, , \; e_{i{}j} \leq 1 \text{\;\ if\;\ }  p_{i,j} = \one \,\bigg\}  $$
 is a\/  $ \kh $--basis  (in topological) of the\/  $ \kh $--module  $ \Fhglnp \, $.
\end{theorem}

\begin{proof}
 Let us write for simplicity  $ \, F_\hbar := \Fhglnp \, $.  The ``semiclassical limit'' of  $ F_\hbar $  is the quotient  $ \, F_0 := F_\hbar \Big/ \hbar\,F_\hbar \, $:  by construction, it is a topological Hopf algebra over  $ \k \, $,  which is isomorphic to  $ \, F\big[\big[GL_{\,n}^{\,p}\big]\big] = {U\big(\liegl_{\,n}^{\,p}\big)}^* \, $,  \,the ``formal superalgebra of functions'' on the (formal) supergroup  $ GL_{\,n}^{\,p} \, $.  Inside  $ F_0 \, $,  denote by  $ \overline{x}_{i{}j} $  the coset of  $ x_{i{}j} $  modulo  $ \, \hbar\,F_\hbar \, $:  then it is known that the set of truncated ordered monomials in the  $ \overline{x}_{i{}j} $'s,  i.e.\
 $ \; \bigg\{\, {\textstyle \mathop{\overrightarrow{\prod}}\limits_{i, j \in I_n}} \hskip-3pt \overline{x}_{i{}j}^{\;e_{i{}j}} \;\bigg|\; e_{i{}j} \in \NN \, , \, \forall \; i \, , j \in \! I_n \, , \; e_{i{}j} \leq 1 \text{\;\ if\;\ }  |i|+|j| = \one \,\bigg\} \; $,
 is a topological  $ \k $--basis  of  $ F\big[\big[GL_{\,n}^{\,p}\big]\big] \, $,  \,that is any  $ \, \overline{\phi} \in F\big[\big[GL_{\,n}^{\,p}\big]\big] \, $  can be uniquely written as a formal  $ \k $--linear  combination
 $ \; \overline{\phi} \, = \hskip-13pt \sum\limits_{\substack{\underline{e} \, \in \, \NN^{I_n \times I_n}  \\   e_{i{}j} \leq 1 \text{\,\ if\,\ }  |i|+|j| = \one}} \hskip-11pt c_{\,\underline{e}} \hskip3pt {\textstyle \mathop{\overrightarrow{\prod}}\limits_{i, j \, \in \, I_n}} \hskip-3pt \overline{x}_{i{}j}^{\;e_{i{}j}} \; $
 of truncated ordered monomials in the  $ \overline{x}_{i{}j} $'s,  for suitable (unique!) coefficients  $ \, c_{\,\underline{e}} \in \k \, $;  hereafter we will write such an expansion with the simpler notation
 $ \; \overline{\phi} \, = \sum_{\,\underline{e}} \hskip1pt c_{\,\underline{e}} \hskip3pt \overrightarrow{\prod}_{\,i,j} \hskip3pt \overline{x}_{i{}j}^{\;e_{i{}j}} \; $.
 From this, a standard argument gives us the expected result, as follows.  Let  $ \, \psi \in \Fhglnp \, $:  then its coset modulo  $ \, \hbar\, F_\hbar \, $  in  $ F_0 $  uniquely expands as
 $ \; \overline{\psi} \, = \sum_{\,\underline{e}} \hskip1pt c^{\,(0)}_{\,\underline{e}} \hskip3pt \overrightarrow{\prod}_{\,i,j} \hskip3pt \overline{x}_{i{}j}^{\;e_{i{}j}} \; $
%
%
%
 for suitable (unique!)  $ \, c^{\,(0)}_{\,\underline{e}} \in \k \, $.  Then  $ \; \psi \, = \, \sum_{\,\underline{e}} \hskip1pt c^{\,(0)}_{\,\underline{e}} \hskip3pt \overrightarrow{\prod}_{\,i,j} \hskip3pt x_{i{}j}^{\;e_{i{}j}} + \, \hbar \, \psi_1 \; $  for some  $ \, \psi_1 \in \Fhglnp \, $,  \,and we can start again with  $ \psi_1 $  instead of  $ \psi \, $.  Iterating this argument, one eventually finds
 $ \; \psi \, = \, \sum_{\,n \in \NN} \hskip1pt \hbar^n \sum_{\,\underline{e}} \hskip1pt c^{\,(n)}_{\,\underline{e}} \hskip1pt \overrightarrow{\prod}_{\,i,j} \hskip3pt x_{i{}j}^{\;e_{i{}j}} = \, \sum_{\,\underline{e}} \Big( \sum_{\,n \in \NN} c^{\,(n)}_{\,\underline{e}} \hbar^n \Big) \hskip1
pt \overrightarrow{\prod}_{\,i,j} \hskip3pt x_{i{}j}^{\;e_{i{}j}} \, $,
 for suitable  $ \, c^{\,(n)}_{\,\underline{e}} \in \k \, $,  \,which actually proves our claim.
\end{proof}

\vskip13pt

\subsection{Multiparametric QFSHSA's for the general linear supergroup}  \label{subsec: Mp-QFSHA's}
 Much like we constructed a QFSHA  $ \Fhglnp $  as dual to Yamane's uniparametric QUESA  $ \Uhglnp \, $,  one might follow the same strategy to construct a  \textsl{multiparametric\/}  QFSHSA dual to the multiparametric QUESA  $ \UhPhiglnp \, $,  eventually finding parallel results.  However, we present hereafter a different approach: as  $ \UhPhiglnp $  is obtained as deformation by twist of  $ \Uhglnp \, $,  we can obtain its dual as 2--cocycle deformation of  $ \, {\big( \Uhglnp \big)}^* = \Fhglnp \, $  using for that the unique 2--cocycle for  $ \Fhglnp $  that ``corresponds'' to the twist  $ \F_\Phi $  for $ \Uhglnp \, $.  This shortcut eventually leads us to obtain a presentation by generators and relations of the final, ``deformed'' object, directly ``by deforming'' the presentation by generators and relations that we already have for  $ \Fhglnp $   --- from  Theorem \ref{thm: pres-Fhg}.  The construction also provides us with a non-degenerate Hopf pairing of our multiparametric QUESA with  $ \UhPhiglnp $  through which the former identifies with the dual to the latter.

\vskip11pt

\begin{free text}  \label{free: deform.'s-vs-duality}
 \textbf{Deformation construction vs.\ linear duality.}
 A standard construction in Hopf algebra theory is that of  \textit{deformations},  either by  \textit{twist},  or by  \textit{2--cocycle}.  This extends to various type of ``generalised'' Hopf algebras, including Hopf superalgebras, and even to ``quantum Hopf superalgebras'', which are Hopf superalgebras only in a suitable topological sense, such as our QUESA's in  \S \ref{sec: mp-QUESA}  and our QFSHSA's in  \S \ref{subsec: QFSHSA's x Yamane}.  Without repeating them here, we quote the main relevant definitions from  \cite{GGP}, \S 4.1.5.  We also need an auxiliary result, connecting the two types of ``deformation'' with linear duality: the statement here is written in a ``plain version'', but it also extends to more general situation such as those of topological Hopf superalgebras, including those associated with quantum supergroups.
\end{free text}

\vskip7pt

\begin{prop}  \label{prop: deformation-vs-duality}
 Let  $ H $  be a Hopf superalgebra, and  $ H^* $  its linear dual.
 \vskip3pt
   {\it (a)}\,  Let  $ \F $  be a twist for  $ H \, $,  and  $ \sigma_{{}_\F} $  the image of  $ \F $  in $ {(H \otimes H)}^* $  for the natural composed embedding  $ \, H \otimes H \lhook\joinrel\relbar\joinrel\longrightarrow H^{**} \otimes H^{**} \lhook\joinrel\relbar\joinrel\longrightarrow {\big( H^* \otimes H^* \big)}^* \, $.  Then  $ \sigma_{{}_\F} $  is a 2--cocycle for  $ H^* \, $,  and there exists a canonical isomorphism  $ \, {\big( H^* \big)}_{\sigma_{{}_\F}} \!\cong {\big( H^\F \,\big)}^* \, $.
 \vskip3pt
   {\it (b)}\,  Assume there exists a natural identification  $ \, {(H \otimes H)}^* = H^* \otimes H^* \, $  (e.g.,  $ H $  is finite-dimensional, or  ``$ \,\otimes $''  is meant in a suitable, topological sense).  Let  $ \sigma $  be a 2--cocycle for  $ H \, $,  \,and let  $ \F_\sigma $  be the image of  $ \sigma $  through this identification.  Then  $ \F_\sigma $  is a twist for  $ H^* \, $,  and there exists a canonical isomorphism  $ \, {\big( H^* \big)}^{\F_\sigma} \cong {\big( H_\sigma \big)}^* \, $.
\qed
\end{prop}

\vskip11pt

\begin{free text}  \label{free: constr.-FhPhiG}
 \textbf{Multiparametric QFSHSA over  $ \textit{GL}_n^{\,p} \, $:  the construction.}
 We resume assumptions, notation and terminology from  \S \ref{subsec: Yam-Mp_QUESA's}.  Let  $ \Uhglnp $ be Yamane's QUESA from  Definition \ref{def: QUESA Uhglnp},  and let  $ \, \Fhglnp \, $  be the QFSHSA from  Definition \ref{def: QFSHSA Fhglnp},  which naturally identifies with  $ \, {\big( \Uhglnp \big)}^* \, $  by  Theorem \ref{thm: pres-Fhg}\textit{(e)}.
                                                              \par
   We choose any antisymmetric matrix  $ \, \Phi = {\big( \phi_{t,\ell} \big)}_{t=1,\dots,n;}^{\ell=1,\dots,n;} \in \mathfrak{so}_n\big(\kh\big) \, $,  \,and from it two we define the element
\begin{equation}   \label{eq: Resh-twist_F-uPhgd - again}
  \F_\Phi  \,\; := \;\,  \exp\Big(\hskip1pt \hbar \; 2^{-1} \, {\textstyle \sum_{t,\ell=1}^n} \, \phi_{t,\ell} \, \varGamma_t \otimes \varGamma_\ell \Big)  \;\; \in \;\;  \Uhglnp \,\widehat{\otimes}\, \Uhglnp
\end{equation}
 which is a  {\sl twist\/}  element for  $ \Uhglnp \, $,  in the standard sense recalled right above.  The corresponding twist deformation of Yamane's QUESA is our MpQUESA, namely  $ \, \UhPhiglnp := {\Uhglnp}^{\F_\Phi} $  of  Definition \ref{def: Multipar UhPhiglnp}.
                                                       \par
   To make life easier, we adopt hereafter the simpler notation  $ \, \uhg := \Uhglnp \, $  and  $ \, \fhg := \Fhglnp \, $.  Now, according to  Proposition \ref{prop: deformation-vs-duality},  the twist  $ \F_\Phi $  of  $ \uhg $  identifies with some  $ \, \sigma_\Phipicc := \sigma_{\F_\Phi} $  which is a 2--cocycle for  $ \, {\big( \uhg \big)}^* = \fhg \, $.  Indeed,  $ \sigma_\Phipicc $  is simply given by  \textsl{evaluation at  $ \F \, $},  namely
\begin{equation}   \label{eq: sigma_Phi}
  \sigma_\Phipicc \, : \, \fhg
  \times \fhg \relbar\joinrel\relbar\joinrel\longrightarrow \kh
  \quad ,  \qquad  (\varphi\,,\psi) \,\mapsto\, \big\langle\,
  \varphi \otimes \psi \, , \, \F \,\big\rangle
\end{equation}
   \indent   Now, from  \eqref{eq: sigma_Phi}  and the formulas  \eqref{eq: pairing-formulas},  direct calculation gives
  $$  \displaylines{
   \sigma_\Phipicc\big( x_{i,\,r} \, , x_{\ell,\,h} \big)  \;
 = \;  \big\langle\, x_{i,\,r} \otimes x_{\ell,\,h} \, , \, \F \,\big\rangle  \; =   \hfill  \cr
   \qquad   = \;  {\textstyle \sum\limits_{m=0}^{+\infty}} {{\,\hbar^m\,} \over {\,m! \, 2^m\,}} \,
   \left\langle\, x_{i,\,r} \otimes x_{\ell,\,h} \, , \, {\left(\, {\textstyle \sum_{t,\,k=1}^n} \,
   \phi_{t,k} \, \varGamma_t \otimes \varGamma_k \right)}^m \,\right\rangle  \; =   \hfill  \cr
   = \;  {\textstyle \sum\limits_{m=0}^{+\infty}} {{\,\hbar^m\,} \over {\,m! \, 2^m\,}} \, \left\langle\, \Delta^{(m-1)}\big( x_{i,\,r} \otimes x_{\ell,\,h} \big) \, , \, {\left(\, {\textstyle \sum_{t,\,k=1}^n} \, \phi_{t,k} \, \varGamma_t \otimes \varGamma_k \right)}^{\otimes m} \,\right\rangle  }  $$
 Let us consider  $ \; \left\langle\, \Delta^{(m-1)}\big( x_{i,\,r} \otimes x_{\ell,\,h} \big) \, ,
 \, {\left(\, {\textstyle \sum_{t,\,k=1}^n} \, \phi_{t,k} \,
 \varGamma_t \otimes \varGamma_k \right)}^{\otimes m} \,\right\rangle \; $.  Definitions give
  $$  \displaylines{
   \left\langle\, \Delta^{(m-1)}\big( x_{i,\,r} \otimes x_{\ell,\,h} \big) \, ,
   \, {\left(\, {\textstyle \sum_{t,\,k=1}^n} \, \phi_{t,k} \,
   \varGamma_t \otimes \varGamma_k \right)}^{\otimes m} \,\right\rangle  \; =   \hfill  \cr
   = \!  \hskip-23pt {\textstyle \sum\limits_{\substack{\qquad s_1,\dots,\,s_{m-1}=1  \\
   \qquad e_1,\dots,\,e_{m-1}=1}}^n}
 \hskip-25pt \epsilon(\underline{s},\underline{e})
 \left\langle x_{i,\,s_1} \otimes x_{\ell,\,e_1} \otimes \cdots \otimes x_{s_{m-1},\,r} \otimes x_{e_{m-1},\,h} \, ,
   {\bigg(\, {\textstyle \sum\limits_{t,\,k=1}^n} \phi_{t,k} \,
   \varGamma_t \otimes \varGamma_k \bigg)}^{\!\!\otimes m} \right\rangle  \; =  \cr
   \hfill   = \;  {\textstyle \sum\limits_{\substack{s_1,\dots,\,s_{m-1}=1  \\
   e_1,\dots,\,e_{m-1}=1}}^n}
 \hskip-11pt \epsilon(\underline{s},\underline{e})
 \, {\textstyle \prod\limits_{c=1}^m} \,
   {\textstyle \sum\limits_{t,\,k=1}^n} \, \phi_{t,k} \, \big\langle x_{s_{c-1},\,s_c} \, ,
   \varGamma_t \big\rangle \, \big\langle x_{e_{c-1},\,e_c} \, , \varGamma_k \, \big\rangle  }  $$
 where we set  $ \, s_0 := i \, $,  $ \, s_m := r \, $,  $ \, e_0 := \ell \, $,  $ \, e_m := h \, $
 and
%
\begin{equation}  \label{eq: sigma(s,e)}
   \epsilon(\underline{s}\,,\underline{e}\,)  \, := \,  {(-1)}^{\sum_{t=0}^{m-2} \sum_{k=t+1}^{m-1} p(x_{s_k,s_{k+1}}) \cdot p(x_{e_t,e_{t+1}})}
\end{equation}
 Now, the formulas defining the pairing imply that
 $ \; \big\langle x_{s_{c-1},\,s_c} \, , \varGamma_t \big\rangle \, \big\langle x_{e_{c-1},\,e_c} \, ,
 \varGamma_k \, \big\rangle \, = \, 0 \; $
 when  $ \, s_{c-1} \not= s_c \, $  or  $ \, e_{c-1} \not= e_c \, $;  \,therefore, in the previous computation all relevant signs  $ \epsilon(\underline{s}\,,\underline{e}\,) $  as in  \eqref{eq: sigma(s,e)}  actually boil down to be  ``$ \, +1 \, $''  and eventually one gets
  $$  \displaylines{
   \sigma_\Phipicc \big( x_{i,\,r} \, , x_{\ell,\,h} \big)  \;
   = \;  \delta_{i,r} \, \delta_{\ell,h} \,
   {\textstyle \sum\limits_{m=0}^{+\infty}} {{\,\hbar^m\,} \over {\,m! \, 2^m\,}} \,
   {\left(\, {\textstyle \sum_{t,\,k=1}^n} \, \phi_{t,k} \, \big\langle x_{i,\,i} \, ,
   \varGamma_t \big\rangle \, \big\langle x_{\ell,\,\ell} \, , \varGamma_k \, \big\rangle \right)}^m  \; =
   \hfill  \cr
   \quad \quad \quad \quad   = \;  \delta_{i,r} \, \delta_{\ell,h} \,
   {\textstyle \sum\limits_{m=0}^{+\infty}} {{\,\hbar^m\,} \over {\,m! \, 2^m\,}} \, {(\phi_{i,\ell})}^m  \; = \;
   \delta_{i,r} \, \delta_{\ell,h} \, \exp\big( \hbar \, \phi_{i,\ell} \big)  \; = \;
   \delta_{i,r} \, \delta_{\ell,h} \, e^{\hbar \, 2^{-1} \, \phi_{i,\ell}}   \hfill  }  $$
%
\begin{equation}  \label{eq: formula x sigma_F}
  \hskip-13pt \text{i.e.} \hskip21pt \qquad   \sigma_\Phipicc\big( x_{i,\,r} \, , x_{\ell,\,h} \big)  \; = \;  \delta_{i,r} \,
  \delta_{\ell,h} \, e^{\hbar \; 2^{-1} \phi_{i,\ell}}   \qquad \qquad
  \forall \;\; i \, , r , \ell , h \in \{1,\dots,n\}   \hfill
\end{equation}
   \indent   Using this formula, the deformed product in  $ {\fhg}_{\sigma_\Phipicc} $  can be described as follows:
  $$  \displaylines{
   x_{r,\,s} \raise-1pt\hbox{$ \; \scriptstyle \dot\sigma_\Phipicc $}\, x_{\ell,\,t}  \; := \;
   \sigma_\Phipicc\big( {(x_{r,\,s})}_{(1)} \, , \, {(x_{\ell,\,t})}_{(1)} \big) \; {(x_{r,\,s})}_{(2)} \, {(x_{\ell,\,t})}_{(2)} \; \sigma_\Phipicc^{-1}\!\big( {(x_{r,\,s})}_{(3)} \, , \, {(x_{\ell,\,t})}_{(3)} \big)  \; = \hfill  \cr
   \hfill   = \;  \sigma_\Phipicc\big( x_{r,\,r} \, , \, x_{\ell,\,\ell} \big) \; x_{r,\,s} \, x_{\ell,\,t} \; \sigma_\Phipicc^{-1}\!\big( x_{s,\,s} \, , \, x_{t,\,t} \big)  \; =  \;  e^{\hbar \; 2^{-1} (\phi_{r,\ell} \, - \, \phi_{s,t})} \, x_{r,\,s} \, x_{\ell,\,t}  }  $$
%
\begin{equation}  \label{eq: formula x sigma-deformed-product}
 \hskip-11pt \text{i.e.} \hskip17pt \qquad
  x_{r,\,s} \raise-1pt\hbox{$ \; \scriptstyle \dot\sigma_\Phipicc $}\, x_{\ell,\,t}  \; = \;
  e^{\hbar \; 2^{-1} (\phi_{r,\ell} \, - \, \phi_{s,t})} \; x_{r,\,s} \, x_{\ell,\,t}
  \quad \qquad  \forall \;\; r , s , \ell , t \in \{1,\dots,n\}
\end{equation}
 \vskip5pt
   Note that this formula shows how the new, deformed product is equivalent modulo  $ \, \hbar \, $  to the old one: this happens because we work with 2--cocycles  $ \sigma_\Phipicc $  of the form  $ \, \sigma_\Phipicc = \exp\big( \hbar \, \varsigma \big) \, $  where  $ \, \varsigma \in {\big( \fhg \otimes \fhg \big)}^* \, $,  \,so that  $ \, \sigma_\Phipicc = \text{id} + \mathcal{O}(\hbar) \, $.
By this same reason,  \textit{any set of elements which generate, as an algebra, the QFSHSA  $ \fhg $  will also generate it w.r.t.\ the new, deformed product}.  For this reason,  \eqref{eq: formula x sigma-deformed-product}  is enough to describe  $ {\fhg}_{\sigma_\Phipicc} $  as the latter is generated (w.r.t.\ the new product) by the  $ x_{r,\,s} $'s,  just like  $ \fhg $  was (with the old product).
                                                                                       \par
   More in detail, from the original presentation of  $ \fhg $  by generators
   --- the  $ x_{r,\,s} $'s  ---   and relations   --- namely, those in  Theorem \ref{thm: pres-Fhg}  ---
   using  \eqref{eq: formula x sigma-deformed-product}  above we find a similar presentation of
   $ {\fhg}_{\sigma_\Phipicc} $  by generators   --- the  $ x_{r,\,s} $'s  again ---
   and relations, where the latter depend on ``multiparameters'' of the form
  $$  q_{\,r,s}  \, := \,  q^{\,\phi_{r,s}} \, = \, e^{\hbar\,\phi_{r,s}}   \qquad \qquad   \forall \;\; r, s \in \{1,\dots,n\}  $$
 Now, when writing these relations in detail, one finds the following outcome:
\end{free text}

\vskip9pt

\begin{theorem}  \label{thm: pres-FhPhiglnp}
 There exists a unique topological  $ \kh $--superbialgebra  $ \FhPhiglnp $  enjoying the following properties:
 \vskip5pt
   (a)\;  it is topologically generated by  $ \, \big\{\, x_{i{}j} \,\big|\, i \, , j \in\! I_n \big\} \, $,  with parity
 $ \, |x_{i{}j}| = p_{i{}j} := p(i)+p(j) \, $,  \,subject to the following relations:
  $$  \displaylines{
   \hskip51pt   x_{i{}j}^2  \; = \;  0   \hskip173pt  \big(\; p_{i{}j} = \one \,\big)  \cr
   x_{i{}j} \, x_{i{}k}  \; = \;  {(-1)}^{p_{i{}j} p_{i{}k}} q_{\,i} \, q_{j,k}^{\;-1} \, x_{i{}k} \, x_{i{}j}  \hskip105pt  \big(\; j < k \,\big)  \cr
   x_{i{}j} \, x_{h{}j}  \; = \;  {(-1)}^{p_{i{}j} p_{h{}j}} q_j \, q_{\,i,h}^{\;+1} \, x_{h{}j} \, x_{i{}j}  \hskip105pt  \big(\, i<h \,\big)  \cr
   \hskip17pt   x_{i{}j} \, x_{h{}k}  \; = \;  {(-1)}^{p_{i{}j} p_{h{}k}\,} q_{\,i,h}^{\;+1} \, q_{j,k}^{\;-1} \, x_{k{}h} \, x_{i{}j}  \hskip81pt  \big(\, i<h \, , \; j>k \,\big)  \cr
   \hskip7pt   x_{i{}j} \, x_{h{}k}  \; = \;  {(-1)}^{p_{i{}j} p_{h{}k}\,} q_{\,i,h}^{\;+1} \, q_{j,k}^{\;-1} \, x_{h{}k} \, x_{i{}j} \, + {(-1)}^{{p_{i{}j} p_{i{}k}}} \big( q_{\,i}^{+1} - q_{\,i}^{-1} \big) \, q_{j,k}^{\;-1} \, x_{i{}k} \, x_{h{}j}  \hskip17pt  \bigg(\; {{i < h} \atop {j < k}}  \,\bigg)  }  $$
 where  $ \; q := e^\hbar \, $,  $ \; q_{\,s} := e^{\hbar \, {(-1)}^{p(s)}} = q^{{(-1)}^{p(s)}} \, $,  $ \; q_{\,r,s} := e^{\hbar \, \phi_{r,s}} = q^{\,\phi_{r,s}} \, $;
 \vskip5pt
   (b)\;  it is complete with respect to the  $ I_\hbar $--adic  relation,  where  $ I_\hbar $  is the two-sided ideal of  $ \FhPhiglnp $  generated by the set  $ \, \big\{ x_{i{}j} - \delta_{i{}j} \,\big|\, i , j \in I_n \big\} \cup \big\{\, \hbar \, 1_{F^\Phipicc_\hbar[[\textit{GL}_n^p]]} \,\big\} \, $;
 \vskip5pt
   (c)\;  its coproduct and counit are defined, in terms of the above presentation, by
\begin{equation}  \label{eq: Phi-coprod+counit for x(i,j) - Phi-case}
  \Delta(x_{ij}) = {\textstyle \sum\limits_{a=1}^n} \, {(-1)}^{p_{i{}a} p_{a{}j}} x_{i{}a} \otimes x_{a{}j}   \quad ,  \qquad
      \epsilon(x_{i{}j}) = \delta_{i{}j}
   \qquad \quad  \forall \;\; i \, , j \in I_n   \quad
\end{equation}
%
%
 \vskip5pt
   (d)\;  it admits a non-degenerate pairing of (topological)\/  $ \kh $--superbialgebras
  $$  \langle\,\ ,\ \rangle : \FhPhiglnp \times \UhPhiglnp \relbar\joinrel\relbar\joinrel\relbar\joinrel\longrightarrow \kh  $$
 given on generators   --- $ \, \forall \, i \, , j \, , k \in I_n \, $,  $ \, t \in I_{n-1} \, $  ---   by
\begin{equation}  \label{eq: pairing-formulas x Phi-case}
  \big\langle x_{i{}j} \, , E_t \big\rangle  \; = \;  \delta_{i,t} \, \delta_{j,t+1} \;\; ,  \quad
\big\langle x_{i{}j} \, , F_t \big\rangle  \; = \,  \delta_{i,t+1} \, \delta_{j,t} \;\; ,  \quad  \big\langle x_{i{}j} \, , \varGamma_k \big\rangle  \; = \;  \delta_{i,k} \, \delta_{j,k}
\end{equation}
 \vskip5pt
   (e)\;  through the pairing in (d) above, it is isomorphic (as a topological\/  $ \kh $--superbialgebra)  to  $ \, {\UhPhiglnp}^* $,  \,the linear dual to  $ \UhPhiglnp \, $.  In particular,  $ \FhPhiglnp $  itself is a  \textsl{(topological) Hopf superalgebra}  over\/  $ \kh \, $.
 \vskip5pt
   (f)\;  it is isomorphic (as a topological\/  $ \kh $--superbialgebra)  to the deformation of  $ \, \Fhglnp $  by the 2--cocycle  $ \sigma_\Phipicc $  in  \eqref{eq: sigma_Phi},  that is  $ \; \FhPhiglnp \cong {\big( \Fhglnp \big)}_{\sigma_\Phipicc} \; $.
\end{theorem}

\begin{proof}
 The key to the whole statement is claim  \textit{(f)},  with the latter   --- and then everything else --- following from the construction detailed in  \S \ref{free: constr.-FhPhiG}  above.
 \vskip5pt
   Let now see the details.  In  \S \ref{free: constr.-FhPhiG}  above we got an explicit, concrete description of the deformed Hopf superalgebra  $ \, {\big( \Fhglnp \big)}_{\sigma_\Phipicc} \, $.  In particular:
 \vskip3pt
   \textit{(1)}\;  as a (topological)  $ \kh $--coalgebra  it coincides with  $ \Fhglnp \, $,  by definition;
 \vskip3pt
   \textit{(2)}\;  the original generators  $ x_{i{}j} $  in  $ \, \Fhglnp = {\big( \Fhglnp \big)}_{\sigma_\Phipicc} \, $  still generate  $ {\big( \Fhglnp \big)}_{\sigma_\Phipicc} $  with respect to its new, deformed product  ``$ \, \raise-1pt\hbox{$ \; \scriptstyle \dot\sigma_\Phipicc $} \, $'';
 \vskip3pt
   \textit{(3)}\;  the link between the new product  ``$ \, \raise-1pt\hbox{$ \; \scriptstyle \dot\sigma_\Phipicc $} \, $''  and the old one  ``$ \, \cdot \, $''  is given by
\begin{equation}  \label{eq: formula x sigma-deformed-product - Phi-case BIS}
   \qquad   x_{r,\,s} \raise-1pt\hbox{$ \; \scriptstyle \dot\sigma_\Phipicc $}\, x_{\ell,\,t}  \; = \;
{\big( q_{\,r,\ell}^{\,+1} \, q_{\,s,t}^{\,-1} \big)}^{+1/2} \; x_{r,\,s} \cdot x_{\ell,\,t}
   \quad \qquad  \forall \;\; r , s , \ell , t \in \{1,\dots,n\}
\end{equation}
 where  $ \; q_{\,r,s}^{\,\pm 1} := e^{\,\pm \hbar \, \phi_{r,s}} = q^{\,\pm \phi_{r,s}} \; $.
 \vskip5pt
   Then we get a presentation for  $ \, {\big( \Fhglnp \big)}_{\sigma_\Phipicc} $  using as generators the  $ x_{i{}j} $'s  (for all  $ \, i \, , j \in I_n \, $)  and deducing a full set of relations among them by taking the original relations among them  \textsl{with respect to the old product and re-writing them in terms of the new one},  by means of repeated applications of  \eqref{eq: formula x sigma-deformed-product - Phi-case BIS}.
                                                                        \par
   In order to illustrate this on a concrete example, we perform the computation to deduce the las trelation in claim  \textit{(a)}   --- where the symbol  ``$ \, \raise-1pt\hbox{$ \; \scriptstyle \dot\sigma_\Phipicc $} \, $'' is omitted for simplicity ---   directly from the last relation in claim  \textit{(a)\/}  of  Theorem \ref{thm: pres-Fhg},  i.e.\ for the original, undeformed Hopf superalgebra  $ \Fhglnp \, $.  By repeated applications of  \eqref{eq: formula x sigma-deformed-product - Phi-case BIS},  direct calculations give (for  $ \, i < h \, $,  $ \, j < k \, $)
  $$  \displaylines{
   x_{i{}j} \raise-1pt\hbox{$ \; \scriptstyle \dot\sigma_\Phipicc $}\, x_{h{}k}
     \; {\buildrel \eqref{eq: formula x sigma-deformed-product - Phi-case BIS} \over =} \;
   {\big( q_{\,i,h}^{\,+1} \, q_{\,j,k}^{\,-1} \big)}^{+1/2} \; x_{i,\,j} \cdot x_{h,\,k}  \; =   \hfill  \cr
   = \;  q_{\,i,h}^{\,+1/2} \, q_{\,j,k}^{\,-1/2} \,
 \Big( {(-1)}^{p_{i{}j} p_{h{}k}\,} \, x_{h{}k} \cdot x_{i{}j} \, + {(-1)}^{{p_{i{}j} p_{i{}k}}} \big( q_{\,i}^{+1} - q_{\,i}^{-1} \big) \, x_{i{}k} \cdot x_{h{}j} \Big)  \; =   \hfill  \cr
   = \;  {(-1)}^{p_{i{}j} p_{h{}k}\,} q_{\,i,h}^{\;+1/2} \, q_{j,k}^{\;-1/2} \, x_{h{}k} \cdot x_{i{}j} \, + {(-1)}^{{p_{i{}j} p_{i{}k}}} \big( q_{\,i}^{+1} - q_{\,i}^{-1} \big) \, q_{\,i,h}^{\;+1/2} \, q_{j,k}^{\;-1/2} \, x_{i{}k} \cdot x_{h{}j}  \; {\buildrel \eqref{eq: formula x sigma-deformed-product - Phi-case BIS} \over =}  \cr
   \qquad   {\buildrel \eqref{eq: formula x sigma-deformed-product - Phi-case BIS} \over =} \;
   {(-1)}^{p_{i{}j} p_{h{}k}\,} q_{\,i,h}^{\;+1/2} \, q_{j,k}^{\;-1/2} \, q_{\,h,i}^{\;-1/2} \, q_{k,j}^{\;+1/2} \; x_{h{}k} \raise-1pt\hbox{$ \; \scriptstyle \dot\sigma_\Phipicc $}\, x_{i{}j}  \,\; +   \hfill  \cr
   \hfill   + \;\,  {(-1)}^{{p_{i{}j} p_{i{}k}}} \big( q_{\,i}^{+1} - q_{\,i}^{-1} \big) \, q_{\,i,h}^{\;+1/2} \, q_{j,k}^{\;-1/2} \, q_{\,i,h}^{\;-1/2} \, q_{k,j}^{\;+1/2} \; x_{i{}k} \raise-1pt\hbox{$ \; \scriptstyle \dot\sigma_\Phipicc $}\, x_{h{}j}  \; =   \qquad  \cr
   \hfill   = \;  {(-1)}^{p_{i{}j} p_{h{}k}\,} q_{\,i,h}^{\;+1} \, q_{j,k}^{\;-1} \; x_{h{}k} \raise-1pt\hbox{$ \; \scriptstyle \dot\sigma_\Phipicc $}\, x_{i{}j} \, + {(-1)}^{{p_{i{}j} p_{i{}k}}} \big( q_{\,i}^{+1} - q_{\,i}^{-1} \big) \, q_{j,k}^{\;-1} \; x_{i{}k} \raise-1pt\hbox{$ \; \scriptstyle \dot\sigma_\Phipicc $}\, x_{h{}j}  }  $$
 which eventually, dropping the symbols  ``$ \, \raise-1pt\hbox{$ \; \scriptstyle \dot\sigma_\Phipicc $} \, $'',  yields the expected formula, q.e.d.
                                                                         \par
   All other relations are obtained through a similar process.  By this analysis, we get that  $ \, \FhPhiglnp := {\big( \Fhglnp \big)}_{\sigma_\Phipicc} \, $  does fulfil claims  \textit{(a)},  \textit{(b)},  \textit{(c)\/}  and  \textit{(f)}.
 \vskip5pt
   We are left now with claims  \textit{(d)\/}  and  \textit{(e)}.  In fact, both follow from the parallel claims in  Theorem \ref{thm: pres-Fhg}  and a direct application of  Proposition \ref{prop: deformation-vs-duality}\textit{(a)},  namely to  $ \, H := \Uhglnp \, $  and  $ \, H^* = \Fhglnp \, $.
\end{proof}

\vskip7pt

\begin{definition}  \label{def: QFSHSA FhPhiglnp}
 We call the (topological) Hopf superalgebra  $ \FhPhiglnp $  introduced in  Theorem \ref{thm: pres-Fhg}  above the  \textit{multiparametric quantized formal series Hopf superalgebra (or just ``QFSHSA'', in short) associated with the Poisson supergroup  $ \textit{GL}_{\,n}^{\,p} \, $  and the multiparameter  $ \, \Phi := {\big( \phi_{i{}j} \big)}_{i \in I_n}^{j \in I_n} \in \lieso_n\big( \kh \big) \, $}.   \hfill  $ \diamond $
\end{definition}

\vskip11pt

\begin{free text}  \label{free: semiclassical}
 \textbf{Specialisation to semiclassical limit.}
 When dealing with the QFSHSA's  $ \Fhglnp $  and  $ \FhPhiglnp \, $,  only one ``specialisation'' of the quantum parameter  $ \hbar $  is possible, namely the one yielding the ``semiclassical limits''
  $$  F_0\big[\big[{GL}_{\,n}^{\,p}\big]\big]  \; := \;  \Fhglnp \Big/ \hbar \, \Fhglnp  $$
 in the first case and
  $$  F_0^\Phipicc\big[\big[{GL}_{\,n}^{\,p}\big]\big]  \; := \;  \FhPhiglnp \Big/ \hbar \, \FhPhiglnp  $$
 in the second one.  One easily sees that
  $$  F_0\big[\big[{GL}_{\,n}^{\,p}\big]\big]  \; \cong \;  F\big[\big[{GL}_{\,n}^{\,p}\big]\big]   \qquad  \text{and}  \qquad   F_0^\Phipicc\big[\big[{GL}_{\,n}^{\,p}\big]\big]  \; \cong \;  F\big[\big[{GL}_{\,n}^{\,p}\big]\big]  $$
 as Hopf superalgebras over  $ \k \, $,  where  $ F\big[\big[{GL}_{\,n}^{\,p}\big]\big] $  is the ``algebra of functions on the formal algebraic group'' associated with  $ {GL}_{\,n}^{\,p} \, $.  There is, however, a difference: indeed, from its ``uniparametric'' quantisation  $ \Fhglnp $  the Hopf superalgebra  $ F\big[\big[{GL}_{\,n}^{\,p}\big]\big] $  inherits in addition a Poisson bracket  $ \{\,\ ,\ \} $   --- turning  $ {GL}_{\,n}^{\,p} $  into a ``formal Poisson group'' ---   which is different from the bracket  $ {\{\,\ ,\ \}}_\Phi $ inherited from the ``multiparametric'' quantisation  $ \FhPhiglnp \, $.  For instance, computations give
  $$  \big\{ x_{i{}j} \, , x_{i{}k} \big\} \, = \, {(-1)}^{p_{i{}j} \, p_{i{}k}} \, {(-1)}^{p(i)}  \quad  ,  \qquad  {\big\{ x_{i{}j} \, , x_{i{}k} \big\}}_\Phi \, = \, {(-1)}^{p_{i{}j} \, p_{i{}k}} \big( {(-1)}^{p(i)} - \phi_{j{}k} \big)  $$
 for all  $ \, i \, , j \, , k \in I_n \, $  with  $ \, j < k \, $.  In a nutshell, this proves that the QFSHA's  $ \Fhglnp $  and  $ \FhPhiglnp $   --- for different  $ \phi $'s  ---   are all quantisations of one and the same formal group  (over  $ {GL}_{\,n}^{\,p} \, $),  but they endow this shared semiclassical limit with different Poisson structures, depending on the different multiparameters  $ \Phi \, $.
\end{free text}

\vskip9pt

   We finish this subsection with a PBW-like theorem, that is the multiparametric counterpart of  Theorem \ref{thm: PBW x Fhglnp}  above (which deals with the uniparametric case): indeed, the roof follows exactly the same line of reasoning, with the same arguments, so we do not need to replicate it again.

\vskip11pt

\begin{theorem}  \label{thm: PBW x FhPhiglnp}
 \textsl{(PBW Theorem for  $ \FhPhiglnp \, $)}
 Let us fix any total order in the set  $ \, \big\{\, x_{i{}j} \,\big|\, i , j \in I_n \,\big\} \, $  of generators of  $ \FhPhiglnp \, $.  Then the set of all ordered monomials in the  $ x_{i{}j} $'s,  namely
  $$  \bigg\{\, {\textstyle \mathop{\overrightarrow{\prod}}\limits_{i, j \in I_n}} \hskip-3pt x_{i{}j}^{\,e_{i{}j}} \,\bigg|\, e_{i{}j} \in \NN \, , \, \forall \; i \, , j \in \! I_n \, , \; e_{i{}j} \leq 1 \text{\;\ if\;\ }  p_{i,j} = \one \,\bigg\}  $$
 is a\/  $ \kh $--basis  (in topological) of the\/  $ \kh $--module  $ \FhPhiglnp \, $.   \qed
\end{theorem}

\bigskip
 \vskip13pt

\section{Polynomial versions of quantum general linear supergroups}  \label{sec: polyn-QFSA's}
 \vskip7pt
   In this section we introduce ``polynomial versions'', so to speak, of the QFSHSA's considered in  \S \ref{sec: mp-QFSHSA's},  both in the uniparametric and in the multiparametric setup.

\vskip11pt

\subsection{Yamane's polynomial QUESA's for the general linear supergroup}  \label{subsec: polyn-QUESA x Yamane}
 Just like in the non-super framework, from Yamane's QUESA  $ \Uhglnp $  as in  \S \ref{subsec: Yamane's QUESA}  one can pull out a ``polynomial version'', i.e.\ a Jimbo-Lusztig version of it: the main difference is that, while  $ \Uhglnp $  is a  \textsl{formal\/}  Hopf superalgebra   --- over  $ \kh $  ---   its polynomial version is instead a usual, non-formal Hopf superalgebra.
                                                                        \par
   We retain terminology and notation as in  \S \ref{subsec: Yamane's QUESA}.  In addition, we let  $ \kqqm $  be the  $ \k $--algebra  Laurent polynomials in the indeterminate  $ q $  with coefficients in  $ \k \, $:  \,its field of fractions is the field  $ \k(q) $  of rational functions in  $ q $  with coefficients in  $ \k \, $.

\vskip9pt

\begin{definition}  \label{def: ration-QUESA}
 We define  \textsl{rational}  \textit{QUESA over  $ \, \liegl_{\,n}^{\,p} \, $},  denoted by  $ \, \Uqglnp \, $,  the associative, unital  $ \k(q) $--superalgebra  with generators
  $$  E_1 \, , \, E_2 \, , \, \dots \, , \, E_{n-1} \, , \, L_1^{\pm 1} \, , \, L_2^{\pm 1} \, , \, \dots \, , \, L_{n-1}^{\pm 1} \, , \, L_n^{\pm 1} \, ,  \, F_1 \, , \, F_2 \, , \, \dots \, , \, F_{n-1}  $$
 having parity  $ \; |E_r| := p_{r,r+1} \; $,  $ \; \big|L_s^{\pm 1}\big| := 0 \; $,  $ \; |F_r| := p_{r+1,r} \; $   --- for all  $ \, r \in I_{n-1} \, $  and  $ \, s \in I_n \, $  ---   and relations (for all  $ \, i \, , j \in I_{n-1} := \{1, \dots , n-1\} \, $,  $ \, k \, , \ell \in I_n := \{1, \dots , n\} \, $)
\begin{equation}  \label{eq: pol-UqgR1}
  \begin{gathered}
    \big[ L_k^{\pm 1} \, , L_\ell^{\pm 1} \,\big]  \; = \;  0  \quad ,
    \qquad  L_k^{\pm 1} \, L_k^{\mp 1}  \; = \;  1
  \end{gathered}
\end{equation}
\begin{equation}  \label{eq: pol-UqgR2}
  \begin{gathered}
    L_k^{\,\pm 1} F_j \, L_k^{\,\mp 1}  \; = \;  q^{\pm(\delta_{k,j+1} - \delta_{k,j})} \, F_j  \quad ,   \qquad  L_k^{\,\pm 1} E_j \, L_k^{\,\mp 1}  \; = \;  q^{\pm(\delta_{k,j} - \delta_{k,j+1})} \, E_j \, = \, 0  \\
  \end{gathered}
\end{equation}
\begin{equation}  \label{eq: pol-UqgR3}
  \begin{gathered}
    [E_i \, , F_j] \, - \, \delta_{i,j} \, \frac{\; K_i^{+1} - K_i^{-1} \;}{\; q_i^{+1} - q_i^{-1} \;}  \; = \;  0
  \end{gathered}
\end{equation}
\begin{equation}  \label{eq: pol-UqgR4}
  \begin{gathered}
    \quad   E_i^2 \, = \, 0 \; ,  \quad  F_i^2 \, = \, 0   \quad \qquad  \text{if \ }  p_{i,\,i+1} = \one = p_{i+1,\,i}
  \end{gathered}
\end{equation}
\begin{equation}  \label{eq: pol-UqgR5}
  \begin{gathered}
    [E_i \, , E_j] \, = \, 0  \quad ,   \qquad   [F_i \, , F_j] \, = \, 0  \qquad  \text{if \ }  |i-j| > 1
  \end{gathered}
\end{equation}
\begin{equation}  \label{eq: pol-UqgR6}
  \begin{gathered}
    \hskip-5pt  E_i^2 \, E_j \, - \, \big( q + q^{-1} \big) \, E_i \, E_j \, E_i \, + \, E_j \, E_i^2 \, = \, 0   \qquad  \text{if \ }  p_{i,\,i+1} = \zero  \text{\;\ \ and \;\ }  |i-j| = 1  \\
    \hskip-0pt  F_i^2 \, F_j - \big( q + q^{-1} \big) \, F_i \, F_j \, F_i \, + \, F_j \, F_i^2 \, = \, 0   \;\qquad  \text{if \ }  p_{i+1,\,i} = \zero  \text{\;\ \ and \;\ }  |i-j| = 1
  \end{gathered}
\end{equation}
\begin{equation}  \label{eq: pol-UqgR7}
  \begin{gathered}
    \big[ \big[ {[E_i,E_j]}_{q_j} , E_k \big]_{q_{j+1}} , E_j \big] \, = \, 0   \qquad  \text{\; if \ }  p_{j,\,j+1} = \one  \text{\;\ \ and \;\ }  k = j+1 = i+2  \\
    \big[ \big[ {[F_i,F_j]}_{q_j} , F_k \big]_{q_{j+1}} , E_j \big] \, = \, 0   \qquad  \text{\; if \ }  p_{j+1,\,j} = \one  \text{\;\ \ and \;\ }  k = j+1 = i+2
  \end{gathered}
\end{equation}
where hereafter we use such notation as
  $$  \displaylines{
   K_i^{\pm 1}  \, := \,  {\Big( L_i^{{(-1)}^{p(i)}} L_{i+1}^{{(-1)}^{p(i+1)}} \Big)}^{\pm 1}  \cr
   \hfill   {[A\,,B]}_c  \, := \,  A B - c \, {(-1)}^{|A|\,|B|} B A  \quad ,   \qquad   [A\,,B]  \, := \,  {[A\,,B]}_1   \hfill   \diamond  }  $$
\end{definition}

\vskip9pt

   The key fact about the superalgebra  $ \Uqglnp $  is the following easy result:

\vskip11pt

\begin{theorem}  \label{thm: Hopf-struct x Uqglnp}
 There exists a unique structure of Hopf  $ \, \kq $--superalgebra  on the superalgebra  $ \Uqglnp $  which is described by the following formulas on generators (for all  $ \, i \in I_{n-1} \, , \; k \in I_n \, $):
  $$  \displaylines{
   \Delta(E_i)  \; = \;  E_i \otimes 1 + K_i^{+1} \otimes E_i \; ,  \qquad  S(E_i)  \; = \;  -K_i^{-1} \, E_i \; ,  \qquad  \epsilon(E_i) \; = \; 0  \cr
   \hskip17pt  \Delta\big(L_k^{\pm 1}\big)  \; = \;  L_k^{\pm 1} \otimes L_k^{\pm 1}  \;\; ,  \hskip24pt \qquad S\big(L_k^{\pm 1}\big) \; = \; L_k^{\mp 1}  \;\; ,   \hskip13pt \quad  \epsilon\big(L_k^{\pm 1}\big)  \; = \;  1  \cr
   \Delta(F_i)  \; = \;  F_i\otimes K_i^{-1} + 1 \otimes F_i \; ,  \;\qquad  S(F_i) \; = \; - F_i \, K_i^{+1} \; ,  \;\qquad  \epsilon(F_i) \; = \; 0  }  $$
\end{theorem}

\begin{proof}
 Consider the field  $ \k(\!(\hbar)\!) $  of Laurent series in  $ \hbar \, $,  which is the quotient field of  $ \kh \, $.  Take in  $ \kh $  the element  $ \, q := \exp(\hbar) \, $:  \,then the subfield  $ \, \k(q) \, $  of  $ \k(\!(\hbar)\!) $  is an isomorphic copy of the field of rational functions in one indeterminate with coefficients in $ \k \, $.  Let  $ \, \k(\!(\hbar)\!) \otimes_\kh \Uhglnp \, $  be the scalar extension   --- from  $ \kh $  to  $ \k(\!(\hbar)\!) $  ---   of  $ \Uhglnp \, $:  inside it, we consider the unital  $ \k(q) $--subsuperalgebra  generated by the set  $ \, \big\{\, E_i \, , \, L_k^{\pm 1} := \exp(\pm \hbar \, \varGamma_k) \, , \, F_i \;\big|\; i \in I_{n-1} \, , \, k \in I_n \,\big\} \, $,  \,denoted by  $ U_q^{\,\prime} \, $.  Then the presentation of  $ \Uhglnp $  by generators and relations induces at once a presentation for  $ U_q^{\,\prime} $  with the given generators and with relations from  \eqref{eq: pol-UqgR1}  through  \eqref{eq: pol-UqgR7}.  Therefore,  $ U_q^{\,\prime} $  is isomorphic to  $ \Uqglnp $  as a  $ \kq $--superalgebra.  Finally, it is immediate to see that the coproduct, counit and antipode maps for  $ \Uhglnp $  automatically induce similar, well-defined maps for  $ U_q^{\,\prime} \, $,  that are explicitly described on generators by the formulas in the statement above.  In particular, the coproduct map takes values in the standard, algebraic tensor product  $ \, U_q^{\,\prime} \otimes_{\k(q)} U_q^{\,\prime} \, $:  \,therefore  $ U_q^{\,\prime} $   --- hence  $ \Uqglnp $  as well ---   is a (standard) Hopf algebra over  $ \k(q) \, $,  as all required conditions are automatically fulfilled as they were by the original structure maps in  $ \Uhglnp \, $.
\end{proof}

\vskip5pt

   For later use, we need also an ``integral'' version of  $ \Uqglnp \, $,  defined as follows:

\vskip9pt

\begin{definition}  \label{def: integr-QUESA}
 Inside  $ \Uqglnp \, $,  consider the elements
 \vskip3pt
   \centerline{ $ \displaystyle{ \varTheta_i  \,\; := \;\,  {{\, L_i^{+{(-1)}^{p(i)}} L_{i+1}^{-{(-1)}^{p(i+1)}} \! - \, L_i^{-{(-1)}^{p(i)}} L_{i+1}^{+{(-1)}^{p(i+1)}} \,} \over {\, q_i^{+1} - q_i^{-1} \,}}   \qquad \qquad  \forall \;\; i \in I_{n-1} } $ }
 \vskip3pt
   We define  \textsl{integral}  \textit{QUESA over  $ \, \liegl_{\,n}^{\,p} \, $},  denoted by  $ \, \Uqintglnp \, $,  the unital subsuperalgebra  over  $ \kqqm $  generated by  $ \; {\big\{\, E_i \, , L_k^{\pm 1} , \varTheta_i \, , F_i \;\big|\; i \in I_{n-1} \, , \, k \in I_n \,\big\}} \; $.
\end{definition}

\vskip7pt

   The following is an immediate consequence of  Theorem \ref{thm: Hopf-struct x Uqglnp}  above:

\vskip11pt

\begin{prop}  \label{prop: Uqintglnp-is-Hopf}
 $ \Uqintglnp $  is a  \textsl{Hopf}  subsuperalgebra   --- over  $ \kqqm $  ---   of the Hopf $ \, \k(q) $--superalgebra  $ \Uqglnp \, $.   \qed
\end{prop}

\vskip13pt

\subsection{Polynomial QFSA's for the general linear supermonoid}  \label{subsec: Mp-polyn-QFSA's x super-MLn}
 Let us consider the unital  $ \, \k $--superalgebra  $ \, \MLnp := \textit{Mat}_{\,n \times n}(\k) \, $  of all square matrices of size  $ n $  with entries in  $ \k \, $,  with  $ \ZZ_2 $--grading given by  $ \, |e_{i{}j}| := p_{i,\,j} = p(i) + p(j) \, $  for all  $ \, i \, , j \in I_n \, $,  where  $ e_{i{}j} $  denotes the  $ (i\,,j) $--th  elementary matrix.  This is of course an algebraic supermonoid (=supersemigroup with unit)   --- with respect to the row-by-column multiplication ---   whose  $ \k $--superalgebra  of ``regular functions'' is  $ \, F\big[\MLnp\big] = \k\big[ \{\, \overline{x}_{i,\,j} \,\big|\, i \, , j \! \in \! I_n \} \big] \, $,  the  $ \k $--superalgebra  of (super)polynomials in the  $ \overline{x}_{i{}j} $'s  with  $ \ZZ_2 $--grading  induced by  $ \, \big|\overline{x}_{i{}j}\big| := p_{i,\,j} = p(i) + p(j) \, $.
 \vskip5pt
   Our next result provides a suitable ``quantisation'' of  $ F\big[\MLnp\big] \, $,  as follows:

\vskip11pt

\begin{theorem}  \label{thm: pres-Fqmln}
 There exists a unique  $ \kqqm $--superbialgebra  $ \Fqintmlnp $  enjoying the following properties:
 \vskip3pt
   (a)\;  it is generated by  $ \, \big\{\, x_{i{}j} \,\big|\, i \, , j \in\! I_n \big\} \, $,
   with parity
 $ \, |x_{i{}j}| = p_{i{}j} := p(i)+p(j) \, $,  \,subject to the following relations, where  $ \, q_s := q^{{(-1)}^{p(s)}} \, $:
  $$  \displaylines{
   \hskip35pt   x_{i{}j}^2  \; = \;  0   \hskip185pt \qquad  \text{\ for \ \ }  p_{i{}j} = \one  \cr
   x_{i{}j} \, x_{i{}k}  \; = \;  {(-1)}^{p_{i{}j} p_{i{}k}} q_i \, x_{i{}k} \, x_{i{}j}  \hskip125pt \qquad  \text{for \ \ } j < k  \cr
   x_{i{}j} \, x_{h{}j}  \; = \;  {(-1)}^{p_{i{}j} p_{h{}j}} q_j \, x_{h{}j} \, x_{i{}j}  \hskip125pt \qquad  \text{for \ \ } i<h  \cr
   \qquad \hskip19pt   x_{i{}j} \, x_{h{}k}  \; = \;  {(-1)}^{p_{i{}j} p_{h{}k}} x_{k{}h} \, x_{i{}j}  \hskip119pt \qquad  \text{for \ \ }  i < h \, , \, j > k   \hskip61pt  \cr
   \hskip3pt   x_{i{}j} \, x_{h{}k}  \; = \;  {(-1)}^{p_{i{}j} p_{h{}k}} x_{h{}k} \, x_{i{}j} \, + {(-1)}^{{p_{i{}j} p_{i{}k}}} \big( q_i^{+1} - q_i^{-1} \big) \, x_{i{}k} \, x_{h{}j}   \hskip11pt  \text{\;\;\ \ for \ \ } i < h \, , \, j < k  }  $$
 \vskip3pt
   (b)\;  its coproduct and counit are defined, in terms of the above presentation, by
\begin{equation}  \label{eq: coprod+counit for x(i,j) - TER}
  \Delta(x_{ij}) = {\textstyle \sum\limits_{a=1}^n} \, {(-1)}^{p_{i{}a} p_{a{}j}} x_{i{}a} \otimes x_{a{}j}   \quad ,  \qquad
      \epsilon(x_{i{}j}) = \delta_{i{}j}
   \qquad \quad  \forall \;\; i \, , j \in I_n   \quad
\end{equation}
 \vskip3pt
   (c)\;  it admits a non-degenerate pairing of\/  $ \kqqm $--superbialgebras
  $$  \langle\,\ ,\ \rangle : \Fqintmlnp \times \Uqintglnp \relbar\joinrel\relbar\joinrel\relbar\joinrel\longrightarrow \kqqm  $$
 given on generators   --- $ \, \forall \, i \, , j \, , k \in I_n \, $,  $ \, t \in I_{n-1} \, $  ---   by
\begin{equation}  \label{eq: pairing-formulas - TER}
  \big\langle x_{i{}j} \, , E_t \big\rangle  \; = \;  \delta_{i,t} \, \delta_{j,t+1} \;\; ,  \quad
\big\langle x_{i{}j} \, , F_t \big\rangle  \; = \,  \delta_{i,t+1} \, \delta_{j,t} \;\; ,  \quad  \big\langle x_{i{}j} \, , \varGamma_k \big\rangle  \; = \;  \delta_{i,k} \, \delta_{j,k}
\end{equation}
 \vskip3pt
   (d)\;  it admits an embedding  $ \, \Fqintmlnp \!\lhook\joinrel\relbar\joinrel\relbar\joinrel\relbar\joinrel\longrightarrow\! \Fhglnp \, $  of Hopf\/  $ \kqqm $--superalgebras   --- where on the right-hand side we consider the\/  $ \kqqm $--module  structure given by restriction from\/  $ \kh $  through  $ \, \kqqm \!\lhook\joinrel\relbar\joinrel\relbar\joinrel\longrightarrow \kh \, $  given by  $ \, P\big(q\,,q^{-1}\big) \mapsto P\big(e^{+\hbar}\,,e^{-\hbar}\big) \, $  ---  which is uniquely given by  $ \, x_{i{}j} \mapsto x_{i{}j} \, $  for all  $ \, i \, , j \in I_n \, $.
 \vskip3pt
   (e)\;  Let  $ \, \Fqmlnp \! := \kq \otimes_{\k[q,q^{-1}]} \Fqintmlnp \, $  be the scalar extension
 of  $ \Fqintmlnp $.  Then the like of claims  \textit{(a)}  through  \textit{(c)} hold true as well, with  $ \, \kq $  replacing  $ \, \kqqm \, $,  $ \Fqmlnp $  replacing  $ \Fqintmlnp \, $,  and  $ \Uqglnp $  replacing  $ \Uqintglnp \, $.
\end{theorem}

\begin{proof}
 The proof mimics closely that of  Theorem \ref{thm: Hopf-struct x Uqglnp},  and then deduces the result out of  Theorem \ref{thm: pres-Fhg}.  First of all, we define  $ \Fqintmlnp $  by generators and relations as prescribed in claim  \textit{(a)}.
                                                                                         \par
   Second, like in the proof of  Theorem \ref{thm: Hopf-struct x Uqglnp}  we observe that  $ \, e^\hbar := \exp(\hbar) \, $  generates in  $ \k(\!(\hbar)\!) $  a field extension of  $ \k $ isomorphic to  $ \k(q) \, $,  hence we identify that extension with  $ \kq $  and  $ e^\hbar $  with  $ q \, $.  Similarly, the  $ \k $--subalgebra  of  $ \kh $  generated by  $ e^{+\hbar} $  and  $ e^{-\hbar} $  identifies with the  $ \k $--algebra  $ \kqqm \, $.  Now, inside  $ \Fhglnp $  we consider the unital  $ \kqqm $--subalgebra  $ F_q^{\,\dagger} $  generated by the set  $ \, \big\{\, x_{i{}j} \;\big|\; i \, , j \in I_n \,\big\} \, $  of the built-in generators of  $ \Fhglnp \, $.  Then the presentation of  $ \Fhglnp $  by generators and relations induces at once a presentation for  $ F_q^{\,\dagger} $,  with the given generators and with relations as in claim  \textit{(a)}.  It follows that  $ F_q^{\,\dagger} $  is isomorphic to  $ \Fqintmlnp $  as a  $ \kqqm $--superalgebra.  Finally, it is immediate to see that  $ F_q^{\,\dagger} $  is actually a sub-superbialgebra   --- over  $ \kqqm $  ---   of  $ \Fhglnp \, $:  more precisely, as such it is a  \textsl{standard\/}  (i.e., non-topological) superbialgebra, as the coproduct map takes values in the standard, algebraic tensor product  $ \, F_q^{\,\dagger} \otimes_{\k[q,\,q^{-1}]} F_q^{\,\dagger} \, $.  Moreover, its coproduct and counit are described on generators.  All this proves that an object  $ \Fqintmlnp $  as claimed which fulfils claims  \textit{(a)},  \textit{(b)},  \textit{(c)\/}  and  \textit{(d)\/}  does exist.
 \vskip3pt
   Eventually, claim  \textit{(d)\/}  follows trivially from the previous ones.
\end{proof}

\vskip7pt

\begin{definition}  \label{def: polyn-QFSA's x MLnp}
 We call  $ \Fqintmlnp \, $,  resp.\  $ \Fqmlnp \, $,  \textit{the  \textsl{integral},  resp.\  \textit{rational},  quantum function superalgebra (or ``QFSA'' in short) over  $ \textit{ML}_{\,n}^{\,p} \, $}.   \hfill  $ \diamond $
\end{definition}

\vskip9pt

   We have also a PBW-like theorem (an ``integral version'' of  Theorem \ref{thm: PBW x Fhglnp})  that provides a  $ \kqqm$--basis  of ordered, truncated monomials in the generators: it is the ``super counterpart'' of a well-known similar result for the quantum function algebra (\`a la jimbo-Lusztig) over  $ ML(n) \, $,  see for instance  \cite{Tn},  Lemma A.2.

\vskip12pt

\begin{theorem}  \label{thm: PBW x Fqmlnp}
 \textsl{(PBW Theorem for  $ \Fqintmlnp $  and  $ \Fqmlnp \, $)}
 Let us fix any total order in the set  $ \, \big\{\, x_{i{}j} \,\big|\, i , j \in I_n \,\big\} \, $  of generators of  $ \Fqintmlnp \, $.  Then the set
  $$  \mathbb{B}  \; := \;  \bigg\{\, {\textstyle \mathop{\overrightarrow{\prod}}\limits_{i, j \in I_n}} \hskip-3pt x_{i{}j}^{\,e_{i{}j}} \;\bigg|\; e_{i{}j} \in \NN \, , \, \forall \; i \, , j \in \! I_n \, , \; e_{i{}j} \leq 1 \text{\;\ if\;\ }  p_{i,j} = \one \,\bigg\}  $$
 of all truncated ordered monomials in the  $ x_{i{}j} $'s  is a\/  $ \kqqm $--basis  of the\/  $ \kqqm $--module  $ \Fqmlnp \, $.  In particular, then, the  $ \kqqm $--module  $ \Fqintmlnp $  is free.
                                                                    \par
   A parallel result holds for  $ \, \Fqmlnp := \kq \otimes_{\k[q\,,\,q^{-1}]} \Fqintmlnp \, $  as well, namely\/  $ \mathbb{B} $  is a  $ \kq $--basis  of  $ \Fqmlnp \, $.
\end{theorem}

\begin{proof}
 It is enough to prove the claim about  $ \Fqintmlnp \, $,  \,so we focus on that.
                                                                         \par
   To begin with, note that  $ \Fqintmlnp $  is clearly spanned over  $ \kqqm $  by the set of all (possibly unordered) monomials in the  $ x_{i{}j} $'s.  Moreover, giving degree 1 to each generator  $ x_{i{}j} $  defines an  $ \NN $--grading  on  $ \Fqintmlnp $  (in a nutshell, because the relations among the generators are ``homogeneous''), with  $ \, \Fqintmlnp = \! \mathop{\oplus}\limits_{m \in \NN} \! \F_m \, $  where each  $ \F_m $  is the  $ \kqqm $--span  of all monomials (in the  $ x_{i{}j} $'s)  of degree  $ m \, $.  Thus, it is enough for us to prove that each direct summand  $ \, \F_m \, (\, m \in \NN \,) \, $  admits as  $ \kqqm $--basis  the set  $ \mathbb{B}_m $  of all  \textsl{truncated, ordered monomial of degree  $ m \, $}.
 \vskip5pt
   First of all, we prove that  $ \mathbb{B}_m $  does generate all of  $ \F_m $  as a  $ \kqqm $--module.
                                                                       \par
   Take any (possibly unordered) monomial in the  $ x_{i{}j} $'s  of degree   $ m \, (\in \NN) \, $,  say  $ \; \underline{x} \, := \, x_{i_1,j_i} \, x_{i_2,j_2} \cdots \, x_{i_m,j_m} \; $:  \,we  define its  \textit{weight\/}  as
  $$  w(\,\underline{x}\,) \; := \; \big(\, m \, , d_{1,1} \, , d_{1,2} \, ,
\dots, d_{1,n} \, , \dots, d_{2,n} \, , d_{3,1}
\, , \dots, d_{n-1,n} \, , d_{n,1} \, , \dots, d_{n,n} \, , i(\,\underline{x}\,) \big)  $$
 where
  $$  \displaylines{
   d_{i,j}  \; := \;  \big| \big\{\, s \! \in \! \{1,\dots,k\} \,\big|\, (i_s,j_s) = (i,j) \big\} \big| \, = \textsl{number of occurrences of  $ x_{i{}j} $  in  $ \underline{x} \; $}  \cr
   \text{and}  \hfill   i(\,\underline{x}\,)  \; := \;  \textsl{number of  \textit{inversions}  of the order occurring in  $ \underline{x} \; $}   \hfill  }  $$
 Then  $ \, w(\,\underline{x}\,) \in \NN^{\,n^2 + 2} \, $,  \,and we consider  $ \NN^{\,n^2 + 2} $  as a totally ordered set with respect to the (total) lexicographic order  $ \leq_{lex} \, $.  Now consider again the defining relations of  $ \Fqintmlnp \, $,  namely
  $$  \displaylines{
   \hskip35pt   x_{i{}j}^2  \; = \;  0   \hskip185pt \qquad  \text{\ for \ \ }  p_{i{}j} = \one  \cr
   x_{i{}j} \, x_{i{}k}  \; = \;  {(-1)}^{p_{i{}j} p_{i{}k}} q_i \, x_{i{}k} \, x_{i{}j}  \hskip125pt \qquad  \text{for \ \ } j < k  \cr
   x_{i{}j} \, x_{h{}j}  \; = \;  {(-1)}^{p_{i{}j} p_{h{}j}} q_j \, x_{h{}j} \, x_{i{}j}  \hskip125pt \qquad  \text{for \ \ } i<h  \cr
   \qquad \hskip19pt   x_{i{}j} \, x_{h{}k}  \; = \;  {(-1)}^{p_{i{}j} p_{h{}k}} x_{k{}h} \, x_{i{}j}  \hskip119pt \qquad  \text{for \ \ }  i < h \, , \, j > k   \hskip61pt  \cr
   \hskip3pt   x_{i{}j} \, x_{h{}k}  \; = \;  {(-1)}^{p_{i{}j} p_{h{}k}} x_{h{}k} \, x_{i{}j} \, + {(-1)}^{{p_{i{}j} p_{i{}k}}} \big(\, q_i^{+1} - q_i^{-1} \,\big) \, x_{i{}k} \, x_{h{}j}   \hskip7pt  \text{\;\;\ \ for \ \ } i < h \, , \, j < k  }  $$
 as in  Theorem \ref{thm: pres-Fqmln}\textit{(a)}.  We will see now how we can use them to re-write any monomial  $ \underline{x} $  as above into a  $ \kqqm $--linear  combination of similar monomials having lower weight in the totally ordered set  $ \, \big(\, \NN^{\,n^2 + 2} , \, \leq_{lex} \!\big) \, $.
                                                              \par
   The very first relation just tells us that in any monomial  $ \underline{x} $  as above we can always, roughly speaking, ``reduce modulo 2 every exponent  $ e_{i{}j} $  with  $ \, p_{i,j} = \one \, $''.  The outcome of any such ``reduction step'' is a new monomial  $ \underline{x}' $  whose total degree  $ m' $  is strictly smaller than the total degree  $ m $  of  $ \underline{x} \, $:  \,in particular, it is also  $ \, w\big(\,\underline{x}'\big) \lneqq w(\,\underline{x}\,) \; $.
                                                              \par
   We use the other relations, instead, to ``decrease the number of inversions'' in  $ \, \underline{x} \, $.  Indeed, let  $ x_{i{}j} $  and  $ x_{h{}k} $  be two consecutive factors in  $ \, \underline{x} \, $  such that  $ \, x_{i{}j} \succneqq x_{h{}k} \, $,  where  $ \preceq $  is our fixed total order in the set of all generators  $ \, \big\{\, x_{r{}s} \,\big|\, r \, , s \in I_n \big\} \, $.
                                                              \par
   First assume that  $ \, h = i \, $  and  $ \, j < k \, $:  then using the second relation in the above list we replace the factor  ``$ \, x_{i{}j} \, x_{i{}k} \, $''  in  $ \, \underline{x} \, $  with the new factor  ``$ \, {(-1)}^{p_{i{}j} p_{i{}k}} q_i \, x_{i{}k} \, x_{i{}j} \, $''.  Then, pulling out the contribution  ``$ \, {(-1)}^{p_{i{}j} p_{i{}k}} q_i \, $''  we and up with  $ \; \underline{x} = {(-1)}^{p_{i{}j} p_{i{}k}} q_i \, \underline{x}' \; $  where  $ \, \underline{x}' \, $  is a new monomial with  $ \, w\big(\,\underline{x}'\big) \lneqq_{lex} w(\,\underline{x}\,) \, $,  \,because we have  $ \, i\big(\,\underline{x}'\big) \lneqq i(\,\underline{x}\,) \, $  in the last component of the two weights, whereas all other components are the same.  If instead we have  $ \, h = i \, $  but  $ \, j > k \, $,  then the same argument   --- but for reading the same relation the other way round ---   gives us  $ \; \underline{x} = {(-1)}^{p_{i{}j} p_{i{}k}} q_i^{-1} \, \underline{x}' \; $  where  $ \, \underline{x}' \, $  is a new monomial such that  $ \, i\big(\,\underline{x}'\big) \lneqq i(\,\underline{x}\,) \, $  and  $ \, w\big(\,\underline{x}'\big) \lneqq_{lex} w(\,\underline{x}\,) \, $,  just as before.
                                                              \par
   As a second step, assume that  $ \, k = j \, $:  then again, a perfectly similar argument, now using the third relation in the list, yields  $ \; \underline{x} = \pm \, q_j^{\pm 1} \underline{x}' \; $  for a suitable choice of signs with  $ \, i\big(\,\underline{x}'\big) \lneqq i(\,\underline{x}\,) \, $  and  $ \, w\big(\,\underline{x}'\big) \lneqq_{lex} w(\,\underline{x}\,) \, $  again.  The third step is for  $ \, i < h \, $  and  $ \, j > k \, $:  in this case one uses the fourth relation and finds  $ \; \underline{x} = \pm \, \underline{x}' \; $  with  $ \, i\big(\,\underline{x}'\big) \lneqq i(\,\underline{x}\,) \, $  and  $ \, w\big(\,\underline{x}'\big) \lneqq_{lex} w(\,\underline{x}\,) \, $.  Moreover, reading the same formula the other way round, one gets a parallel result in the case when  $ \, h < i \, $  and  $ \, k > j \, $  as well.
                                                              \par
   Finally, we are left with the case when  $ \, i < h \, $  and  $ \, j < k \, $  or  $ \, h < i \, $  and  $ \, k < j \, $;  we start with the first case.  Using the fifth (and last) relation above, we replace the factor  ``$ \, x_{i{}j} \, x_{h{}k} \, $''  in  $ \, \underline{x} \, $  with  ``$ \, {(-1)}^{{p_{i{}j} p_{i{}k}}} x_{h{}k} \, x_{i{}j} \, + {(-1)}^{{p_{i{}j} p_{i{}k}}} \big(\, q_i^{+1} - q_i^{-1} \,\big) \, x_{i{}k} \, x_{h{}j} \, $'':  \,then
  $$  \underline{x}  \; = \;  {(-1)}^{{p_{i{}j} p_{i{}k}}} \underline{x}' \, + \, {(-1)}^{{p_{i{}j} p_{i{}k}}} \big(\, q_i^{+1} - q_i^{-1} \,\big) \, \underline{x}''  $$
 where  $ \underline{x}' $  and  $ \underline{x}'' $  are monomials such that  $ \, i\big(\,\underline{x}'\big) \lneqq i(\,\underline{x}\,) \, $  and  $ \, w\big(\,\underline{x}'\big) \lneqq_{lex} w(\,\underline{x}\,) \, $  and also  $ \, w\big(\,\underline{x}''\big) \lneqq_{lex} w(\,\underline{x}\,) \, $,  \,just by construction.
 \vskip3pt
   Now, in all steps considered above, the initial monomial  $ \underline{x} $  is always re-written as a  $ \kqqm $--linear  combination of (one or two) monomials whose weight is strictly smaller than that of  $ \underline{x} \, $.  Therefore, by induction we can assume that those monomials do belong to the  $ \kqqm $--span  of  $ \mathbb{B}_m \, $,  and then the same holds for  $ \underline{x} $  as well, q.e.d.
 \vskip5pt
   We are still left to show that the set  $ \mathbb{B}_m $  is linearly independent over  $ \kqqm \, $.
%
%
                                                              \par
   First,  $ \Fqintmlnp \, $,  as a  $ \kqqm $--module,  is torsion free, due to the embedding  $ \, \Fqintmlnp \!\lhook\joinrel\relbar\joinrel\relbar\joinrel\longrightarrow\! \Fhglnp \, $,  \,see  Theorem \ref{thm: PBW x Fqmlnp};  then each  $ \F_m $  is torsion free too.
%
%
                                                              \par
   Now assume that  $ \; \sum_{b \in \mathbb{B}_m} c_b \, b \, = \, 0 \; $  is some  $ \kqqm $--linear  dependence relation among elements of  $ \mathbb{B}_m \, $,  with  $ \, c_b \not= 0 \, $  for some  $ \, b \in \mathbb{B}_m \, $.
 Since  $ \F_m $  is torsion free, we can assume that  $ \; \overline{c_b} := c_b \mod (q-1) \kqqm \, $  is such that  $ \, \overline{c_b} \not= 0 \, $  for some  $ \, b \in  \mathbb{B}_m \, $.  Then the relation  $ \; \sum_{b \in \mathbb{B}_m} c_b \, b \, = \, 0 \; $  in  $ \F_m \, \big( \subseteq \Fqintmlnp \big) \, $  yields  $ \; \sum_{\overline{c_b} \not= 0} \, \overline{c_b} \, \overline{b} \, = \, 0 \; $  in  $ F\big[M\!L_{\,n}^{\,p}\big] \, $,  where  $ \, \overline{b} \, $  denotes the coset of  $ b $  modulo  $ \, (q-1) \Fqintmlnp \, $:  \,but this is impossible, because the  $ \overline{b} $'s  are just the truncated monomials in the  $  \overline{x}_{i{}j} $'s,  which form a  $ \k $--basis  of  $ F\big[M\!L_{\,n}^{\,p}\big] \, $.  Thus  $ \mathbb{B}_m $  is  $ \kqqm $--linearly  independent, q.e.d.
\end{proof}

\vskip5pt

\begin{rmks}  \label{rmks: Fqmlnp-not-Hopf}
   \textit{(a)}\,  It is worth stressing that  $ \Fqmlnp $  is a superbialgebra,  but  \textsl{not a  \textit{Hopf}  superalgebra}.  To get such a Hopf superalgebra   ---
%
%
 denoted by  $ F_q\big[GL_{\,n}^{\,p}\big] $  ---   we should enlarge  $ \Fqmlnp $  by adding the inverse of a carefully chosen ``quantum Berezinian'', mimicking the non-super case where one introduces the inverse of a suitable ``quantum determinant'' (cf.\ for instance  \cite{Tn}).  See also  \cite{Ma},  where a ``quantum Berezinian'' is introduced (and later inverted), though through a different perspective.  However, filling in all details needed to correctly introduce such a ``quantum Berezinian'' is far from trivial, hence way beyond the scopes of this work.
 \vskip3pt
   \textit{(b)}\,  An alternative way to construct  $ \Fqmlnp $  (mimicking the non-super case: see  \cite{Tn})  goes as follows.  Instead of taking the full (linear) dual of  $ \Uhglnp $   --- thus finding the \textsl{formal\/}  Hopf superalgebra  $ \, \Fhglnp := {\Uhglnp}^* \, $  ---   one might start from the polynomial QUESA  $ \Uqglnp $  and take for it its  \textsl{finite dual\/}  (or  \textsl{``restricted dual''}, or  \textsl{``Hopf dual''})  $ \, {\Uqglnp}^\circ \, $,  \,just like in the non-super setup of affine algebraic groups.  However, in this case the standard finite dual is a bit too large: one instead has to consider such a ``finite dual'' with respect to a specifically chosen tensor subcategory of the category of all finite-dimensional  $ \Uqglnp $--modules,  namely that of all  $ \Uqglnp $--modules  ``of Type I'' (using Lusztig's terminology).  The ``finite dual with respect to this category'' then is, by definition, the  $ \kq $--subalgebra  of  $ {\Uhglnp}^* $  generated by all matrix coefficients relative to objects in that category.  Now, if we still restrict to the tensor (sub)category   --- within that of Type I  $ \Uqglnp $--modules  ---   generated by the standard (or ``natural'') vector representation  $ V_q $  of  $ \Uqglnp \, $,  one finds the  $ \kq $--subalgebra  of  $ {\Uhglnp}^* $  generated by the matrix coefficients from objects in this subcategory is actually generated by the matrix coefficients of  $ V_q $  alone.  Now, the latter matrix coefficients are just the  $ x_{i{}j} $'s,  hence one gets  $ \Fqmlnp \, $.
 \vskip3pt
   \textit{(c)}\,  On a different note, our choice of an ``integral form''  $ \Fqintmlnp $  of  $ \Fqmlnp $  is somewhat na\"\i{}ve.  Indeed, the best way to introduce such a notion would be to start by considering in  $ \Uqglnp $  the Lusztig-like ``restricted''  $ \kqqm $--integral  form  $ \widehat{U}_q\big(\liegl_{\,n}^{\,p}\big) \, $,  that is, by definition, the unital  $ \kqqm $--subalgebra  of  $ \Uqglnp $ generated by all  $ q $--divided  powers of the $ E_i $'s  and the  $ F_i $'s  and all  $ q $--binomial  coefficients in the  $ L_i^{\pm 1} $'s  (taking the latter as in  \cite{DL}, \S 3,  rather than as in Lusztig's works).  After proving that  $ \widehat{U}_q\big(\liegl_{\,n}^{\,p}\big) $  is in fact a Hopf subsuperalgebra   ---  over  $ \kqqm $  ---   of  $ \Uqglnp \, $,  one should take as  ``$ \kqqm $--integral  form'' of  $ \Fqmlnp $  the ``full  $ \kqqm $--dual  of  $ \widehat{U}_q\big(\liegl_{\,n}^{\,p}\big) $'',  i.e.\ the subset  $ \widehat{F}_q\big[M\!L_{\,n}^{\,p}\big] $  of  $ \Fqmlnp $  of all elements that are  $ \kqqm $--valued  on  $ \widehat{U}_q\big(\liegl_{\,n}^{\,p}\big) \, $.  Now  $ \, \widehat{U}_q\big(\liegl_{\,n}^{\,p}\big) \supseteq \Uqintglnp \, $  yields  $ \, \widehat{F}_q\big[M\!L_{\,n}^{\,p}\big] \subseteq \Fqintmlnp \, $;  \,in fact, we expect identity to hold, i.e.\  $ \, \widehat{F}_q\big[M\!L_{\,n}^{\,p}\big] = \Fqintmlnp \, $.
                                                 \par
   Now, although the plan sketched above is quite clear, fixing all technicalities and filling in all details so to get a mathematically sound outcome is still highly demanding, actually deserving an independent paper on its own.  Here we restricted ourselves to a less ambitious goal, leaving such a task to the interested reader.
\end{rmks}

\vskip13pt

\subsection{Multiparametric QFSA's for the general linear supermonoid}  \label{subsec: Mp-polyn-QFSA's x super-MLnp}
 Just like we did in the uniparametric case, we can introduce suitable (polynomial) QFSA's for  $ {ML}_{\,n}^{\,p} $  in the multiparametric case as well.  To this end, we fix a first indeterminate  $ q $  and a whole antysimmetric matrix  $ \, \bq := {\big(\,q_{\,i{}j}\big)}_{i \in I_n}^{j \in I_n} \, $  of further indeterminates (so that  $ \, q_{\,i{}j} + q_{\,j{}i} = 0 \, $  for all  $ \, i < j \, $):  \,we let  $ \, \kbqqm := \k\Big[ {\big\{ q^{\pm 1} , q_{\,i{}j}^{\pm 1} \big\}}_{i \in I_n}^{j \in I_n} \Big] \, $  be the associated  $ \k $--algebra  of Laurent polynomials, and  $ \, \kbq := \k\Big( {\big\{ q \, , q_{\,i{}j}^{\pm 1} \big\}}_{i \in I_n}^{j \in I_n} \Big) \, $  be the associated field of fractions.
                                                        \par
   The following result is the announced multiparametric counterpart of  Theorem \ref{thm: pres-Fqmln}:  its proof is entirely similar, hence it is left to the reader.

\vskip11pt

\begin{theorem}  \label{thm: pres-Fbqmln}  {\ }
 \vskip3pt
   (a)\;  There exists a unique  $ \kbqqm $--superbialgebra  $ \Fbqintmlnp $  such that:
 \vskip5pt
   \quad   (a.1)\;  it is generated by  $ \, \big\{\, x_{i{}j} \,\big|\, i \, , j \in\! I_n \big\} \, $,
   with parity
 $ \, |x_{i{}j}| = p_{i{}j} := p(i)+p(j) \, $,  \,subject to the following relations, where  $ \, q_s := q^{{(-1)}^{p(s)}} \, $:
  $$  \displaylines{
   \hskip51pt   x_{i{}j}^2  \; = \;  0   \hskip173pt  \big(\; p_{i{}j} = \one \,\big)  \cr
   x_{i{}j} \, x_{i{}k}  \; = \;  {(-1)}^{p_{i{}j} p_{i{}k}} q_{\,i} \, q_{j,k}^{\;-1} \, x_{i{}k} \, x_{i{}j}  \hskip105pt  \big(\; j < k \,\big)  \cr
   x_{i{}j} \, x_{h{}j}  \; = \;  {(-1)}^{p_{i{}j} p_{h{}j}} q_j \, q_{\,i,h}^{\;+1} \, x_{h{}j} \, x_{i{}j}  \hskip105pt  \big(\, i<h \,\big)  \cr
   \hskip17pt   x_{i{}j} \, x_{h{}k}  \; = \;  {(-1)}^{p_{i{}j} p_{h{}k}\,} q_{\,i,h}^{\;+1} \, q_{j,k}^{\;-1} \, x_{k{}h} \, x_{i{}j}  \hskip81pt  \big(\, i<h \, , \; j>k \,\big)  \cr
   \hskip7pt   x_{i{}j} \, x_{h{}k}  \; = \;  {(-1)}^{p_{i{}j} p_{h{}k}\,} q_{\,i,h}^{\;+1} \, q_{j,k}^{\;-1} \, x_{h{}k} \, x_{i{}j} \, + {(-1)}^{{p_{i{}j} p_{i{}k}}} \big( q_{\,i}^{+1} - q_{\,i}^{-1} \big) \, q_{j,k}^{\;-1} \, x_{i{}k} \, x_{h{}j}  \hskip17pt  \bigg(\; {{i < h} \atop {j < k}}  \,\bigg)  }  $$
 \vskip3pt
   \quad   (a.2)\;  its coproduct and counit maps are defined, in terms of generators, by
\begin{equation}  \label{eq: coprod+counit for x(i,j) - QUATER}
  \Delta(x_{ij}) = {\textstyle \sum\limits_{a=1}^n} \, {(-1)}^{p_{i{}a} p_{a{}j}} x_{i{}a} \otimes x_{a{}j}   \quad ,  \qquad
      \epsilon(x_{i{}j}) = \delta_{i{}j}
   \qquad \quad  \forall \;\; i \, , j \in I_n   \quad
\end{equation}
 \vskip1pt
   (b)\;  Let  $ \, \Fbqmlnp := \kbq \otimes_{\k[q^{\pm 1},\,\bq^{\pm 1}]} \Fbqintmlnp \, $  be the scalar extension of  $ \Fbqintmlnp \, $.  Then the like of claims  \textit{(a)}  through  \textit{(c)} hold true as well, with  $ \, \kbq $  replacing  $ \, \kbqqm \, $  and  $ \Fbqmlnp $  replacing  $ \Fbqintmlnp \, $.
\end{theorem}

\vskip7pt

\begin{definition}  \label{def: Mp-polyn-QFSA's x MLnp}
 We call  $ \Fbqintmlnp \, $,  resp.\  $ \Fbqmlnp \, $,  \textit{the  \textsl{integral},  resp.\  \textit{rational},  multiparametric quantum function superalgebra over  $ \textit{ML}_{\,n}^{\,p} \, $}.   \hfill  $ \diamond $
\end{definition}

\vskip9pt

\begin{rmks}  \label{rmks: QFSA's not-Hopf}
    \textit{(a)}\;  Choose a field extension  $ \, \k' \, $  of  $ \k \, $,  and an antisymmetric matrix  $ \, \Phi = {\big( \phi_{t,\ell} \big)}_{t=1,\dots,n;}^{\ell=1,\dots,n;} \in \mathfrak{so}_n\big(\k'[[\hbar]]\big) \, $  such that the set  $ \, \big\{\, e^{\hbar \, \phi_{t,\ell}} := \exp(\hbar \, \phi_{t,\ell}) \,\big|\, t < \ell \,\big\} \cup \big\{\, e^\hbar := \exp(\hbar) \,\big\} \, $  is algebraically independent over  $ \k \, $.  Let  $ \FhPhiglnp $  be the corresponding multiparametric QFSHA over  $ \k'[[\hbar]] \, $.  Then, by its very construction,  $ \Fbqintmlnp $  comes equipped with a canonical embedding
\begin{equation*}  \label{eq: embed-Fbqint-FhPhi}
 \Fbqintmlnp \,\lhook\joinrel\relbar\joinrel\relbar\joinrel\longrightarrow\, \FhPhiglnp \; ,  \qquad  x_{i{}j} \mapsto  x_{i{}j}  \quad \big(\; \forall \;\; i \, , j \in I_n \,\big)
\end{equation*}
 which extends the embedding  $ \; \kbqqm \lhook\joinrel\relbar\joinrel\relbar\joinrel\longrightarrow \k'[[\hbar]] \; $  of their respective ground rings given by  $ \; c \mapsto c \; (\, \forall \; c \in \k \,) \, $,  $ \, q^{\pm 1} \mapsto e^{\pm \hbar} \, $,  $ \; q_{i{}j}^{\pm 1} \mapsto e^{\pm \hbar \, \phi_{i,j} / 2} \; \big(\, \forall \; i \, , j \in I_n \big) \; $.
                                                                                   \par
   Now, through this embedding, one sees that  $ \Fbqintmlnp $  also admits a canonical, non-degenerate pairing which is the multiparametric counterpart of that in  Theorem \ref{thm: pres-Fqmln}\textit{(c)},  and is described in fact by the very same formulas.  Unfortunately, the definition of the ``multiparametric counterpart of  $ \Uqintglnp $''  is somewhat tricky, so we do not delve into details any more.  Similar remarks apply to  $ \Fbqmlnp $  as well.
 \vskip3pt
    \textit{(b)}\;  Here again, we stress that both  $ \Fbqintmlnp $  and  $ \Fbqmlnp $  are just  $ \k $--superbial\-gebras,  but  \textsl{not Hopf\/}  superalgebras.  Again, to get Hopf-like objects one should enlarge those superbialgebras by adding the inverse of some carefully chosen ``quantum Berezinian'', which goes far beyond the scopes of the present work.  Also, about our choice of ``integral forms'' in this multiparameter setup, a similar comment as in  Remarks \ref{rmks: Fqmlnp-not-Hopf}\textit{(c)\/}  is in order again.
\end{rmks}

\vskip9pt

\begin{free text}  \label{free: Mp-QFSA's as 2-cocyc-deform.'s}
 \textbf{Multiparametric QFSA's as 2--cocycle deformations.}
 It is worth stressing that one can also obtain  $ \Fbqintmlnp \, $,  resp.\  $ \Fbqmlnp \, $,  from its ``uniparametric counterpart''  $ \Fqintmlnp \, $,  resp.\  $ \Fqmlnp \, $,  directly via a process of deformation by 2--cocycle, after an initial extension of scalars.  We present the construction for  $ \Fbqintmlnp \, $,  from which that for  $ \Fbqmlnp $  follows too, e.g.\ by scalar extension.
 \vskip5pt
   Let  $ \; \Fqbullintmlnp := \kbqqm \otimes_{\k[\,q\,,\,q^{-1}]} \Fqintmlnp \; $,  \,with its built-in structure of  $ \kbqqm $--superbialgebra; then let us extend again to the larger superbialgebra
 $ \; F_{q,\,\bullet}^{\sqrt{\textsl{int}}}\big[M\!L_{\,n}^{\,p}\big] := \k\big[\,q^{\pm 1/2},\bq^{\pm 1/2}\big] \otimes_{\k[\,q^{\pm 1}\,,\,\bq^{\pm 1}]} \Fqbullintmlnp \; $
 where  $ \k\big[q^{\pm 1/2},\bq^{\pm 1/2}\big] $  is the extension of  $ \kbqqm $  obtained by taking formal square roots  $ \, q^{\pm 1/2} $  and  $ \, q_{i{}j}^{\,\pm 1/2} \, $  of  $ q^{\pm 1} $  and of  $ \, q_{i{}j}^{\,\pm 1} \, $  ($ \, \forall \; i \, , j \in I_n \, $).  Due to  Theorem \ref{thm: PBW x Fqmlnp},  the set of truncated ordered monomials (with obvious abuse of notation)
  $$  \mathbb{B}  \; := \;  \bigg\{\, {\textstyle \mathop{\overrightarrow{\prod}}\limits_{i, j \in I_n}} \hskip-3pt x_{i{}j}^{\,e_{i{}j}} \;\bigg|\; e_{i{}j} \in \NN \, , \, \forall \; i \, , j \in \! I_n \, , \; e_{i{}j} \leq 1 \text{\;\ if\;\ }  p_{i,j} = \one \,\bigg\}  $$
 is a  $ \k\big[q^{\pm 1/2},\,\bq^{\pm 1/2}\big] $--basis  of  $ F_{q,\,\bullet}^{\sqrt{\textsl{int}}}\big[M\!L_{\,n}^{\,p}\big] \, $.  Then we use this  $ \k\big[q^{\pm 1/2},\,\bq^{\pm 1/2}\big] $--basis  to define a map  $ \, \sigma_\Phipicc \in {\Big( F_{q,\,\bullet}^{\sqrt{\textsl{int}}}\big[M\!L_{\,n}^{\,p}\big] \otimes F_{q,\,\bullet}^{\sqrt{\textsl{int}}}\big[M\!L_{\,n}^{\,p}\big] \Big)}^* \, $  by setting
\begin{equation}  \label{eq: sigmaPhi-x-Fbqintmlnp}
 \sigma_\Phipicc \bigg(\, {\textstyle \mathop{\overrightarrow{\prod}}\limits_{i, j \in I_n}} \hskip-3pt x_{i{}j}^{\,e'_{i{}j}} \, , {\textstyle \mathop{\overrightarrow{\prod}}\limits_{i, j \in I_n}} \hskip-3pt x_{i{}j}^{\,e''_{i{}j}} \bigg)  \; := \;  {\textstyle \prod\limits_{i \not= j}} \, \delta_{e'_{i{}j},0} \, \delta_{e''_{i{}j},0} \, {\textstyle \prod\limits_{r,s=1}^n} \, q_{r,s}^{\,e'_{r{}r} \, e''_{s{}s}}
\end{equation}
 Then this  $ \sigma_\Phipicc $  is a 2--cocycle for the Hopf algebra  $ F_{q,\,\bullet}^{\sqrt{\textsl{int}}}\big[M\!L_{\,n}^{\,p}\big] \, $,  in the usual sense: indeed, if we pick the embedding  $ \, \Fqintmlnp \lhook\joinrel\longrightarrow \Fhglnp \, $  in  Theorem \ref{thm: pres-Fhg}\textit{(d)\/}  and extend it to an embedding  $ \, F_{q,\,\bullet}^{\sqrt{\textsl{int}}}\big[M\!L_{\,n}^{\,p}\big] \lhook\joinrel\relbar\joinrel\relbar\joinrel\longrightarrow \k\big[q^{\pm 1/2},\bq^{\pm 1/2}\big] \!\mathop{\otimes}\limits_{\k[\,q^{+1/2},\,q^{-1/2}]}\hskip-5pt \Fhglnp \, $  (via scalar extension), then the  $ \sigma_\Phipicc $  defined by  \eqref{eq: sigmaPhi-x-Fbqintmlnp}  is nothing but the restriction to  $ F_{q,\,\bullet}^{\sqrt{\textsl{int}}}\big[M\!L_{\,n}^{\,p}\big] $  of the 2--cocycle  $ \sigma_\Phipicc $  of  $ \, \k\big[q^{\pm 1/2},\bq^{\pm 1/2}\big] \!\mathop{\otimes}\limits_{\k[\,q^{+1/2},\,q^{-1/2}]}\hskip-5pt \Fhglnp \, $  given (via scalar extension) by  \eqref{eq: sigma_Phi}  in  \S \ref{free: constr.-FhPhiG}   --- in particular, then, the latter is a 2--cocycle just because the former is.
                                                                           \par
   Using the 2--cocycle  $ \sigma_\Phipicc $  of  $ F_{q,\,\bullet}^{\sqrt{\textsl{int}}}\big[M\!L_{\,n}^{\,p}\big] $  we can construct the deformed Hopf superalgebra  $ \, {\Big( F_{q,\,\bullet}^{\sqrt{\textsl{int}}}\big[M\!L_{\,n}^{\,p}\big] \Big)}_{\!\sigma_\Phipicc} \, $.  Then, repeating the analysis made in  \S \ref{free: constr.-FhPhiG}  and in the proof of  Theorem \ref{thm: pres-FhPhiglnp}  (up to minimal changes) we find that
 \vskip7pt
   \centerline{ \textsl{$ \underline{\text{Claim}} $:}\;  \textit{$ {\Big( F_{q,\,\bullet}^{\sqrt{\textsl{int}}}\big[M\!L_{\,n}^{\,p}\big] \Big)}_{\!\sigma_\Phipicc} \, $  has the structure described in  Theorem \ref{thm: pres-Fbqmln},} }
   \centerline{ \textsl{but}  \textit{for the scalar extension from\/  $ \kbqqm $  to\/  $ \k\big[q^{\pm 1/2},\bq^{\pm 1/2}\big] \, $} }
 \vskip7pt
   Finally, at this point  $ \Fbqintmlnp $  identifies with the  $ \kbqqm $--subalgebra  of  $ {\Big( F_{q,\,\bullet}^{\sqrt{\textsl{int}}}\big[M\!L_{\,n}^{\,p}\big] \Big)}_{\sigma_\Phipicc} $  generated (once again!) by the  $ x_{i{}j} $'s.  So, in force of the above  \textsl{$ \underline{\text{Claim}} \, $},  we can loosely say (just neglecting any intermediate change of scalars) that
 \vskip3pt
   \centerline{ \textit{``$ \, \Fbqintmlnp $  is a 2--cocycle deformation of  $ \Fqintmlnp $''} }
 \vskip5pt
\noindent
 and similarly with  $ \Fbqmlnp  $,  resp.\  $ \Fqmlnp  $,  replacing  $ \Fbqintmlnp  $,  resp.\  $ \Fqintmlnp $.
\end{free text}

\vskip11pt

   The following ``multiparametric PBW theorem'' (an ``integral version'' of  Theorem \ref{thm: PBW x FhPhiglnp})  now provides a  $ \kbqqm $--basis  of ordered (truncated) monomials:

\vskip15pt

\begin{theorem}  \label{thm: PBW x Fbqmlnp}
 \textsl{(PBW Theorem for  $ \Fbqintmlnp $  and  $ \Fbqmlnp \, $)}
 Let us fix any total order in the set  $ \, \big\{\, x_{i{}j} \,\big|\, i , j \in I_n \,\big\} \, $  of generators of  $ \Fbqintmlnp \, $.  Then the set
  $$  \mathbb{B}  \; := \;  \bigg\{\, {\textstyle \mathop{\overrightarrow{\prod}}\limits_{i, j \in I_n}} \hskip-3pt x_{i{}j}^{\,e_{i{}j}} \;\bigg|\; e_{i{}j} \in \NN \, , \, \forall \; i \, , j \in \! I_n \, , \; e_{i{}j} \leq 1 \text{\;\ if\;\ }  p_{i,j} = \one \,\bigg\}  $$
 of all truncated ordered monomials in the  $ x_{i{}j} $'s  is a\/  $ \kbqqm $--basis  of the\/  $ \kqqm $--module  $ \Fbqmlnp \, $.  In particular, then, the  $ \kbqqm $--module  $ \Fbqintmlnp $  is free.
                                                                    \par
   A parallel result holds for  $ \, \Fbqmlnp := \kbq \otimes_{\k[q^{\pm 1},\,\bq^{\pm 1}]} \Fbqintmlnp \, $  as well, namely\/  $ \mathbb{B} $  is a  $ \kbq $--basis  of  $ \Fbqmlnp \, $.
\end{theorem}

\begin{proof}
 Wee can mimic the proof of  Theorem \ref{thm: PBW x Fqmlnp},  yet we take a different path.  Again, we prove the statement for  $ \Fbqintmlnp \, $,  the one for  $ \Fbqmlnp $  follows easily.
 \vskip3pt
   In short, the claim follows at once from the fact that  $ \Fbqintmlnp $  is (up to details) a 2--cocycle deformation of  $ \Fqintmlnp \, $.  Now, formula  \eqref{eq: sigmaPhi-x-Fbqintmlnp}  together with the analysis performed in  \S \ref{free: constr.-FhPhiG}  (with minimal, irrelevant changes to adapt it to the ``polynomial setup'') show that, through the deformation process, every monomial in  $ \mathbb{B} $  for the deformed algebra (i.e., w.r.t.\ the new, deformed product) coincide with the same monomial in the undeformed algebra (i.e., w.r.t.\ the old, undeformed product)  \textsl{but for a coefficient which is a monomial in  $ q^{\pm 1/2} $  and in the  $ q_{i{}j}^{\,\pm 1/2} $}.  Thus, the fact that they form a basis in  $ \Fqintmlnp $  implies that they form a basis in  $ \Fbqintmlnp $  too.
 \vskip3pt
   Indeed, to be precise, the above proves that following.  With notations and assumptions as in  \S \ref{free: Mp-QFSA's as 2-cocyc-deform.'s},  the fact that the truncated, ordered monomials (for the initial product) form a basis of  $ \Fqintmlnp $  implies the same statement for the extended superbialgebra  $ F_{q,\,\bullet}^{\sqrt{\textsl{int}}}\big[M\!L_{\,n}^{\,p}\big] $   --- which is obvious ---   and then also that the similar monomials,  \textsl{with respect to the  \textit{deformed}  product},  do form a basis of the deformed superbialgebra  $ {\Big( F_{q,\,\bullet}^{\sqrt{\textsl{int}}}\big[M\!L_{\,n}^{\,p}\big] \Big)}_{\!\sigma_\Phipicc} \, $.  In particular, these monomials are linearly independent over  $ \k\big[q^{\pm 1/2},\bq^{\pm 1/2}\big] \, $,  hence also over  $ \kbqqm $ inside  $ \Fbqintmlnp \, $,  identified with the  $ \kbqqm $--subsuperalgebra  of  $ {\Big( F_{q,\,\bullet}^{\sqrt{\textsl{int}}}\big[M\!L_{\,n}^{\,p}\big] \Big)}_{\!\sigma_\Phipicc} $  generated by the  $ x_{i{}j} $'s  (see  \S \ref{free: Mp-QFSA's as 2-cocyc-deform.'s}).  On the other hand, these monomials also span  $ {\Big( F_{q,\,\bullet}^{\sqrt{\textsl{int}}}\big[M\!L_{\,n}^{\,p}\big] \Big)}_{\!\sigma_\Phipicc} $  over  $ \kbqqm \, $,  by the same argument we already applied in the ``uniparametric'' case of  $ \Fqintmlnp $   --- cf.\ the proof of  Theorem \ref{thm: PBW x Fqmlnp}  ---   hence we are done.
\end{proof}

\vskip11pt
  \eject

\begin{free text}  \label{free: specialisations-FqPhimlnp}
 \textbf{Specialisations of integral QFSA's.}
   Let us consider ``integral'' QFSA's  $ \Fqintmlnp $  and  $ \Fbqintmlnp \, $.  Let  $ R $  be any  $ \k $--algebra,  and let  $ \, \chi : \kqqm \relbar\joinrel\longrightarrow R \, $  and  $ \, \mathbf{\psi} : \kbqqm \relbar\joinrel\longrightarrow R \, $  be  $ \k $--algebra  morphisms.  Then the  $ R $--superbialgebras
 \vskip5pt
  $$  F_{R,\,\chi}^{\,\textsl{int}}\big[M\!L_{\,n}^{\,p}\big]  \, := \, R \hskip-7pt\mathop{\otimes}\limits_{\k[q^{+1},\,q^{-1}]}\hskip-7pt \Fqintmlnp  \quad  ,   \qquad  F_{R,\,\mathbf{\psi}}^{\Phipicc,\,\textsl{int}}\big[M\!L_{\,n}^{\,p}\big]  \, := \,  R \hskip-7pt\mathop{\otimes}\limits_{\k[q^{\pm 1},\,\bq^{\pm 1}]}\hskip-7pt \Fbqintmlnp  $$
 \vskip5pt
\noindent
 given by scalar extension   --- via  $ \chi $  and via  $ \mathbf{\psi} \, $,  respectively ---   can be legitimally called ``specialisations'' of  $ \Fqintmlnp $  and  $ \Fbqintmlnp \, $,  respectively.
 \vskip5pt
   When  $ \, \chi\big(q^{\pm 1}\big) = 1 \, $,  we have  $ \; F_{R,\,\chi}^{\,\textsl{int}}\big[M\!L_{\,n}^{\,p}\big] \, \cong \, R \otimes_\k F\big[{M\!L}_{\,n}^{\,p}\big] \, $  as Hopf superalgebras over  $ R \, $:  in addition,  $ F\big[{M\!L}_{\,n}^{\,p}\big] $  inherits from this ``uniparametric'' quantisation  $ \Fqintmlnp $  a Poisson bracket  $ \{\,\ ,\ \} $  which makes  $ {M\!L}_{\,n}^{\,p} $  into a Poisson monoid.
 \vskip5pt
   Similarly, when  $ \, \mathbf{\psi}\big(q^{\pm 1}\big) = 1 = \mathbf{\psi}\big(q_{i{}j}^{\pm 1}\big) \, $  for all  $ \, i \, , j \in I_n \, $,  then we find again that  $ \; F_{R,\,\mathbf{\psi}}^{\Phi,\,\textsl{int}}\big[M\!L_{\,n}^{\,p}\big] \, \cong \, R \otimes_\k F\big[{M\!L}_{\,n}^{\,p}\big] \, $  as Hopf superalgebras over  $ R \, $.  In addition, if  $ \, \mathbf{\psi}(q_{\,i{}j}) = {\big( \mathbf{\psi}(q) \big)}^{z_{i{}j}} \, $  for some  $ \, z_{i{}j} \in \ZZ \, $,  then a Poisson bracket  $ {\{\,\ ,\ \}}_{\mathbf{\psi}} $  is canonically defined on  $ \, R \,\otimes_\k F\big[{M\!L}_{\,n}^{\,p}\big] \, $,  \,which makes  $ {M\!L}_{\,n}^{\,p} $  into a Poisson monoid whose bracket depends on the ``multiparameter'' encoded by  $ \mathbf{\psi} \, $.
\end{free text}

\vskip15pt

\subsection{Comparison with Manin's work}  \label{subsec: compare-Manin}
 \vskip7pt
   We conclude this work with a quick comparison between our results and those by Manin in  \cite{Ma}.  Indeed, Manin considers multiparametric quantisations of  $ F\big[{M\!L}_{\,n}^{\,p}\big] \, $,  that we denote by  $ \F_\bq^\Mpicc\big[{M\!L}_{\,n}^{\,p}\big] \, $:  they are generated by elements  $ z_{i{}j} \, (\, i \, , j \in I_n \,) \, $  with relations depending on an antisymmetric matrix of invertible parameters  $ {\big( q_{i{}j} \big)}_{i \in I_n}^{j \in I_n} \, $.  Therefore, we are naturally led to compare any such  $ \F_\bq^\Mpicc\big[{M\!L}_{\,n}^{\,p}\big] $  with our polynomial multiparametric QFSA's  $ \Fbqmlnp \, $.
 \vskip5pt
   Beyond the first-sight similarity of their own constructions   --- as both (super)alge\-bras have essentially ``the same'' set of generators and relations depending on the same set of parameters ---   same extra similarities also show up in some of the relations between generators.  One then might be led to guess that, up to suitably re-writing some relations (in either one of the two presentations, or in both), one can eventually achieve the same kind of presentation for both algebras, thus concluding that they are isomorphic, at least  \textsl{as (super)algebras}.  On the other hand, failing to get this would not allow us to deduce that the two are non-isomorphic either.
 \vskip5pt
   Nevertheless, there is another difference which seems to be a major ``obstruction'' for  $ \F_\bq^\Mpicc\big[{M\!L}_{\,n}^{\,p}\big] $  and  $ \Fbqmlnp $  to be isomorphic  \textsl{as superbialgebras}.  Indeed, the generators  $ z_{i{}j} $  of  $ \F_\bq^\Mpicc\big[{M\!L}_{\,n}^{\,p}\big] $  and the generators  $ x_{i{}j} $  of  $ \Fbqmlnp $  behave rather differently with respect to the coproduct, namely (for all  $ \, i \, , j \in I_n \, $)
 \vskip3pt
  $$  \Delta(z_{i{}j}) \, = \, {\textstyle \sum_{k=1}^n} z_{i{}k} \otimes z_{k{}j}   \qquad  \text{and}  \qquad   \Delta(x_{i{}j}) \, = \, {\textstyle \sum_{k=1}^n} {(-1)}^{p_{i{}k} \, p_{k{}j}} x_{i{}k} \otimes x_{k{}j}  $$
 \vskip5pt
   Although this is not a conclusive argument, it certainly drives us to guess that  $ \F_\bq^\Mpicc\big[{M\!L}_{\,n}^{\,p}\big] $  and  $ \Fbqmlnp $  might indeed be non-isomorphic.

\bigskip
 \vfill
  \eject

\appendix
\section{Further details for  \S \ref{free: constr.-FhG}}  \label{appendix}
 \vskip11pt
   In the present Appendix we present the details of the intermediate claims which were left unproved in the construction of  $ \Fhglnp $  in  \S \ref{free: constr.-FhG}.

\vskip9pt

   \textit{$ \underline{\text{Proof of (4--a)}} $:}\;  Proving  $ \, \big\langle\, 1 \, , \rho \,\big\rangle = 0 \, $  is trivial, while the second case is a sheer matter of straightforward computations.  In detail, direct calculations give:
 \vskip5pt
   ---  for  $ \; \rho = \varGamma_k \, \varGamma_\ell - \varGamma_\ell \, \varGamma_k \; $  we get
  $$  \displaylines{
   \quad   \langle x_{i{}j} \, , \varGamma_k \, \varGamma_\ell - \varGamma_\ell \, \varGamma_k \rangle  \; = \;
{\textstyle \sum\limits_{a=1}^n}\, {(-1)}^{p_{i{}a} p_{a{}j}} \big\langle\, x_{i{}a} \otimes x_{a{}j} \, , \varGamma_k \otimes \varGamma_\ell - \varGamma_\ell \otimes \varGamma_k \,\big\rangle  \; =   \hfill  \cr
   = \;  {\textstyle \sum\limits_{a=1}^n}\, {(-1)}^{p_{i{}a} p_{a{}j}} \Big( \big\langle\, x_{i{}a} \, , \varGamma_k \,\big\rangle \, \big\langle\, x_{a{}j} \, , \varGamma_\ell \,\big\rangle - \big\langle\, x_{i{}a} \, , \varGamma_\ell \,\big\rangle \, \big\langle\, x_{a{}j} \, , \varGamma_k \,\big\rangle \Big)  \; =  \cr
   \hfill   = \;  {\textstyle \sum\limits_{a=1}^n}\, {(-1)}^{p_{i{}a} p_{a{}j}} \Big( \delta_{i,a,k,\ell,j} - \delta_{i,a,k,\ell,j} \Big)  \; = \;  0   \quad  }  $$
 where hereafter we write  $ \, \delta_{e_1,\dots,e_s} := 1 \, $  if  $ \, e_1 = \cdots = e_s \, $  and  $ \, \delta_{e_1,\dots,e_s} := 0 \, $  otherwise;
 \vskip5pt
   ---  for  $ \; \rho = \varGamma_k \, E_r - E_r \, \varGamma_k - (\delta_{k,\,r} - \delta_{k,\,r+1}) \, E_r \; $  we get
  $$  \displaylines{
   \big\langle\, x_{i{}j} \, , \varGamma_k \, E_r - E_r \, \varGamma_k - (\delta_{k,\,r} - \delta_{k,\,r+1}) \, E_r \big\rangle  \; =   \hfill  \cr
   =  \left(\, {\textstyle \sum\limits_{a=1}^n}\, {(-1)}^{p_{i{}a} p_{a{}j}} \big\langle\, x_{i{}a} \otimes x_{a{}j} \, , \varGamma_k \otimes E_r - E_r \otimes \varGamma_k \big\rangle \!\right) - (\delta_{k,\,r} - \delta_{k,\,r+1}) \big\langle\, x_{i{}j} \, , E_r \big\rangle  \; =  \cr
   \quad \quad \quad   = \;  \bigg(\; {\textstyle \sum\limits_{a=1}^n}\, \Big( {(-1)}^{p_{i{}a} p_{a{}j}} \big\langle\, x_{i{}a} \, , \varGamma_k \big\rangle \, \big\langle\, x_{a{}j} \, , E_r \big\rangle \, -   \hfill  \cr
   - \, {(-1)}^{(p_{i,a} + p_{r,r+1}) p_{a,j}} \big\langle\, x_{i{}a} \, , E_r \big\rangle \, \big\langle\, x_{a{}j} \, , \varGamma_k \big\rangle \Big) \bigg) \, -  \cr
   \hfill   - \, (\delta_{k,r} - \delta_{k,r+1}) \, \big\langle\, x_{i{}j} \, , E_r \big\rangle  \; =   \quad  \cr
   \quad \quad   = \;  \left(\; {\textstyle \sum\limits_{a=1}^n}\, \Big( {(-1)}^{p_{i{}a} p_{a{}j}} \delta_{i,a,k} \, \delta_{a+1,j,r+1} - {(-1)}^{(p_{i,a} + p_{r,r+1}) p_{a,j}} \delta_{i+1,a,r+1} \, \delta_{a,j,k} \Big) \right) \, -   \hfill  \cr
   \hfill   - \, (\delta_{k,r} - \delta_{k,r+1}) \, \delta _{i+1,j,r+1}  \; =   \quad \quad  \cr
   \quad \quad \quad   = \;  \delta_{i,k} \, \delta_{i+1,j,r+1} - \delta_{i+1,j,r+1} \, \delta_{j,k} - (\delta_{k,r} - \delta_{k,r+1}) \, \delta _{i+1,j,r+1}  \; = \;  0  \quad  }  $$
 and a similar computation for  $ \; \rho = \varGamma_k \, F_r - F_r \, \varGamma_k - (\delta_{k,\,r} - \delta_{k+1,\,r}) \, F_r \; $  yields also
  $$  \big\langle\, x_{i{}j} \, , \varGamma_k \, F_r - F_r \, \varGamma_k - (\delta_{k,\,r} - \delta_{k+1,\,r}) \, F_r \big\rangle  \; = \; 0  $$
 \vskip3pt
   ---  for  $ \; \rho \, = \, E_r \, F_s - {(-1)}^{p_{r,r+1} p_{s+1,s}} F_s \, E_r - \delta_{r{}s} \frac{\;e^{+\hbar H_r} - e^{-\hbar H_r}\;}{q_r^{+1} - q_r^{-1}} \; $,  \; we split the computation, with  $ \; \big\langle\, x_{i{}j} \, , E_r \, F_s - {(-1)}^{p_{r,r+1} p_{s+1,s}} F_s \, E_r \big\rangle \; $  and  $ \; \left\langle\, x_{i{}j} \, , \, \delta_{r{}s} \frac{\;e^{+\hbar H_r} - e^{-\hbar H_r}\;}{q_r^{+1} - q_r^{-1}} \,\right\rangle \; $  dealt with separately.  For the first term we get
  $$  \displaylines{
   \big\langle\, x_{i{}j} \, , E_r \, F_s - {(-1)}^{p_{r,r+1} p_{s+1,s}} F_s \, E_r \big\rangle  \; =   \hfill  \cr
   \qquad   = \;  {\textstyle \sum\limits_{a=1}^n}\, {(-1)}^{p_{i{}a} p_{a{}j}} \big\langle x_{i{}a} \otimes x_{a{}j} \, , E_r \otimes F_s - {(-1)}^{p_{r,r+1} p_{s+1,s}} F_s \otimes E_r \big\rangle \; =   \hfill
   }  $$
  $$  \displaylines{
   \qquad   =  \; {\textstyle \sum\limits_{a=1}^n}\, \Big( {(-1)}^{(p_{i,a} + p_{r,r+1}) p_{a,j}} \langle x_{i{}a} \, , E_r \rangle \langle x_{a{}j} \, , F_s \rangle \, -   \hfill  \cr
   \qquad \qquad \qquad   - \, {(-1)}^{(p_{s+1,s} + p_{i,a}) p_{a,j} + p_{r,r+1} p_{s+1,s}} \langle x_{i{}a} \, , F_s \rangle \langle x_{a{}j} \, , E_r \rangle \Big)  \; =   \hfill  \cr
   \qquad   =  \; {\textstyle \sum\limits_{a=1}^n}\, \Big( {(-1)}^{(p_{i,a} + p_{r,r+1}) p_{a,j}} \, \delta_{i+1,a} \, \delta_{i,r} \, \delta_{a,j+1} \, \delta_{j,s} \, -   \hfill  \cr
   \qquad \qquad \qquad   - \, {(-1)}^{(p_{s,s+1} + p_{i,a}) p_{a,j} + p_{r,r+1} p_{s+1,s}} \, \delta_{i,a+1} \, \delta_{s,a} \, \delta_{a,r} \, \delta_{a+1,j} \Big)  \; =   \hfill  \cr
   = \!  \Big(\! {(-\!1)}^{\! (p_{r,r+1} + p_{r,r+1}) p_{r+1,r}} \delta_{i,j} \, \delta_{r,s} \, \delta_{i,r} - {(-1)}^{\! (p_{r,r+1} + p_{r+1,r}) p_{r,r+1} + p_{r,r+1} p_{r+1,r}} \delta_{i,j} \, \delta_{r,s} \, \delta_{i,r+1} \Big)  =  \cr
  \hfill   = \;\,  \delta_{i,j} \, \delta_{r,s} \Big( \delta_{i,r} - {(-1)}^{p_{r,r+1}} \, \delta_{i,r+1} \Big)  } $$
 and for the second we find
  $$  \displaylines{
   \left\langle x_{i{}j} \; , \, \delta_{r{}s} \frac{\;e^{+\hbar H_r} - e^{-\hbar H_r}\;}{q_r^{+1} - q_r^{-1}} \right\rangle  \; =   \hfill  \cr
   = \;  \delta_{r,\,s} \, \delta_{i,\,j} \, \frac{\;e^{+\hbar \, ({(-1)}^{p(r)} \delta_{i,r} - {(-1)}^{p(r+1)} \delta_{i,r+1})} - e^{-\hbar \, ({(-1)}^{p(r)} \delta_{i,\,r} - {(-1)}^{p(r+1)} \delta_{i,\,r+1})}\;}{e^{+ \hbar \, {(-1)}^{p(r)}} - e^{- \hbar \, {(-1)}^{p(r)}}}  \; =  \cr
   \hfill   = \;  \delta_{i,\,j} \, \delta_{r,\,s} \, \big(\, \delta_{i,\,r} - {(-1)}^{p_{r,r+1}} \delta_{i,\,r+1} \,\big)  }  $$
 hence eventually direct comparison yields the result we were looking for;
 \vskip3pt
   ---  for the last cases, we need repeated applications of some intermediate formulas: the first one is
\begin{equation}  \label{eq: x(i,j) - ErEs}
   \big\langle\, x_{i{}j} \, , E_r E_s \big\rangle  \; = \;  \delta_{\,i,\,r,\,s-1,\,j-2}
\end{equation}
 that follows from
  $$  \displaylines{
   \quad   \big\langle\, x_{i{}j} \, , E_r E_s \big\rangle  \; = \;  \big\langle\, \Delta(x_{i{}j}) \, , E_r \otimes E_s \big\rangle  \; =   \hfill  \cr
   \quad \qquad   = \;  {\textstyle \sum\limits_{a=1}^n}\, {(-1)}^{p_{i{}a} p_{a{}j}} \, \big\langle\, x_{i{}a} \otimes x_{a{}j} \, , E_r \otimes E_s \big\rangle  \; =   \hfill  \cr
   \quad \qquad \qquad   = \;  {\textstyle \sum\limits_{a=1}^n}\, {(-1)}^{p_{i,a} p_{a,j} \, + \, p_{a,j} p_{r,r+1}} \, \big\langle\, x_{i{}a} \, , E_r \big\rangle \, \big\langle\, x_{a{}j} \, , E_s \big\rangle  \; =   \hfill  \cr
   \quad \qquad \qquad \qquad   = \;  {\textstyle \sum\limits_{a=1}^n}\, {(-1)}^{(p_{i,a} \, + \, p_{r,\,r+1}) \, p_{a,j}} \, \delta_{i,\,r} \, \delta_{a,\,r+1} \, \delta_{a,\,s} \, \delta_{j,\,s+1}  \; =   \hfill  \cr
   \quad \qquad \qquad \qquad \qquad   = \;  {(-1)}^{(p_{i,i+1} + p_{i,i+1}) \, p_{i+1,i+2}} \, \delta_{\,i,\,r,\,s-1,\,j-2}  \; =   \hfill  \cr
   \quad \qquad \qquad \qquad \qquad \qquad   = \;  {(-1)}^{\zero \,\cdot\, p_{i+1,i+2}} \, \delta_{\,i,\,r,\,s-1,\,j-2}  \; = \;  \delta_{\,i,\,r,\,s-1,\,j-2}   \hfill  }  $$
   \indent   It is the easily seen that  \eqref{eq: x(i,j) - ErEs}  implies  $ \; \big\langle\, x_{i{}j} \, , \rho \,\big\rangle = 0 \; $  when  $ \; \rho \, = \, E_i^{\,2} \; $  as in  \eqref{eq:UqgR4}  or  $ \; \rho \, = \, [E_i\,,E_j] \; $  as in  \eqref{eq:UqgR5}   --- for the  ``$ E $--half''  of either formula: a similar analysis takes care of the  ``$ F $--half''  case as well.
 \vskip5pt
   The second intermediate formula is
\begin{equation}  \label{eq: x(i,j) - ErEsEt}
   \big\langle\, x_{i{}j} \, , E_r E_s E_t \big\rangle  \; = \;  \delta_{i,\,r,\,s-1,\,t-2,\,j-3}
\end{equation}
 which is the output of the following calculation:
  $$  \displaylines{
   \quad   \big\langle\, x_{i{}j} \, , E_r E_s E_t \big\rangle  \; = \;  \big\langle \Delta(x_{i{}j}) \, , E_r E_s \otimes E_t \big\rangle  \; =   \hfill  \cr
   \quad \qquad   = \;  {\textstyle \sum\limits_{a=1}^n}\, {(-1)}^{p_{i{}a} p_{a{}j}} \big\langle x_{i{}a} \otimes x_{a{}j} \, , E_r E_s \otimes E_t \big\rangle  \; =   \hfill
  }  $$
  $$  \displaylines{
   \quad \qquad \qquad   = \;  {\textstyle \sum\limits_{a=1}^n}\, {(-1)}^{(p_{i,a} + p_{r,r+1} + p_{s,s+1}) \, p_{a,j}} \, \big\langle x_{i{}a} \, , E_r E_s \big\rangle \, \big\langle x_{a{}j} \, , E_t \big\rangle  \; =   \hfill  \cr
   \quad \qquad \qquad \qquad   = \;  {\textstyle \sum\limits_{a=1}^n}\, {(-1)}^{(p_{i,a} + p_{r,r+1} + p_{s,s+1}) \, p_{a,j}} \, \delta_{i,\,r,\,s-1,a-2} \, \delta_{a,\,t} \, \delta_{j,\,t+1}  \; =   \hfill  \cr
   \quad \qquad \qquad \qquad \qquad   = \;  {(-1)}^{(p_{i,i+2} + p_{i,i+1} + p_{i+1,i+2}) \, p_{i+2,i+3}} \, \delta_{i,\,r,\,s-1,\,t-2,\,j-3}  \; =   \hfill  \cr
   \quad \qquad \qquad \qquad \qquad \qquad   = \;  {(-1)}^{\zero \,\cdot\, p_{i+2,i+3}} \, \delta_{i,\,r,\,s-1,\,t-2,\,j-3}  \; = \;  \delta_{i,\,r,\,s-1,\,t-2,\,j-3}   \hfill  }  $$
   \indent   It is clear then that  \eqref{eq: x(i,j) - ErEsEt}  implies  $ \; \big\langle\, x_{i{}j} \, , \rho \,\big\rangle = 0 \; $  for every  $ \, \rho \, $  as in the left-hand side of  \eqref{eq:UqgR6}  concerning the generators  $ E_t \, $:  \,a parallel analysis takes care of the  ``$ F $--part''  of that same formula as well.
 \vskip5pt
   The third intermediate
   formula is
\begin{equation}  \label{eq: x(i,j) - ErEsEtEv}
\begin{gathered}
 \big\langle\, x_{i{}j} \, , E_r E_s E_t E_v \big\rangle  \; = \;  \delta_{i,\,r,\,s-1,\,t-2,\,v-3,\,j-4}
\end{gathered}
\end{equation}
 which comes out of the following computation:
  $$  \displaylines{
   \big\langle\, x_{i{}j} \, , E_r E_s E_t E_v \big\rangle  \; = \;  \big\langle \Delta(x_{i{}j}) \, , E_r E_s E_t \otimes E_v \big\rangle  \; =   \hfill  \cr
   \qquad   = \;  {\textstyle \sum\limits_{a=1}^n}\, {(-1)}^{p_{i{}a} p_{a{}j}} \big\langle x_{i{}a} \otimes x_{a{}j} \, , E_r E_s E_t \otimes E_v \big\rangle  \; =   \hfill  \cr
   \qquad \qquad   = \;  {\textstyle \sum\limits_{a=1}^n}\, {(-1)}^{p_{i,a} p_{a,j} + p_{a,j} p_{r,r+1} + p_{a,j} p_{s,s+1} + p_{a,j} p_{t,t+1}} \, \big\langle x_{i{}a} \, , E_r E_s E_t \big\rangle \, \big\langle x_{a{}j} \, , E_v \big\rangle  \; =   \hfill  \cr
   \qquad \qquad \qquad   = \;  {\textstyle \sum\limits_{a=1}^n}\, {(-1)}^{(p_{i,a} + p_{r,r+1} + p_{s,s+1} + p_{a,j} p_{t,t+1}) \, p_{a,j}} \, \delta_{i,\,r,\,s-1,\,t-2,a-3} \, \delta_{a,\,v} \, \delta_{j,\,v+1}  \; =   \hfill  \cr
   \qquad \qquad \qquad \qquad   = \;  {(-1)}^{(p_{i,i+3} + p_{i,i+1} + p_{i+1,i+2} + p_{i+2,i+3}) \, p_{i+3,i+4}} \, \delta_{i,\,r,\,s-1,\,t-2,\,v-3,\,j-4}  \; =   \hfill  \cr
   \qquad \qquad \qquad \qquad \qquad   = \;  {(-1)}^{\zero \,\cdot\, p_{i+3,i+4}} \, \delta_{i,\,r,\,s-1,\,t-2,\,v-3,\,j-4}  \,\; = \;\,  \delta_{i,\,r,\,s-1,\,t-2,\,v-3,\,j-4}  }  $$
%
%
   \indent   Then from  \eqref{eq: x(i,j) - ErEsEtEv}  one easily deduces that  $ \; \big\langle\, x_{i{}j} \, , \rho \,\big\rangle = 0 \; $  for every  $ \, \rho \, $  as in the left-hand side of  \eqref{eq:UqgR7},  first line   --- i.e. the one for generators  $ E_t \, $:  \,a parallel analysis takes care of the parallel case where the  $ F $'s  play the role of the  $ E $'s  as well.
 \vskip9pt
   \textit{$ \underline{\text{Proof of (4--b)}} $:}\;  The fact that  $ \, \mathcal{R} $  be a coideal in  $ \hat{\mathbb{U}}_\hbar $  is equivalent to the fact that the quotient topological  $ \kh $--superalgebra  $ \, \hat{\mathbb{U}}_\hbar \Big/ \mathcal{R} \, $  be in fact a bialgebra, with coproduct and counit given as in Theorem \ref{thm: Hopf-struct x Uhglnp}.  But that result stands true, hence we deduce (somewhat backwards) that  $ \, \mathcal{R} $  is indeed a coideal in  $ \hat{\mathbb{U}}_\hbar \, $,  \,q.e.d.
                                                                                 \par
   If instead one wants to go through a direct proof, showing that  $ \, \mathcal{R} $  is a coideal, then again this is a matter of sheer calculation.  In detail, computations show that the generators of the form
  $$  \displaylines{
   \varGamma_{k,\ell}  \; := \;  \varGamma_k \, \varGamma_\ell \, - \, \varGamma_\ell \, \varGamma_k  \cr
   E^{(\varGamma)}_{k,j}  \; := \;  \big[ \varGamma_k \, , E_j \big] \, - \, (\delta_{k,j} - \delta_{k,j+1}) \, E_j \cr
   E_i^{\,2}   \qquad  \big(\,\text{with \ }  \, p_{i,\,i+1} = \one \,\big)  \cr
   E_{i,j}  \; := \;  E_i^2 \, E_j \, - \, \big( q + q^{-1} \big) \, E_i \, E_j \, E_i \, + \, E_j \, E_i^2   \qquad   \big(\,\text{with \ }  \, p_{i,\,i+1} = \zero \, , \; |i-j| = 1 \,\big)  }  $$
 one finds that they are skew-primitive, namely
  $$  \displaylines{
   \Delta\big(\varGamma_{k,\ell}\big)  \; = \;  \varGamma_{k,\ell} \otimes 1 \, + \, 1 \otimes \varGamma_{k,\ell}  \cr
   \Delta\big(E^{\,(\varGamma)}_{k,j}\big)  \; = \;  E^{\,(\varGamma)}_{k,j} \otimes 1 \, + \, e^{+H_j} \otimes E^{\,(\varGamma)}_{k,j}  \cr
   \Delta\big(E_i^{\,2}\big)  \; = \;  E_i^{\,2} \otimes 1 \, + \, e^{+2 H_i} \otimes E_i^{\,2}  \cr
   \Delta\big(E_{i,j}\big)  \; = \;  E_{i,j} \otimes 1 \, + \, e^{+2\,\hbar\,H_i + H_j} \otimes E_{i,j}  }  $$
%
 The remaining generators are treated similarly, with similar outcomes.  Overall, this eventually allows to conclude that  $ \, \mathcal{R} $  is indeed a coideal, as claimed.
 \vskip9pt
   \textit{$ \underline{\text{Proof of (5--a)}} $:}\;  Proving that  $ \, \big\langle\, \eta \, , 1 \,\big\rangle = 0 \, $  is trivial: for every possible  $ \eta $  from the first four lines in  \eqref{eq: gen.'s-ideal-J-in-Fh},  in each product of the form  $ \, x_{r{}s} \, x_{t{}\ell} \, $  there is always  $ \, r=s \, $  or  $ \, \ell=t \, $,  \,hence  $ \; \big\langle\, x_{r{}s} \, x_{t{}\ell} \, , 1 \,\big\rangle \, = \, \epsilon\big(x_{r,\,s}\,x_{t,\ell}\big) \, = \, \epsilon(x_{r{}s}) \, \epsilon(x_{t{}\ell}) \, = 0 \; $,  \;that is enough to get  $ \, \big\langle\, \eta \, , 1 \big\rangle = 0 \, $.  The same holds as well for the  $ \eta $'s  from the fifth (and last) line in  \eqref{eq: gen.'s-ideal-J-in-Fh},  but for the cases when  $ \, i=j \, $  and  $ \, h=k \, $:  \,but then we have
  $$  \displaylines{
   \big\langle\, x_{i{}j} \, x_{h{}k} \, - \, {(-1)}^{p_{i{}j} p_{h{}k}} \, x_{h{}k} \, x_{i{}j} \, - \, {(-1)}^{p_{i{}j} + p_{i{}k} + p(i)} \big(\, e^{+\hbar} - e^{-\hbar} \,\big) \, x_{i{}k} \, x_{h{}j} \, , \, 1 \,\big\rangle  \; =   \hfill  \cr
   = \;  \epsilon(x_{i{}i}) \, \epsilon(x_{h{}h}) \, - \, {(-1)}^0 \, \epsilon(x_{h{}h}) \, \epsilon(x_{i{}i}) \, - \, {(-1)}^{p_{i{}j} + p_{i{}k} + p(i)} \big(\, e^{+\hbar} - e^{-\hbar} \,\big) \, \epsilon(x_{i{}h}) \, \epsilon(x_{h{}i})  \; =  \cr
   \hfill   = \;  1 - 1 - {(-1)}^{p_{i{}j} + p_{i{}k} + p(i)} \big(\, e^{+\hbar} - e^{-\hbar} \,\big) \, \delta_{i{}h} \, \delta_{h{}i}  \; = \;  1 - 1 - \big(\, e^{+\hbar} - e^{-\hbar} \,\big) \, 0 \; = \;  0  }  $$
 \vskip5pt
   As to proving  $ \, \big\langle\, \eta \, , \gamma \,\big\rangle = 0 \, $  (for all  $ \eta $'s  and all  $ \gamma $'s  as mentioned above), we proceed in steps.  To begin with, direct computations give
 \vskip3pt
\begin{equation}  \label{eq: x*x_E}
  \begin{gathered}
     \hskip-15pt   \big\langle\, x_{i{}j} \, x_{h{}k} \, , E_r \,\big\rangle  \; = \;
     \big\langle\, x_{i{}j} \otimes x_{h{}k} \, , \Delta(E_r) \big\rangle  \; =   \hfill  \\
     = \;  \big\langle\, x_{i{}j} \otimes x_{h{}k} \, , E_r \otimes 1 + e^{+\hbar H_r} \otimes E_r \,\big\rangle  \; =   \hfill  \\
     = \;  \big\langle\, x_{i{}j} \otimes x_{h{}k} \, , E_r \otimes 1 \,\big\rangle \, + \, \big\langle\, x_{i{}j} \otimes x_{h{}k} \, , \, e^{+\hbar H_r} \otimes E_r \,\big\rangle  \; =  \\
     = \;  {(-1)}^{p_{r,r+1} \, p_{h,k}} \big\langle\, x_{i{}j} \, , E_r \,\big\rangle \, \big\langle\, x_{h{}k} \, , 1 \,\big\rangle \, + \, \big\langle\, x_{i{}j} \, , \, e^{+\hbar H_r} \big\rangle \, \big\langle\, x_{h{}k} \, , E_r \,\big\rangle  \; =  \\
     = \;  {(-1)}^{p_{r,r+1} p_{h,k}} \, \delta_{i,\,r,j-1} \, \delta_{h,\,k} \, + \,
\delta_{i,\,j} \, e^{+\hbar({(-1)}^{p(r)} \delta_{i,r} - {(-1)}^{p(r+1)} \delta_{i,r+1})} \, \delta_{h,k-1,\,r}  \; =   \hfill  \\
     \hfill   = \;  \delta_{i,\,r,j-1} \, \delta_{h,\,k} \, + \,
     e^{+\hbar({(-1)}^{p(r)} \delta_{i,r} - {(-1)}^{p(r+1)} \delta_{i,r+1})} \, \delta_{i,\,j} \, \delta_{h,k-1,\,r}
  \end{gathered}
\end{equation}
 for  $ \, \gamma = E_r \, $,  \,then
\begin{equation}  \label{eq: x*x_G}
  \begin{gathered}
     \hskip-9pt   \big\langle\, x_{i{}j} \, x_{h{}k} \, , \varGamma_s \,\big\rangle  \; = \;
     \big\langle\, x_{i{}j} \otimes x_{h{}k} \, , \Delta(\varGamma_s) \big\rangle  \; =
     \big\langle\, x_{i{}j} \otimes x_{h{}k} \, , \varGamma_s \otimes 1 + 1 \otimes \varGamma_s \,\big\rangle  \; =   \hfill  \\
     = \; \big\langle\, x_{i{}j} \, , \varGamma_s \,\big\rangle \, \langle\, x_{h{}k} \, , 1 \,\rangle \, + \, \langle\, x_{i{}j} \, , 1 \,\rangle \, \,\big\langle x_{h{}k} \, , \varGamma_s \,\big\rangle  \; =  \\
     \hfill   = \;  \delta_{i{}s} \, \delta_{j{}s} \, \delta_{h{}k} \, + \, \delta_{i{}j} \, \delta_{h{}s} \, \delta_{k{}s}  \; = \;  \delta_{i{}j} \, \delta_{h{}k} \, (\delta_{i{}s} + \delta_{k{}s})
  \end{gathered}
\end{equation}
 for  $ \, \gamma = \varGamma_s \, $,  \,and finally
\begin{equation}  \label{eq: x*x_F}
  \begin{gathered}
     \hskip-15pt   \big\langle\, x_{i{}j} \, x_{h{}k} \, , F_r \,\big\rangle  \; = \;
     \big\langle\, x_{i{}j} \otimes x_{h{}k} \, , \Delta(F_r) \big\rangle  \; =   \hfill  \\
     \qquad   = \;  \big\langle\, x_{i{}j} \otimes x_{h{}k} \, , F_r \otimes e^{-\hbar H_r} + 1 \otimes F_r \,\big\rangle  \; =   \hfill  \\
     = \;  \big\langle\, x_{i{}j} \otimes x_{h{}k} \, , F_r \otimes e^{-\hbar H_r} \,\big\rangle \, + \, \big\langle\, x_{i{}j} \otimes x_{h{}k} \, , \, 1 \otimes F_r \,\big\rangle  \; =  \\
     = \;  {(-1)}^{p_{r,r+1} \, p_{h,k}} \big\langle\, x_{i{}j} \, , F_r \,\big\rangle \, \big\langle\, x_{h{}k} \, , e^{-\hbar H_r} \,\big\rangle \, + \, \big\langle\, x_{i{}j} \, , 1 \big\rangle \, \big\langle\, x_{h{}k} \, , F_r \,\big\rangle  \; =  \\
     = \;  {(-1)}^{p_{r,r+1} p_{h,k}} \, \delta_{i-1,\,r,j} \, \delta_{h,\,k} \, e^{-\hbar({(-1)}^{p(r)} \delta_{h,r} - {(-1)}^{p(r+1)} \delta_{h,r+1})} \, + \,
\delta_{i,\,j} \, \delta_{h-1,\,r,k}  \; =  \\
     \hfill   = \;  \delta_{i-1,\,r,j} \, \delta_{h,\,k} \, e^{-\hbar(\delta_{h,r} - {(-1)}^{p_{r,r+1}} \delta_{h,r+1})} \, + \,
     \delta_{i,\,j} \, \delta_{h-1,\,r,k}
  \end{gathered}
\end{equation}
 for  $ \, \gamma = F_r \, $.  From all this we get the following:
 \vskip5pt
   $ \underline{\text{If  $ \, p_{i{}j} = \one \, $}} \, $,  \,then  $ \, i \not= j \, $,  \,hence  \eqref{eq: x*x_E},  \eqref{eq: x*x_G}  and  \eqref{eq: x*x_F}  give
\begin{equation*}
   \big\langle\, x_{i{}j}^{\,2} \, , E_r \,\big\rangle \; = \;  0  \quad ,  \qquad
   \big\langle\, x_{i{}j}^{\,2} \, , \varGamma_s \,\big\rangle \; = \;  0  \quad ,  \qquad
   \big\langle\, x_{i{}j}^{\,2} \, , F_r \,\big\rangle \; = \;  0
\end{equation*}
 so we are done with the generators of  $ \J $  of the form  $ \, \eta = x_{i{}j}^{\,2} \, $  with  $ \, p_{i{}j} = \one \, $.
 \vskip5pt
   $ \underline{\text{If  $ \, j < k \, $}} \, $,  \,then  \eqref{eq: x*x_E}   --- applied twice ---   gives
\begin{equation*}
  \begin{gathered}
     \big\langle\, x_{i{}j} \, x_{i{}k} \, - \, {(-1)}^{p_{i{}j} p_{i{}k}} \, e^{+\hbar \, {(-1)}^{p(i)}} \, x_{i{}k} \, x_{i{}j} \, , E_r \,\big\rangle  \; =   \hfill  \\
     \hskip13pt   = \;  \delta_{i,\,r,j-1} \, \delta_{i,\,k} \, + \, e^{+\hbar({(-1)}^{p(r)} \delta_{i,r} - {(-1)}^{p(r+1)} \delta_{i,r+1})} \, \delta_{i,\,j} \, \delta_{i,k-1,\,r} \, -   \hfill  \\
     \hskip26pt   - \, {(-1)}^{p_{i{}k} p_{i{}j}} \, e^{+\hbar \, {(-1)}^{p(i)}} \big(\, \delta_{i,\,r,k-1} \, \delta_{i,\,j} \, + e^{+\hbar({(-1)}^{p(r)} \delta_{i,r} - {(-1)}^{p(r+1)} \delta_{i,r+1})} \, \delta_{i,\,k} \, \delta_{i,j-1,\,r} \big)  \; =  \\
     \hskip39pt   = \;  \delta_{i,j-1,k,\,r} \, + \, e^{+\hbar \, {(-1)}^{p(i)}} \, \delta_{i,\,j,\,k-1,\,r} \, -   \hfill  \\
     \hfill   - \, {(-1)}^{p_{i{}k} p_{i{}j}} \, e^{+\hbar \, {(-1)}^{p(i)}} \big(\, \delta_{i,\,j,\,k-1,\,r} \, + e^{+\hbar \, {(-1)}^{p(r)}} \, \delta_{i,j-1,\,k,\,r} \big)  \; =  \\
     \hskip39pt   = \;  0 \, + \, e^{+\hbar \, {(-1)}^{p(i)}} \, \delta_{i,\,j,\,k-1,\,r} \, - \, {(-1)}^{p_{i{}k} p_{i{}j}} \, e^{+\hbar \, {(-1)}^{p(i)}} \, \delta_{i,\,j,\,k-1,\,r} \, + \, 0  \; =   \hfill  \\
     \hskip149pt   = \;  \Big(\, 1 \, - \, {(-1)}^\zero \,\Big) \, e^{+\hbar \, {(-1)}^{p(i)}} \, \delta_{i,\,j,\,k-1,\,r}  \,\; = \;\,  0  \\
  \end{gathered}
\end{equation*}
 which yields  $ \; \langle\, \eta \, , \gamma \,\rangle \, = \, 0 \; $  for  $ \, \eta = x_{i,\,j} \, x_{i,\,k} - {(-1)}^{p_{i{}j} p_{i{}k}} \, e^{+\hbar \, {(-1)}^{p(i)}} \, x_{i,\,k} \, x_{i,\,j} \, $  (with  $ \, j<k \, $)  and  $ \, \gamma = E_r \, $.  Similar computations   --- exploiting  \eqref{eq: x*x_G}  and \eqref{eq: x*x_F},  respectively ---   take care of the cases  $ \, \gamma = \varGamma_s \, $  and  $ \, \gamma = F_r \, $  alike.
 \vskip5pt
   $ \underline{\text{If  $ \, i < h \, $}} \, $,
   \,a similar analysis
 to the previous one proves  $ \, \langle\, \eta \, , \gamma \,\rangle = 0 \, $  for  $ \, \eta = x_{i{}j} \, x_{h{}j} - {(-1)}^{p_{i{}j} p_{h{}j}} e^{+\hbar \, {(-1)}^{p(j)}} x_{h{}j} \, x_{i{}j} \, $  (with  $ \, i\!<\!h \, $)  and  $ \, \gamma = E_r \, $,  $ \, \gamma = \varGamma_s \, $  or  $ \, \gamma = F_r \, $.
 \vskip5pt
   $ \underline{\text{If  $ \, i < h \, $  and  $ \, j > k \, $}} \, $,  \,then from  \eqref{eq: x*x_E}  we get
\begin{equation*}
  \begin{gathered}
     \hskip-15pt   \big\langle\, x_{i{}j} \, x_{h{}k} \, - \, {(-1)}^{p_{i{}j} p_{h{}k}} \, x_{h{}k} \, x_{i{}j} \, , E_r \,\big\rangle  \; =   \hfill  \\
     \hskip-6pt   = \;  \delta_{i,\,r,\,j-1} \, \delta_{h,\,k} \, + \, e^{+\hbar \, ({(-1)}^{p(r)} \delta_{i,r} - {(-1)}^{p(r+1)} \delta_{i,r+1})} \, \delta_{i,\,j} \, \delta_{h,\,r,\,k-1} \, -   \hfill  \\
     \hfill   - \, {(-1)}^{p_{i{}j} p_{h{}k}} \, \big(\, \delta_{h,\,r,\,k-1} \, \delta_{i,\,j} \, + \, e^{+\hbar \, ({(-1)}^{p(r)} \delta_{h,r} - {(-1)}^{p(r+1)} \delta_{h,r+1})} \, \delta_{h,\,k} \, \delta_{i,\,r,\,j-1} \,\big)  \; =  \\
     \hskip-6pt   = \;  \delta_{i,\,r,\,j-1} \, \delta_{h,\,k} \, + \, e^{+\hbar \, ({(-1)}^{p(r)} \delta_{i,r} - {(-1)}^{p(r+1)} \delta_{i,r+1})} \, \delta_{i,\,j} \, \delta_{h,\,r,\,k-1} \; -   \hfill  \\
     \hfill   - \; {(-1)}^{p_{i{}j} p_{h{}k}} \, \delta_{h,\,r,\,k-1} \, \delta_{i,\,j} \, - \, {(-1)}^{p_{i{}j} p_{h{}k}} \, e^{+\hbar \, ({(-1)}^{p(r)} \delta_{h,r} - {(-1)}^{p(r+1)} \delta_{h,r+1})} \, \delta_{h,\,k} \, \delta_{i,\,r,\,j-1}  \; =  \\
     \hfill   = \,  \delta_{i,\,r,\,j-1} \, \delta_{h,\,k}  \, + \,  \delta_{i,\,j} \, \delta_{h,\,r,\,k-1}  \, - \,  \delta_{h,\,r,\,k-1} \, \delta_{i,\,j}  \, - \,  \delta_{h,\,k} \, \delta_{i,\,r,\,j-1}  \; = \;  0
  \end{gathered}
\end{equation*}
 which yields  $ \; \langle\, \eta \, , \gamma \,\rangle \, = \, 0 \; $  for  $ \, \eta = x_{i{}j} \, x_{h{}k} \, - \, {(-1)}^{p_{i{}j} p_{h{}k}} \, x_{h{}k} \, x_{i{}j} \, $  (with  $ \, i<h \, $  and  $ \, j>k \, $)  and  $ \, \gamma = E_r \, $.  Then similar calcutations   --- using  \eqref{eq: x*x_G}  and \eqref{eq: x*x_F},  respectively ---   settle the cases  $ \, \gamma = \varGamma_s \, $  and  $ \, \gamma = F_r \, $
   alike.
%
 \vskip5pt
   $ \underline{\text{If  $ \, i < h \, $  and  $ \, j < k \, $}} \, $,  \,then from  \eqref{eq: x*x_E}  again we get
  $$  \displaylines{
   \big\langle\, x_{i{}j} \, x_{h{}k} \, - \, {(-1)}^{p_{i{}j} p_{h{}k}} x_{h{}k} \, x_{i{}j} \, - \, {(-1)}^{p_{i{}j} p_{i{}k} + p(i)} \big( e^{+\hbar} - e^{-\hbar} \big) \, x_{i{}k} \, x_{h{}j} \, , E_r \,\big\rangle  \; =   \hfill  \cr
   \hskip7pt   = \;  \delta_{i,\,r,\,j-1} \, \delta_{h,\,k} \, + \, e^{+\hbar \, ({(-1)}^{p(r)} \delta_{i,r} - {(-1)}^{p(r+1)} \delta_{i,r+1})} \, \delta_{i,\,j} \, \delta_{h,\,k-1,\,r} \, -   \hfill  \cr
   \hskip19pt   - \;  {(-1)}^{p_{i{}j} p_{h{}k}} \, \big(\, \delta_{h,\,r,\,k-1} \, \delta_{i,\,j} \, + \, e^{+\hbar \, ({(-1)}^{p(r)} \delta_{h,r} - {(-1)}^{p(r+1)} \delta_{h,r+1}} \delta_{h,\,k} \, \delta_{i,\,j-1,\,r} \,\big)  \; -   \hfill  \cr
   \hskip31pt   - \;  {(-1)}^{p_{i{}j} p_{i{}k} + p(i)} \big( e^{+\hbar} - e^{-\hbar} \big) \, \delta_{i,\,r,\,k-1} \, \delta_{h,\,j}  \; -   \hfill  \cr
   \hfill   - \;  {(-1)}^{p_{i{}j} p_{i{}k} + p(i)} \big(\, e^{+\hbar} - e^{-\hbar} \,\big) \, e^{+\hbar \, ({(-1)}^{p(r)} \delta_{i,r} - {(-1)}^{p(r+1)} \delta_{i,r+1})} \, \delta_{i,\,k} \, \delta_{h,\,j-1,\,r}  \; =   \qquad  \cr
     \hskip7pt   = \;  \delta_{i,\,r,\,j-1} \, \delta_{h,\,k} \, + \, \delta_{i,\,j} \, \delta_{h,\,k-1,\,r} \, - \, {(-1)}^{p_{i{}j} p_{h{}k}} \, \big(\, \delta_{h,\,r,\,k-1} \, \delta_{i,\,j} \, + \, \delta_{h,\,k} \, \delta_{i,\,j-1,\,r} \,\big)  \; -   \hfill  \cr
     \hfill   - \;  \big(\, e^{+\hbar} - e^{-\hbar} \,\big) \, {(-1)}^{p_{i{}j} p_{i{}k} + p(i)} \big(\, \delta_{i,\,r,\,k-1} \, \delta_{h,\,j} \, - \, \delta_{i,\,k} \, \delta_{h,\,j-1,\,r} \,\big)  \; =   \qquad \cr
     \hskip0pt   = \;  \delta_{i,\,r,\,j-1} \, \delta_{h,\,k} \, + \, \delta_{i,\,j} \, \delta_{h,\,k-1,\,r} \, - \, {(-1)}^\zero \, \delta_{h,\,r,\,k-1} \, \delta_{i,\,j} \, - \, {(-1)}^\zero \, \delta_{h,\,k} \, \delta_{i,\,j-1,\,r}  \; -   \hfill  \cr
     \hfill   - \big( e^{+\hbar} \! - e^{-\hbar} \,\big) {(-1)}^{p_{i{}j} p_{i{}k} + p(i)} \, \delta_{i,\,r,\,k-1} \, \delta_{h,\,j} \, + \big( e^{+\hbar} - e^{-\hbar} \,\big) {(-1)}^{p_{i{}j} p_{i{}k} + p(i)} \, \delta_{i,\,k} \, \delta_{h,\,j-1,\,r}  \, = \,  0  }  $$
 where the last identity follows because we have found an element (looking at all possible values of the Kronecker  $ \delta_{q,\,p} $'s)  of the form
  $ \; A + B - A - B - C - D \; $,
 \;where the last two summands are  $ \, C = 0 = D \, $  since the various restrictions eventually force  $ \, i < h = j < k = i+1 \, $  and  $ \, j < k = i < h = j+1 \, $,  \,which both are impossible.  Thus for  $ \, \eta = x_{i{}j} \, x_{h{}k} \, - \, {(-1)}^{p_{i{}j} p_{h{}k}} x_{h{}k} \, x_{i{}j} \, - \, {(-1)}^{p_{i{}j} + p_{i{}k} + p(i)} \big( e^{+\hbar} - e^{-\hbar} \big) \, x_{i{}k} \, x_{h{}j} \, $  (with  $ \, i<h \, $  and  $ \, j<k \, $)  and  $ \, \gamma = E_r \, $   we get  $ \; \langle\, \eta \, , \gamma \,\rangle \, = \, 0 \; $,  \,as expected.
                                                                     \par
   Parallel calculations, using  \eqref{eq: x*x_G}  and \eqref{eq: x*x_F},
   work
 for  $ \, \gamma = \varGamma_s \, $  and  $ \, \gamma = F_r \, $  too.
 \vskip9pt
   \textit{$ \underline{\text{Proof of (5--b)}} $:}\;  We prove that  $ \, \Delta(\eta\hskip0,5pt) \in \J \otimes \tilde{\mathbb{F}}_\hbar + \tilde{\mathbb{F}}_\hbar \otimes \J \, $  by explicitly computing  $ \, \Delta(\eta) \, $,  for every generator  $ \eta $  of  $ \J $  as in  \eqref{eq: gen.'s-ideal-J-in-Fh}.
 \vskip7pt
   For  $ \; \eta = x_{i{}j}^{\,2} \; $   --- with  $ \, p_{i{}j} = \one \, $  ---   computations give
  $$  \displaylines{
   \Delta\big( x_{i{}j}^{\,2} \big)  \; = \;  {\Delta(x_{i{}j})}^2  \; = \;  {\Big(\, {\textstyle \sum_{t=1}^n} \, {(-1)}^{p_{i{}t} \, p_{t{}j}} x_{i{}t} \otimes x_{t{}j} \Big)}^{\!2}  \; =   \hfill  \cr
   \qquad   = \;  {\textstyle \sum_{h,k=1}^n} \, {(-1)}^{p_{i{}h} \, p_{h{}j}} {(-1)}^{p_{i{}k} \, p_{k{}j}} \big( x_{i{}h} \otimes x_{h{}j} \big) \big( x_{i{}k} \otimes x_{k{}j} \big)  \; =   \hfill  \cr
   \qquad \qquad   = \;  {\textstyle \sum_{h,k=1}^n} \, {(-1)}^{p_{i{}h} \, p_{h{}j}} {(-1)}^{p_{i{}k} \, p_{k{}j}} {(-1)}^{p_{i{}k} \, p_{h{}j}} \, x_{i{}h} \, x_{i{}k} \otimes x_{h{}j} \, x_{k{}j}  \; =   \hfill \cr
   \quad   = \;  {\textstyle \sum_{\ell=1}^n} \, {(-1)}^{p_{i{}\ell} \, p_{\ell{}j}} \, x_{i{}\ell}^{\,2} \otimes x_{\ell{}j}^{\,2}  \; +   \hfill  \cr
   \quad \quad \qquad \qquad   + \;  {\textstyle \sum_{\substack{h,k=1  \\  h<k}}^n} \, \Big( {(-1)}^{p_{i{}h} \, p_{h{}j} \, + \, p_{i{}k} \, p_{k{}j} \, + \, p_{i{}k} \, p_{h{}j}} \cdot x_{i{}h} \, x_{i{}k} \otimes x_{h{}j} \, x_{k{}j}  \; +   \hfill  \cr
   \quad \quad \quad \qquad \qquad \qquad   + \;  {(-1)}^{p_{i{}k} \, p_{k{}j} \, + \, p_{i{}h} \, p_{h{}j} \, + \, p_{i{}h} \, p_{k{}j}} \cdot x_{i{}k} \, x_{i{}h} \otimes x_{k{}j} \, x_{h{}j} \Big)  \; =   \hfill  \cr
   \quad   = \;  {\textstyle \sum_{\ell=1}^n} \, {(-1)}^{p_{i{}\ell} \, p_{\ell{}j}} \, x_{i{}\ell}^{\,2} \otimes x_{\ell{}j}^{\,2}  \; +   \hfill  \cr
   \quad \quad \qquad \qquad   + \;  {\textstyle \sum_{\substack{h,k=1  \\  h<k}}^n}
\, {(-1)}^{p_{i{}h} \, p_{h{}j} + p_{i{}k} \, p_{k{}j}} \,
\Big( {(-1)}^{p_{i{}k} \, p_{h{}j}} \, x_{i{}h} \, x_{i{}k} \otimes x_{h{}j} \, x_{k{}j}  \; +   \hfill  \cr
   \quad \quad \quad \qquad \qquad \qquad   + \;  {(-1)}^{p_{i{}h} \, p_{k{}j}} \, x_{i{}k} \, x_{i{}h} \otimes x_{k{}j} \, x_{h{}j} \Big)  \; =   \hfill  \cr
   \quad   = \;  {\textstyle \sum_{\ell=1}^n} \, {(-1)}^{p_{i{}\ell} \, p_{\ell{}j}} \, x_{i{}\ell}^{\,2} \otimes x_{\ell{}j}^{\,2}  \; + \;  {\textstyle \sum_{\substack{h,k=1  \\  h<k}}^n}
\, {(-1)}^{p_{i{}h} \, p_{h{}j} + p_{i{}k} \, p_{k{}j}} \, \times   \hfill  \cr
   \hfill   \times \, \Big( {(-1)}^{p_{i{}k} \, p_{h{}j}} \, x_{i{}h} \, x_{i{}k} \otimes x_{h{}j} \, x_{k{}j}  \; + \;  {(-1)}^{p_{i{}h} p_{k{}j}} \, x_{i{}k} \, x_{i{}h} \otimes x_{k{}j} \, x_{h{}j} \Big)  }  $$
 where we took into account that  $ \, p_{i,j} = \one \, $  means that either  $ \, p(i) = \zero \, $  or  $ \, p(j) = \zero \, $,  hence  $ \, p_{i{}t} = \zero \, $  or  $ \, p_{t{}j} = \zero \, $  for all  $ t \, $.
 \vskip5pt
   Now let us consider the summands occurring in the last two sums above:
 \vskip3pt
   --- \  in the first sum, each summand  $ \, {(-1)}^{p_{i{}\ell} \, p_{\ell{}j}} \, x_{i{}\ell}^{\,2} \otimes x_{\ell{}j}^{\,2} \, $  obviously belongs to  $ \, \J \otimes \tilde{\mathbb{F}}_\hbar + \tilde{\mathbb{F}}_\hbar \otimes \J \, $,  \,because  $ \, p_{i,j} = \one \, $  implies that either  $ \, p_{i{}\ell} = \one \, $  or  $ \, p_{\ell{}j} = \one \, $  and accordingly it is  $ \, x_{i{}\ell}^{\,2} \in \J \, $  or  $ \, x_{\ell{}j}^{\,2} \in \J \, $;  hence we are done.
 \vskip3pt
   --- \  in the second sum, we re-write each tensor factor  $ \, x_{i{}h} \, x_{i{}k} \, $  and  $ \, x_{h{}j} \, x_{k{}j} \, $  as
%
%
  $$  x_{i{}h} x_{i{}k}  \; = \;  \hat\gamma_{h{}k}^{(i)} + \hat{c}_{h{}k}^{(i)} \, x_{i{}k} \, x_{i{}h} \quad ,   \qquad   x_{h{}j} x_{k{}j}  \; = \;  \check\gamma_{h{}k}^{(j)} + \check{c}_{h{}k}^{(i)} \, x_{k{}j} \, x_{h{}j}  $$
 where we set  $ \; \hat\gamma_{h{}k}^{(i)} \, := \, x_{ih} \, x_{ik} - \hat{c}_{h{}k}^{(i)} \, x_{i{}k} \, x_{i{}h} \, \in \J \; $  with  $ \; \hat{c}_{h{}k}^{(i)} := {(-1)}^{p_{i{}h} p_{i{}k}} e^{+\hbar \, {(-1)}^{p(i)}} \; $  and similarly also  $ \; \check\gamma_{h{}k}^{(j)} \, := \, x_{h{}j} \, x_{k{}j} - \check{c}_{h{}k}^{(i)} \, x_{k{}j} \, x_{h{}j} \, \in \J \; $  with  $ \; \check{c}_{h{}k}^{(i)} := {(-1)}^{p_{h{}j} p_{k{}j}} e^{+\hbar \, {(-1)}^{p(j)}} \; $   --- cf.\  \eqref{eq: gen.'s-ideal-J-in-Fh}.  Then the contribution to the  $ (h,k) $--th  summand in the second sum that occurs in the last line above can be re-written as
  $$  \displaylines{
   {(-1)}^{p_{i{}k} \, p_{h{}j}} \, \Big(\, \hat\gamma_{h{}k}^{(i)} \otimes \check\gamma_{h{}k}^{(j)}  \; + \;  \hat\gamma_{h{}k}^{(i)} \otimes \check{c}_{h{}k}^{(i)} \, x_{k{}j} \, x_{h{}j}  \; + \;  \hat{c}_{h{}k}^{(i)} \, x_{i{}h} \, x_{i{}k} \otimes \check\gamma_{h{}k}^{(j)} \,\Big)  \; +   \hfill  \cr
%
%
   \hfill   + \;  \, {(-1)}^{p_{i{}k} \, p_{h{}j}} \, \hat{c}_{h{}k}^{(i)} \, \check{c}_{h{}k}^{(i)} \; x_{i{}k} \, x_{i{}h} \otimes x_{k{}j} \, x_{h{}j}  \; + \;  {(-1)}^{p_{i{}h} p_{k{}j}} \, x_{i{}k} \, x_{i{}h} \otimes x_{k{}j} \, x_{h{}j}  }  $$
 with the first line obviously providing an element in  $ \, \J \otimes \tilde{\mathbb{F}}_\hbar + \tilde{\mathbb{F}}_\hbar \otimes \J \, $,  while the second
%
%
 provides
 to the overall sum the following contribution
  $$  \displaylines{
   {(-1)}^{p_{i{}k} \, p_{h{}j} \, + \,p_{i{}h} p_{i{}k} \, + \, p_{h{}j} p_{k{}j}} e^{+\hbar \, ({(-1)}^{p(i)} + {(-1)}^{p(j)})} \, x_{i{}k} \, x_{i{}h} \otimes x_{k{}j} \, x_{h{}j}  \; +   \hfill  \cr
   \hfill   + \;  {(-1)}^{p_{i{}h} p_{k{}j}} \, x_{i{}k} \, x_{i{}h} \otimes x_{k{}j} \, x_{h{}j}  \; =   \quad  \cr
   \hfill   = \;  {(-1)}^{p_{i{}h} \, p_{k{}j}} \big( {(-1)}^{p_{i{}h} \, p_{k{}j} \, + \, p_{i{}k} \, p_{h{}j} \, + \, p_{i{}h} p_{i{}k} \, + \, p_{h{}j} p_{k{}j}} \, + \, 1 \,\big) \, x_{i{}k} \, x_{i{}h} \otimes x_{k{}j} \, x_{h{}j}  \; = \;  0  }  $$
 where we exploited the facts that  $ \, p_{i{}j} \, = \one \, $  implies both  $ \; {(-1)}^{p(i)} + {(-1)}^{p(j)} = \zero \; $  and  $ \, p_{i{}h} \, p_{k{}j} \, + \, p_{i{}k} \, p_{h{}j} \, + \, p_{i{}h} \, p_{i{}k} \, + \, p_{h{}j} \, p_{k{}j} \, = \, p_{i{}j} \, = \, \one \, $.  Tiding everything up, we find that the second sum also belongs  $ \, \J \otimes \tilde{\mathbb{F}}_\hbar + \tilde{\mathbb{F}}_\hbar \otimes \J \, $,  \,and we are done.
 \vskip5pt
   For  $ \; \eta \, = \, x_{i{}j} \, x_{i{}k} - {(-1)}^{p_{i{}j} p_{i{}k}} \, e^{+\hbar \, {(-1)}^{p(i)}} x_{i{}k} \, x_{i{}j} \; $   --- with  $ \, j < k \, $  ---   we compute
  $$  \displaylines{
   \Delta(\eta)  \; = \;  \Delta\big( x_{i{}j} \, x_{i{}k} - {(-1)}^{p_{i{}j} p_{i{}k}} \, e^{+\hbar \, {(-1)}^{p(i)}} x_{i{}k} \, x_{i{}j} \big)  \; =   \hfill  \cr
   \hskip15pt   = \;  {\bigg(\, {\textstyle \sum\limits_{t=1}^n} {(-1)}^{p_{i{}t} \, p_{t{}j}} x_{i{}t} \otimes x_{t{}j} \bigg)} {\bigg(\, {\textstyle \sum\limits_{\ell=1}^n} {(-1)}^{p_{i{}\ell} \, p_{\ell{}k}} x_{i{}\ell} \otimes x_{\ell{}k} \bigg)}  \; -   \hfill  \cr
   \hfill   - \;  {(-1)}^{p_{i{}j} p_{i{}k}} \, e^{+\hbar \, {(-1)}^{p(i)}} {\bigg(\, {\textstyle \sum\limits_{t=1}^n} {(-1)}^{p_{i{}t} \, p_{t{}k}} x_{i{}t} \otimes x_{t{}k} \bigg)} {\bigg(\, {\textstyle \sum\limits_{\ell=1}^n} {(-1)}^{p_{i{}\ell} \, p_{\ell{}j}} x_{i{}\ell} \otimes x_{\ell{}j} \bigg)}  \; =  \cr
   = \;  {\textstyle \sum_{h=1}^n} \, A'_h  \, + \,  {\textstyle \sum_{\substack{t,\ell=1  \\  t<\ell}}^n} \, \big( B'_{t,\ell} + C'_{t,\ell} \big)  \, + \,  {\textstyle \sum_{h=1}^n} \, A''_h  \, + \,  {\textstyle \sum_{\substack{t,\ell=1  \\  t<\ell}}^n} \, \big( B''_{t,\ell} + C''_{t,\ell} \big)  \; =  \cr
   \hfill   = \;  {\textstyle \sum_{h=1}^n} \big( A'_h + A''_h \big)  \, + \,  {\textstyle \sum_{\substack{t,\ell=1  \\  t<\ell}}^n} \, \big( B'_{t,\ell} + C'_{t,\ell} + B''_{t,\ell} + C''_{t,\ell} \big)  }  $$
 where we used notation (for all  $ \, h, t, \ell \in I_n := \{1,\dots,n\} \, $  with  $ \, t < \ell \, $)
\begin{align*}
   A'_h \,  &  \, := \,  {(-1)}^{p_{i{}h} p_{h{}k}} \, x_{i{}h}^{\,2} \otimes x_{h{}j} \, x_{h{}k}  \\
   A''_h \,  &  \, := \,  -{(-1)}^{p_{i{}j} p_{i{}k}} e^{\hbar \, {(-1)}^{p(i)}} {(-1)}^{p_{i{}h} p_{h{}j}} \, x_{i{}h}^{\,2} \otimes x_{h{}k} \, x_{h{}j}  \\
   B'_{t,\ell}  &  \, := \,  {(-1)}^{p_{i{}t} p_{t{}j} + p_{i{}\ell} p_{\ell{}k} + p_{i{}\ell} p_{t{}j}} \, x_{i{}t} \, x_{i{}\ell} \otimes x_{t{}j} \, x_{\ell{}k}  \\
   B''_{t,\ell}  &  \, := \,  -{(-1)}^{p_{i{}j} p_{i{}k}} e^{\hbar \, {(-1)}^{p(i)}} {(-1)}^{p_{i{}t} p_{t{}k} + p_{i{}\ell} p_{\ell{}j} + p_{i{}\ell} p_{t{}k}} \, x_{i{}t} \, x_{i{}\ell} \otimes x_{t{}k} \, x_{\ell{}j}   \\
   C'_{t,\ell}  &  \, := \,  {(-1)}^{p_{i{}\ell} p_{\ell,j} + p_{i{}t} p_{t{}k} + p_{i{}t} p_{\ell{}j}} \, x_{i{}\ell} \, x_{i{}t} \otimes x_{\ell{}j} \, x_{t{}k}   \\
   C''_{t,\ell}  &  \, := \,  -{(-1)}^{p_{i{}j} p_{i{}k}} e^{\hbar \, {(-1)}^{p(i)}} {(-1)}^{p_{i{}\ell} p_{\ell{}k} + p_{i{}t} p_{t{}j} + p_{i{}t} p_{\ell{}k}} \, x_{i{}\ell} \, x_{i{}t} \otimes x_{\ell{}k} \, x_{t{}j}
\end{align*}
   \indent   Now, note that  $ \; A'_h , A''_h \in \J \otimes \tilde{\mathbb{F}}_\hbar + \tilde{\mathbb{F}}_\hbar \otimes \J \; $  (for all  $ h \, $)  so that the first sum in the expansion of  $ \Delta(\eta) $   --- namely  $ \, {\textstyle \sum_{h=1}^{n+1}} \big( A'_h + A''_h \big) \, $  ---   belongs to  $ \, \J \otimes \tilde{\mathbb{F}}_\hbar + \tilde{\mathbb{F}}_\hbar \otimes \J \, $.
                                                                \par
   As to the second sum, we fix more notation.  Given  $ \, X, Y \in \tilde{\mathbb{F}}_\hbar \, $,  let us write  $ \, X \equiv_\J Y \, $  if  $ X $  is equivalent to  $ Y $  modulo  $ \J $  (in  $  \tilde{\mathbb{F}}_\hbar \, $),  Similarly, for any  $ \, X_\otimes \, , Y_\otimes \in \tilde{\mathbb{F}}_\hbar^{\,\otimes 2} \, $,  we write  $ \, X_\otimes \equiv_{\J_\otimes} Y_\otimes \, $  if  $ X_\otimes $  is equivalent to  $ Y_\otimes $  modulo  $ \, \J_\otimes := \J \otimes \tilde{\mathbb{F}}_\hbar + \tilde{\mathbb{F}}_\hbar \otimes \J \, $.  Now, definitions give
 $ \,\; x_{i{}t} \, x_{i{}\ell} \, \equiv_\J {(-1)}^{p_{i{}t} p_{i{}\ell}} e^{\hbar \, {(-1)}^{p(i)}} x_{i{}\ell} \, x_{i{}t} \;\, $
 (for  $ \, t < \ell \, $),  which implies
\begin{align*}  \label{eq: B'[t,l] rewritten}
 B'_{t,\ell}  &  \,\; \equiv_{\J_\otimes}  \, x_{i{}\ell} \, x_{i{}t} \otimes\! \big( {(-1)}^{p_{i{}t} p_{i{}\ell}} e^{\hbar \, {(-1)}^{p(i)}}
{(-1)}^{p_{i{}t} p_{t{}j} + p_{i{}\ell} p_{\ell{}k} + p_{i{}\ell} p_{t{}j}} \, x_{t{}j} \, x_{\ell{}k} \big)  \\
 B''_{t,\ell}  &  \,\; \equiv_{\J_\otimes}  \, x_{i{}\ell} \, x_{i{}t} \otimes\! \big(\! - \! {(-1)}^{p_{i{}t} p_{i{}\ell}} e^{2 \hbar \, {(-1)}^{p(i)}} {(-1)}^{p_{i{}j} p_{i{}k}} {(-1)}^{p_{i{}t} p_{t{}k} + p_{i{}\ell} p_{\ell{}j} + p_{i{}\ell} p_{t{}k}} \, x_{t{}k} \, x_{\ell{}j} \big)
\end{align*}
 Bundling everything up we find, in the second sum in the expansion of  $ \Delta(\eta) $,  that
\begin{equation}  \label{eq: B'+C'+B"+C" => b'+c'+b"+c"}
  B'_{t,\ell} + C'_{t,\ell} + B''_{t,\ell} + C''_{t,\ell}  \;\; \equiv_{\J_\otimes} \;  x_{i{}\ell} \, x_{i{}t} \,\otimes \big(\, \text{b}'_{t,\ell} + \text{c}'_{t,\ell} + \text{b}''_{t,\ell} + \text{c}''_{t,\ell} \,\big)
\end{equation}
 where we adopted notation
\begin{align*}
   \text{b}'_{t,\ell}  &  \,\; := \;\,  {(-1)}^{p_{i{}t} p_{i{}\ell} + p_{i{}t} p_{t{}j} + p_{i{}\ell} p_{\ell{}k} + p_{i{}\ell} p_{t{}j}} \, e^{\hbar \, {(-1)}^{p(i)}} \, x_{t{}j} \, x_{\ell{}k}  \\
   \text{b}''_{t,\ell}  &  \,\; := \;\,  - {(-1)}^{p_{i{}t} p_{i{}\ell} + p_{i{}j} p_{i{}k} + p_{i{}t} p_{t{}k} + p_{i{}\ell} p_{\ell{}j} + p_{i{}\ell} p_{t{}k}} \, e^{2 \, \hbar \, {(-1)}^{p(i)}} \, x_{t{}k} \, x_{\ell{}j}   \\
   \text{c}'_{t,\ell}  &  \,\; := \;\,  {(-1)}^{p_{i{}\ell} p_{\ell{}j} + p_{i{}t} p_{t{}k} + p_{i{}t} p_{\ell{}j}} \, x_{\ell{}j} \, x_{t{}k}  \,\; \equiv_\J \;  {(-1)}^{p_{i{}\ell} p_{\ell{}j} + p_{i{}t} p_{t{}k} + p_{i{}t} p_{\ell{}j} + p_{t{}k} p_{\ell{}j}} \, x_{t{}k} \, x_{\ell{}j}   \\
   \text{c}''_{t,\ell}  &  \,\; := \;\,  -{(-1)}^{p_{i{}j} p_{i{}k} + p_{i{}\ell} p_{\ell{}k} + p_{i{}t} p_{t{}j} + p_{i{}t} p_{\ell{}k}} \, e^{\hbar \, {(-1)}^{p(i)}} \, x_{\ell{}k} \, x_{t{}j}
\end{align*}
 where in third line we used the relation  $ \; x_{\ell{}j} \, x_{t{}k} \, \equiv_\J {(-1)}^{p_{t{}k} p_{\ell{}j}} \, x_{t{}k} \, x_{\ell{}j} \; $.  Also, we have
  $$  x_{t{}j} \, x_{\ell{}k}  \,\; \equiv_\J \;  {(-1)}^{p_{t{}j} p_{\ell{}k}} \, x_{\ell{}k} \, x_{t{}j} \, + \, {(-1)}^{p_{t{}j} p_{t{}k} + p(t)} \big(\, e^{+\hbar} - e^{-\hbar} \big) \, x_{t{}k} \, x_{\ell{}j}  $$
 Using this to re-write  $ \text{b}'_{t,\ell} \, $  along with the formulas, we find that
\begin{equation}  \label{eq: formula x b'+c'+b"+c"}
  \text{b}'_{t,\ell} + \text{c}'_{t,\ell} + \text{b}''_{t,\ell} + \text{c}''_{t,\ell}  \,\; = \;\,  \kappa'_{t,\ell} \, x_{\ell{}k} \, x_{t{}j}  \; + \;  \kappa''_{t,\ell} \, x_{t{}k} \, x_{\ell{}j}
\end{equation}
 where the coefficients are given by
  $$  \displaylines{
   \kappa'_{t,\ell}  \,\; := \;\,  e^{\hbar \, {(-1)}^{p(i)}} {(-1)}^{p_{i{}t} p_{t{}j} + p_{i{}\ell} p_{\ell{}k}} \Big( {(-1)}^{p_{i{}t} p_{i{}\ell} + p_{i,\ell} p_{t,j} + p_{t{}j} p_{\ell{}k}} \, - \, {(-1)}^{p_{i{}j} p_{i{}k} + p_{i{}t} p_{\ell{}k}} \Big)   \hfill  \cr
   \kappa''_{t,\ell}  \,\; := \;\,  {(-1)}^{p_{i{}t} p_{i{}\ell} + p_{i{}t} p_{t{}j} + p_{i{}\ell} p_{\ell{}k} + p_{i{}\ell} p_{t{}j} + p_{t{}j} p_{t{}k} + p(t)} e^{\hbar \, {(-1)}^{p(i)}} \big(\, e^{+\hbar} - e^{-\hbar} \,\big)  \; +   \hfill  \cr
   \hfill   + \;  {(-1)}^{p_{i{}\ell} p_{\ell{}j} + p_{i{}t} p_{t{}k} + p_{i{}t} p_{\ell{}j} + p_{t{}k} p_{\ell{}j}}  \; - \;  {(-1)}^{p_{i{}j} p_{i{}k} + p_{i{}t} p_{t{}k} + p_{i{}\ell} p_{\ell{}j} + p_{i{}\ell} p_{t{}k} + p_{i{}t} p_{i{}\ell}} \, e^{2 \, \hbar \, {(-1)}^{p(i)}}  }  $$
 Finally, an in-depth (yet trivial!) inspection shows that  $ \; \kappa'_{t,\ell} = 0 \; $  and  $ \; \kappa''_{t,\ell} = 0 \; $.  This along with  \eqref{eq: formula x b'+c'+b"+c"}  and  \eqref{eq: B'+C'+B"+C" => b'+c'+b"+c"}  yields
 $ \; B'_{t,\ell} + C'_{t,\ell} + B''_{t,\ell} + C''_{t,\ell} \, \in \, \J \otimes \tilde{\mathbb{F}}_\hbar + \tilde{\mathbb{F}}_\hbar \otimes \J \; $
 (for all  $ \, t < \ell \, $),  whence we conclude that  $ \; \Delta(\eta) \, \in \, \J \otimes \tilde{\mathbb{F}}_\hbar + \tilde{\mathbb{F}}_\hbar \otimes \J \, =: \, \J_\otimes \, $,  \;q.e.d.
 \vskip5pt
   For  $ \; \eta \, = \, x_{ij} \, x_{hj} - {(-1)}^{p_{i,j} p_{h,j}} \, e^{+\hbar \, {(-1)}^{p(j)}} x_{hj} \, x_{ij} \; $   --- with  $ \, i < h \, $  ---   we get  $ \; \Delta(\eta) \, \in \, \J \otimes \tilde{\mathbb{F}}_\hbar + \tilde{\mathbb{F}}_\hbar \otimes \J \, =: \, \J_\otimes \; $  with a parallel analysis to the
   above one.
%
 \vskip5pt
   Finally, for  $ \; \eta \, = \, x_{i{}j} \, x_{h{}k} \, - {(-1)}^{p_{i{}j} p_{h{}k}} x_{h{}k} \, x_{i{}j} \, - {(-1)}^{{p_{i{}j} p_{i{}k}} + p(i)} \big( e^{+\hbar} - e^{-\hbar} \,\big) \, x_{i{}k} \, x_{h{}j} \; $  our computations go as follows.  First of all, acting as before we find
  $$  \displaylines{
   \quad   \Delta\big( x_{i{}j} \, x_{h{}k} \big)  \; = \;  {\textstyle \sum_{t=1}^n} \, \sigma_{i,j}^{h,k}(t) \, x_{i{}t} \, x_{h{}t} \otimes x_{t{}j} \, x_{t{}k} \, +   \hfill  \cr
   \hfill   + \, {\textstyle \sum_{\substack{b,d=1  \\  b<d}}^n} \, \alpha_{i,j}^{h,k}(b,d) \, x_{i{}b} \, x_{h{}d} \otimes x_{b{}j} \, x_{d{}k} \, + \, {\textstyle \sum_{\substack{b,d=1  \\  b<d}}^n} \, \omega_{i,j}^{h,k}(b,d) \, x_{i{}d} \, x_{h{}b} \otimes x_{d{}j} \, x_{b{}k}   \quad  }  $$
 where  $ \; \sigma_{i,j}^{h,k}(t) \, = \, {(-1)}^{p_{i{}t} p_{t{}j} + p_{h{}t} p_{t{}k} + p_{h{}t} p_{t{}j}} \; $,  $ \; \alpha_{i,j}^{h,k}(b,d) \, = \, {(-1)}^{p_{i{}b} p_{b{}j} + p_{h{}d} p_{d{}k} + p_{h{}d} p_{b{}j}} \; $  and  $ \; \omega_{i,j}^{h,k}(b,d) \, = \, {(-1)}^{p_{i{}d} p_{d{}j} + p_{h{}b} p_{b{}k} + p_{h{}b} p_{d{}j}} \; $   --- so that  $ \, \alpha_{i,j}^{h,k}(t,t) = \sigma_{i,j}^{h,k}(t) = \omega_{i,j}^{h,k}(t,t) \, $  ---   and then we have also parallel formulas for  $ \, \Delta\big( x_{h{}k} \, x_{i{}j} \big) \, $  and  $ \, \Delta\big( x_{i{}k} \, x_{h{}j} \big) \, $  as well, just by switching indices in the correct way.  It follows that
  $$  \displaylines{
   \Delta(\eta)  \; = \;  \Delta\Big( x_{i{}j} \, x_{h{}k} \, + A_{i,j}^{h,k} \, x_{h{}k} \, x_{i{}j} \, + \varOmega_{i,j}^{h,k} \, x_{i{}k} \, x_{h{}j} \Big)  \; =   \hfill  \cr
   \quad   = \;  {\textstyle \sum_{t=1}^n} \, \sigma_{i,j}^{h,k}(t) \, x_{i{}t} \, x_{h{}t} \otimes x_{t{}j} \, x_{t{}k} \; +   \hfill  \cr
   \hfill   + \; {\textstyle \sum_{b<d}} \, \big(\, \alpha_{i,j}^{h,k}(b,d) \, x_{i{}b} \, x_{h{}d} \otimes x_{b{}j} \, x_{d{}k} \, + \, \omega_{i,j}^{h,k}(b,d) \, x_{i{}d} \, x_{h{}b} \otimes x_{d{}j} \, x_{b{}k} \,\big) \; +   \quad  \cr
   \quad \quad   + \; A_{i,j}^{h,k} \, {\textstyle \sum_{t=1}^n} \, \sigma_{h,k}^{i,j}(t) \, x_{h{}t} \, x_{i{}t} \otimes x_{t{}k} \, x_{t{}j} \; +   \hfill  \cr
   \hfill   + \; A_{i,j}^{h,k} \, {\textstyle \sum_{b<d}} \, \big(\, \alpha_{h,k}^{i,j}(b,d) \, x_{h{}b} \, x_{i{}d} \otimes x_{b{}k} \, x_{d{}j} \, + \, \omega_{h,k}^{i,j}(b,d) \, x_{h{}d} \, x_{i{}b} \otimes x_{d{}k} \, x_{b{}j} \,\big) \; +   \quad  \cr
   \quad \quad \quad   + \; \varOmega_{i,j}^{h,k} \, {\textstyle \sum_{t=1}^n} \, \sigma_{i,k}^{h,j}(t) \, x_{i{}t} \, x_{h{}t} \otimes x_{t{}k} \, x_{t{}j} \; +   \hfill  \cr
   \hfill   + \; \varOmega_{i,j}^{h,k} \, {\textstyle \sum_{b<d}} \, \big(\, \alpha_{i,k}^{h,j}(b,d) \, x_{i{}b} \, x_{h{}d} \otimes x_{b{}k} \, x_{d{}j} \, + \, \omega_{i,k}^{h,j}(b,d) \, x_{i{}d} \, x_{h{}b} \otimes x_{d{}k} \, x_{b{}j} \,\big)  \; \equiv_{\J_\otimes}  \cr
   \hfill   \equiv_{\J_\otimes} \;   {\textstyle \sum_{t=1}^n} \, C_{i,j}^{h,k}(t)  \, + \, {\textstyle \sum_{b<d}} \, D_{i,j}^{h,k}(b,d)  }  $$
 with  $ \; A_{i,j}^{h,k} := {(-1)}^{\one + p_{i{}j} p_{h{}k}} \; $  and  $ \; \varOmega_{i,j}^{h,k} := {(-1)}^{\one + p_{i{}j} p_{i{}k} + p(i)} \big( e^{+\hbar} - e^{-\hbar} \,\big) \; $.
   \hfill   In particular,
  \eject
\noindent
 $ \Delta(\eta) $  is a linear combination of homogeneous tensors  $ \; x_{r_1,\,s_1} \, x_{\ell_1,\,c_1} \otimes x_{r_2,\,s_2} \, x_{\ell_2,\,c_2} \; $.  Now, by the very definition of  $ \J $   --- cf.\ \eqref{eq: gen.'s-ideal-J-in-Fh}  we have
\begin{equation}  \label{eq: gen.'s-ideal-J-rewritten}
  \begin{gathered}
   \hskip1pt   x_{r{}s} \, x_{r{}c}  \; \equiv_\J \;  {(-1)}^{p_{r{}s} p_{r{}c}} \, e^{+\hbar \, {(-1)}^{p(r)}} \, x_{r{}c} \, x_{r{}s}  \hskip45pt \qquad  \text{if \ \ } s < c   \qquad  \cr
   \hskip1pt   x_{r{}s} \, x_{\ell\,{}s}  \; \equiv_\J \;  {(-1)}^{p_{r{}s} p_{\ell{}s}} \, e^{+\hbar \, {(-1)}^{p(s)}} \, x_{\ell\,{}s} \, x_{r{}s}  \hskip45pt \qquad  \text{if \ \ } r < \ell   \qquad  \cr
   \qquad \quad   x_{r{}s} \, x_{\ell\,{}c}  \; \equiv_\J \;  {(-1)}^{p_{\ell{}c} p_{r{}s}} x_{\ell\,{}c} \, x_{r{}s}  \qquad \hskip31pt \qquad \  \text{if \ \ }  r < \ell \, , \; s > c   \hskip15pt  \cr
   \hskip11pt   x_{r{}s} \, x_{\ell\,{}c}  \; \equiv_\J \,  A_{r,s}^{\ell,c} \; x_{\ell\,{}c} \, x_{r{}s} \, + \, \varOmega_{r,s}^{\ell,c} \; x_{r{}c} \, x_{\ell\,{}s}   \hskip55pt   \text{if \ \ }  r < \ell \, , \; s < c   \hskip11pt
  \end{gathered}
\end{equation}
 Then, using the first two relations we can re-write each  $ t $--th  summand in the three sums above of the form  ``$ \, \sum_{t=1}^n \, $''  as a suitable multiple of the single homogeneous tensor  $ \,  x_{h{}t} \, x_{i{}t} \otimes x_{t{}k} \, x_{t{}j} \, $:  \,therefore, from this we find that in the relation
\begin{equation}   \label{eq: last-expr x Delta(eta)}
  \Delta(\eta)  \,\; \equiv_{\J_\otimes} \;\,  {\textstyle \sum_{t=1}^n}\, C_{i,j}^{h,k}(t) \, + \, {\textstyle \sum_{b<d}}\, D_{i,j}^{h,k}(b,d)
\end{equation}
 found above the first sum is eventually boils down to
  $$  C_{i,j}^{h,k}(t)  \,\; = \;\,  \kappa_{i,j}^{h,k}(t) \, x_{h{}t} \, x_{i{}t} \otimes x_{t{}k} \, x_{t{}j}  $$
 where the coefficient  $ \, \kappa_{i,j}^{h,k}(t) \, $  is explicitly given   --- out of direct computation ---   by
  $$  \displaylines{
   \qquad   \kappa_{i,j}^{h,k}(t)  \,\; = \;\,  \sigma_{i,j}^{h,k}(t) \, {(-1)}^{p_{i{}t} p_{h{}t} + p_{t{}j} p_{t{}k}} \, e^{+2\,\hbar \, {(-1)}^{p(t)}} \, +   \hfill  \cr
   \hfill   + \, A_{i,j}^{h,k} \, \sigma_{h,k}^{i,j}(t) \, + \, \varOmega_{i,j}^{h,k} \, \sigma_{i,k}^{h,j}(t) \, {(-1)}^{p_{i{}t} p_{h{}t}} \, e^{+\hbar \, {(-1)}^{p(t)}}   \qquad  }  $$
 Eventually, a sheer calculation proves that the right-hand side expression above is identically zero, so that  $ \, \kappa_{i,j}^{h,k}(t) = 0 \, $  for all  $ \, t \in \{1,\dots,n\} \, $,  \,so  $ \; {\textstyle \sum_{t=1}^n} C_{i,j}^{h,k}(t) \, = \, 0 \; $.
 \vskip5pt
   Now we go and compute the summands  $ \, D_{i,j}^{h,k}(b,d) \, $  in the second sum in right-hand side of  \eqref{eq: last-expr x Delta(eta)}.  By construction (see the above analysis) they are given by
  $$  \displaylines{
   \quad   D_{i,j}^{h,k}(b,d) \; := \;  {\textstyle \sum_{b<d}} \, \big(\, \alpha_{i,j}^{h,k}(b,d) \, x_{i{}b} \, x_{h{}d} \otimes x_{b{}j} \, x_{d{}k} \, + \, \omega_{i,j}^{h,k}(b,d) \, x_{i{}d} \, x_{h{}b} \otimes x_{d{}j} \, x_{b{}k} \,\big) \; +   \hfill  \cr
   \quad \qquad   + \; A_{i,j}^{h,k} \, {\textstyle \sum_{b<d}} \, \big(\, \alpha_{h,k}^{i,j}(b,d) \, x_{h{}b} \, x_{i{}d} \otimes x_{b{}k} \, x_{d{}j} \, + \, \omega_{h,k}^{i,j}(b,d) \, x_{h{}d} \, x_{i{}b} \otimes x_{d{}k} \, x_{b{}j} \,\big) \; +   \quad  \cr
   \quad \qquad \qquad   + \; \varOmega_{i,j}^{h,k} \, {\textstyle \sum_{b<d}} \, \big(\, \alpha_{i,k}^{h,j}(b,d) \, x_{i{}b} \, x_{h{}d} \otimes x_{b{}k} \, x_{d{}j} \, + \, \omega_{i,k}^{h,j}(b,d) \, x_{i{}d} \, x_{h{}b} \otimes x_{d{}k} \, x_{b{}j} \,\big)  }  $$
 and then, using identities  \eqref{eq: gen.'s-ideal-J-rewritten}  once more, we get for them the new expression
\begin{equation}  \label{eq: def-Dijhk(b,d)}
 \begin{gathered}
   \hskip-3pt   D_{i,j}^{h,k}(b,d)  \; = \;  {\textstyle \sum\limits_{b<d}} \Big(\, \beta_{i,j}^{h,k}(b,d) \, x_{h{}d} \, x_{i{}b} \otimes x_{d{}k} \, x_{b{}j}  \, + \,  \delta_{i,j}^{h,k}(b,d) \, x_{i{}d} \, x_{h{}b} \otimes x_{b{}k} \, x_{d{}j}  \,\; +  \\
   \hfill   + \;\,  \vartheta_{i,j}^{h,k}(b,d) \, x_{h{}d} \, x_{i{}b} \otimes x_{b{}k} \, x_{d{}j}  \, + \,  \lambda_{i,j}^{h,k}(b,d) \, x_{i{}d} \, x_{h{}b} \otimes x_{b{}k} \, x_{d{}j} \Big)
 \end{gathered}
\end{equation}
 where we use notation
\begin{equation*}  \label{eq: def-beta+delta+theta+lambda}
 \begin{gathered}
   \beta_{i,j}^{h,k}(b,d) \,\; := \;\,  \alpha_{i,j}^{h,k}(b,d) \, A_{i,b}^{h,d} \, A_{b,j}^{d,k}  \, + \,  A_{i,j}^{h,k} \, \omega_{h,k}^{i,j}(b,d)   \hfill  \\
   \delta_{i,j}^{h,k}(b,d) \,\; := \;\,  \alpha_{i,j}^{h,k}(b,d) \, \varOmega_{i,b}^{h,d} \, \varOmega_{b,j}^{d,k}  \, + \,  A_{i,j}^{h,k} \, \alpha_{h,k}^{i,j}(b,d) \, {(-1)}^{p_{i{}d} \, p_{h{}b}} \; +   \hfill  \\
   \qquad \qquad \qquad   + \,  \omega_{i,j}^{h,k}(b,d) \, {(-1)}^{p_{b{}k} \, p_{d{}j}}  \, - \,  \varOmega_{i,j}^{h,k} \, \varOmega_{i,b}^{h,d} \, \alpha_{i,k}^{h,j}(b,d)
  \\
   \vartheta_{i,j}^{h,k}(b,d) \,\; := \;\,  \alpha_{i,j}^{h,k}(b,d) \, A_{i,b}^{h,d} \, \varOmega_{b,j}^{d,k}  \, - \,  \varOmega_{i,j}^{h,k} \, A_{i,b}^{h,d} \, \alpha_{i,k}^{h,j}(b,d)   \hfill  \\
   \lambda_{i,j}^{h,k}(b,d) \,\; := \;\,  \alpha_{i,j}^{h,k}(b,d) \, \varOmega_{i,b}^{h,d} \, A_{b,j}^{d,k}  \, + \,  \varOmega_{i,j}^{h,k} \, \omega_{i,k}^{h,j}(b,d)   \hfill  \\
 \end{gathered}
\end{equation*}
   \indent   Finally, direct inspection (long and tricky, yet harmless) shows that
  $$  \beta_{i,j}^{h,k}(b,d) \, := \,  0  \;\; ,  \qquad  \delta_{i,j}^{h,k}(b,d) \, := \,  0  \;\; ,  \qquad  \vartheta_{i,j}^{h,k}(b,d) \, := \,  0  \;\; ,  \qquad  \lambda_{i,j}^{h,k}(b,d) \, := \,  0  $$
 hence from  \eqref{eq: def-Dijhk(b,d)}  we conclude that all  $ D_{i,j}^{h,k}(b,d) $'s  are zero.  This together with  \eqref{eq: last-expr x Delta(eta)}  yields  $ \; \Delta(\eta) \, \equiv_{\J_\otimes} 0 \, $,  \;which ends the proof of  \textit{(5--b)},  \,q.e.d.   \qed

\vskip21pt

\end{document}